\documentclass[12pt]{amsart}

\usepackage[margin=1in]{geometry} 
\usepackage{caption}
\usepackage{subcaption}
\usepackage{microtype}
\usepackage{tikz}
\usepackage{fourier} 
\usepackage{multido}
\usepackage{pgfplots}
\usepackage{mathrsfs}
\usepackage{fancyhdr}

\usepackage[utf8]{inputenc}
\usepackage{graphicx}
\graphicspath{{Figures/}}
\usepackage{amssymb}
\usepackage{bbm}
\usepackage{bm}
\usepackage{mathtools}
\usepackage{amsthm}
\usepackage{amsbsy}
\usepackage{amstext}
\usepackage{amscd}
\usepackage{nccmath}
\usepackage{enumitem}
\usepackage{amsmath}

\usepackage{enumerate}
\usepackage{xcolor}
\usepackage{threeparttable}

\usepackage[authoryear,round]{natbib}

\usepackage[colorlinks=true, linkcolor=blue, citecolor=blue, urlcolor=blue]{hyperref}

\usepackage{graphics}
\usepackage{appendix}

\usetikzlibrary{arrows}
\newcommand{\midarrow}{\tikz \draw[-triangle 90] (0,0) -- +(.1,0);}

\numberwithin{equation}{section}
\numberwithin{figure}{section}

\allowdisplaybreaks

\newtheorem{assumption}{Assumption}[]
\newtheorem{theorem}{Theorem}[section]
\newtheorem{lemma}[theorem]{Lemma}
\newtheorem{corollary}[theorem]{Corollary}
\newtheorem{proposition}[theorem]{Proposition}
\theoremstyle{remark}
\newtheorem{remark}{Remark}[section]

\newcommand{\bbA}{\mathbf{A}}
\newcommand{\bba}{\bm a}
\newcommand{\bbB}{\mathbf{B}}
\newcommand{\bbb}{\bm b}
\newcommand{\bbC}{\mathbf{C}}
\newcommand{\bbc}{\bm c}
\newcommand{\bbD}{\mathbf{D}}
\newcommand{\bbM}{\mathbf{M}}
\newcommand{\bbX}{\mathbf{X}}

\newcommand{\bbr}{\mathbf{r}}
\newcommand{\bbT}{\mathbf{T}}
\newcommand{\bbI}{\mathbf{I}}
\newcommand{\bbu}{\mathbf{u}}

\begin{document}\thispagestyle{empty}

	\begin{center}
	\Large\bf
	On the rate of convergence in the CLT for LSS of large-dimensional sample covariance matrices
	\end{center}
	\vspace{0.5cm}
	\renewcommand{\thefootnote}{\fnsymbol{footnote}}
	\hspace{5ex}	
	
	\begin{center}
 	\begin{minipage}[t]{0.4\textwidth}
		\begin{center}
			Jian Cui  \\
			\footnotesize {Northeast Normal University}\\
			{\it cuij836@nenu.edu.cn}
		\end{center}
	\end{minipage}
	\hfill
	\begin{minipage}[t]{0.4\textwidth}
		\begin{center}
			Jiang Hu  \\ 
			\footnotesize {Northeast Normal University}\\
			{\it huj156@nenu.edu.cn}
		\end{center} 
	\end{minipage}

	\begin{minipage}[t]{0.4\textwidth}
		\begin{center}
			Zhidong Bai  \\ 
			\footnotesize {Northeast Normal University}\\
			{\it baizd@nenu.edu.cn}
		\end{center} 
	\end{minipage}
	\hfill
	\begin{minipage}[t]{0.4\textwidth}
		\begin{center}
			Guorong Hu  \\ 
			\footnotesize {Northeast Normal University}\\
			{\it hugr@nenu.edu.cn}
		\end{center} 
	\end{minipage}
	\end{center}
%
%
	
\vspace{4ex}
\begin{center}
 	\begin{minipage}{0.8\textwidth} {
		\textbf{Abstract:} This paper investigates the rate of convergence for the central limit theorem of linear spectral statistic (LSS) associated with large-dimensional sample covariance matrices. We consider matrices of the form $\bbB_n=\frac{1}{n}\bbT_p^{1/2}\bbX_n\bbX_n^*\bbT_p^{1/2},$ where $\bbX_n= (x_{i j} ) $ is a $p \times n$ matrix whose entries are independent and identically distributed (i.i.d.) real or complex variables, and $\bbT_p$ is a $p\times p$ nonrandom Hermitian nonnegative definite matrix with its spectral norm uniformly bounded in $p$.
			Employing Stein's method, we establish that if the entries $x_{ij}$ satisfy $\mathbb{E}|x_{ij}|^{10}<\infty$ and the ratio of the dimension to sample size $p/n\to y>0$ as $n\to\infty$, then the convergence rate of the normalized LSS of $\bbB_n$ to the standard normal distribution, measured in the Kolmogorov-Smirnov distance, is $O(n^{-1/2+\kappa})$ for any fixed $\kappa>0$.}
	\end{minipage}
\end{center}

\vspace{4ex}

	\section{Introduction} 
	\subsection{Main result}
	
	In random matrix theory  (RMT),  the central limit theorems (CLTs) for the linear spectral statistics (LSSs) of large-dimensional sample covariance matrices represent a well-established area of research  (e.g., \cite{Jonsson82Lb,BaiS04C,BaiM07A,BaiJ09C,BaiW10F,ZhengB15S,BaoL15S,HuL19H,LiuH23C,BaoH24S}). This paper utilizes Stein's method to determine the rate of convergence in the CLT for LSS of such matrices, specifically in Kolmogorov-Smirnov distance. 
	
	More precisely, we consider the \textit{sample covariance matrix} 
	\begin{align*}
		\bbB_n=\frac{1}{n}\bbT_p^{1/2}\bbX_n\bbX_n^*\bbT_p^{1/2},
	\end{align*}
	where $\bbX_n=(x_{ij})$ is a $p\times n$ matrix with independent and identically distributed (i.i.d.) real or complex entries, and $\bbT_p$ is a $p\times p$ nonrandom Hermitian nonnegative definite matrix. For a given test function $f$,  the LSS of $\bbB_n$ is defined as 
	$$\sum_{i=1}^pf(\lambda_i^{\bbB_n})=p\int f(x) dF^{\bbB_n}(x),$$ 
	where $\lambda_1^{\bbB_n}\geq\dots\geq \lambda_p^{\bbB_n}$ are the ordered eigenvalues of $\bbB_n$, and $$F^{\bbB_n}(x)=\frac{1}{p}\sum_{i=1}^{p}I({\lambda_i^{\bbB_n}}\leq x)$$ 
	is the \textit{empirical spectral distribution}  (ESD) of $\bbB_n$, with $I(\cdot)$ denoting the indicator function.  Under the conditions that as $n\to\infty$, $y_n=p/n\to y>0$  and
	 $H_p=F^{\bbT_p}\stackrel{d}{\rightarrow}H$, \cite{Silverstein95S} demonstrated that $F^{\bbB_n}\stackrel{d}{\rightarrow}F^{y,H}$ almost surely (a.s.). The Stieltjes transform of $F^{y,H}$, denoted by $s:=s_{F^{y,H}}(z)$, satisfies the equation:
	\begin{align}\label{equation of s}
		s=\int \frac{1}{t(1-y-y z s)-z} d H(t).
	\end{align}
	Here, $\stackrel{d}{\rightarrow}$ signifies convergence in distribution, $H$  and $F^{y,H}$ are proper distribution functions, and the $\textit{Stieltjes transform}$ of  distribution function $G$ is defined by
	\begin{align*}
		s_{G}(z)=\int\frac{1}{\lambda-z}dG(\lambda), \quad\quad z\in\mathbb{C}^+.
	\end{align*}	
	We refer to $F^{y,H}$ as the $\textit{limiting spectral distribution}$ (LSD) of $\bbB_n$. Since the matrix $\underline{\bbB}_n:=(1/n)\bbX_n^*\bbT_p\bbX_n$ shares the same nonzero eigenvalues as $\bbB_n$, the LSD of $\underline{\bbB}_n$ is given by
	\begin{align*}
		\underline{F}^{y,H}=(1-y)\delta_0+yF^{y,H}. 
	\end{align*}
	Consequently, equation \eqref{equation of s} can be reformulated as 
	\begin{align}\label{equation of s^0}
		\underline{s}^0=-\left(z-y\int\frac{t}{1+t\underline{s}^0}dH(t)\right)^{-1},
	\end{align}
	where $\underline{s}^0:=s_{\underline{F}^{y,H}}$ is the Stieltjes transform of 	$\underline{F}^{y,H}$ and $\delta_0$ is the Dirac delta function. Let $F^{y_n,H_p}$ denote the distribution defined by \eqref{equation of s} with parameters $(y,H)$ replaced by $(y_n,H_p)$. Accordingly, from \eqref{equation of s^0}, we have equation
	\begin{align}\label{equation of s_n^0}
		\underline{s}_n^0=-\left(z-y_n\int\frac{t}{1+t\underline{s}_n^0}dH_p(t)\right)^{-1},
	\end{align}
	where $\underline{s}_n^0:=s_{\underline{F}^{y_n,H_p}}$ and $\underline{F}^{y_n,H_p}=(1-y_n)\delta_0+y_nF^{y_n,H_p}$. \cite{BaiS04C}  proved that if the fourth moment of $x_{11}$ exists, then under some mild technical conditions, 
	\begin{align}\label{LSS tend to normal distribution}
		\frac{\int f(x)dG_n(x)-\mu_n(f)}{\sqrt{\sigma_n(f)}}\stackrel{d}{\rightarrow} \mathcal{N}(0,1),
	\end{align}
	where
	\begin{align*}
		G_n(x)=F^{\bbB_n}(x)-F^{y_n,H_p}(x),
	\end{align*}
	\begin{align*}
		\mu_n(f)=-\frac{1}{2\pi i}\oint_{\mathcal{C}} f(z) \frac{y_n\int \underline{s}_n^0(z)^3t^2(1+t\underline{s}_n^0(z))^{-3}dH_p(t)}{(1-y_n\int \underline{s}_n^0(z)^2t^2(1+t\underline{s}_n^0(z))^{-2}dH_p(t))^2}dz,
	\end{align*} 
	\begin{align*}
		\sigma_n(f)=-\frac{1}{2\pi^2}\oint_{\mathcal{C}_1}\oint_{\mathcal{C}_2}f^{\prime}(z_1)f^{\prime}(z_2)a_n(z_1,z_2)\int_{0}^{1}\frac{1}{1-ta_n(z_1,z_2)}dtdz_1 dz_2.
	\end{align*}
	and
	\begin{align*}
		a_n(z_1,z_2)=y_n\underline{s}_n^0(z_1)\underline{s}_n^0(z_2)\int\frac{t^2dH_p(t)}{(1+t\underline{s}_n^0(z_1))(1+t\underline{s}_n^0(z_2))}.
	\end{align*}
	The contours $\mathcal{C},\mathcal{C}_1,\mathcal{C}_2$ are closed and oriented in the positive direction in the complex plane, each enclosing the support set defined in \eqref{supportset}. We may assume $\mathcal{C}_1$ and $\mathcal{C}_2 $ to be nonoverlapping. 

	Now, we are in a position to present our main result. Before that, we state the assumptions needed in the paper. \begin{assumption}\label{the moment assumption}
		For each $n$, $\bbX_n=(x_{ij})_{p\times n}$, where $x_{ij}$ are i.i.d. with common moments
		\begin{align*}
			&\mathbb{E}x_{ij}=0, \quad \mathbb{E}\lvert x_{ij}\rvert^2=1,  \quad \mathbb{E}\lvert x_{ij}\rvert^{10}<\infty, \\
			& \quad \beta_x=\mathbb{E}\lvert x_{ij}\rvert^4-\lvert\mathbb{E}x_{ij}^2\rvert^2-2, \quad \alpha_x=\lvert\mathbb{E}x_{ij}^2\rvert^2.
		\end{align*}
	\end{assumption}
	
	\begin{assumption}
		As $min\left\{p,n\right\}\to\infty$, the ratio of the dimension to sample size (RDS) $y_n=p/n\to y>0 $.
	\end{assumption}
	
	\begin{assumption}\label{assumption on population matrix}
		$\bbT_p$ is $p\times p$ nonrandom Hermitian nonnegative definite matrix with its spectral norm bounded in $p$, with $H_p=F^{\bbT_p} \stackrel{d}{\rightarrow} H$, a proper distribution function.
	\end{assumption}
	
	\begin{assumption}\label{assumption about test function}
		Let $f$ be a function analytic on an open region containing the interval 
		\begin{align}\label{supportset}
			\left[\liminf_p \lambda_{min}^{\bbT_p}I_{(0,1)}(y)(1-\sqrt{y})^2, \limsup_p\lambda_{max}^{\bbT_p}(1+\sqrt{y})^2 \right].
		\end{align}
	\end{assumption}
	
	\begin{assumption}\label{assumption about RG case}
		The matrix $\bbT_p$ and entries $x_{ij}$ are real and $\beta_x=0$.
	\end{assumption}
	
	\begin{assumption}\label{assumption about CG case}
		The entries $x_{ij}$ are complex, with $\alpha_x=0$ and $\beta_x=0$.
	\end{assumption}
	
	\begin{remark}
		The assumptions listed above, with the exception of $\mathbb{E}\lvert x_{ij}\rvert^{10}<\infty$, are identical to those in \cite{BaiS04C}. The condition $\mathbb{E}\lvert x_{ij}\rvert^{10}<\infty$ is imposed  to ensure that truncating the random variables
		 $x_{ij}$ at $n^{1/4}$ does not affect the derived convergence rate.  
	\end{remark}
	Let $\mathbb{K}(Y,Z)$ denote the Kolmogorov-Smirnov distance between two real random variables $Y$ and $Z$, defined as 
	\begin{align*}
		\mathbb{K}(Y,Z):=\sup_{x\in\mathbb{R}}\lvert \mathbb{P}(Y\leq x)-\mathbb{P}(Z\leq x)\rvert.
	\end{align*}
	Throughout this paper, $Z$ will represent a standard normal random variable and $\Phi$ denotes its distribution function. We now state the main theorem.
	
	\begin{theorem}\label{main theorem}
	$(\romannumeral1)$ Under Assumptions \ref{the moment assumption}--\ref{assumption about RG case}, for any fixed $\kappa>0$, there exists a positive constant $K$ independent of $n$, such that 
		\begin{align}\label{convergence rate in RG case}
			\mathbb{K}\left(\frac{\int f(x)dG_n(x)-\mu_n(f)}{\sqrt{\sigma_n(f)}},Z\right)\leq Kn^{-1/2+\kappa}. 
		\end{align}
 	$(\romannumeral2)$ Under Assumptions \ref{the moment assumption}--\ref{assumption about test function} and \ref{assumption about CG case}, for any fixed $\kappa>0$,   there exists a positive constant $K$ independent of $n$, such that  
		\begin{align}\label{convergence rate in CG case}
			\mathbb{K}\left(\frac{\int f(x)dG_n(x)}{\sqrt{(1/2)\sigma_n(f)}},Z\right)\leq Kn^{-1/2+\kappa}. 
		\end{align}
	\end{theorem} 	
	
	\begin{remark}
	The CLT \eqref{LSS tend to normal distribution}, established by \cite{BaiS04C}, relies fundamentally on the CLT for martingale difference sequences. And regarding the convergence rates for the CLT for martingale difference sequences, \cite{Bolthausen82E} and \cite{MachkouriO07E} demonstrated that the exact rate can reach $n^{-1/2}\log n$ under some mild conditions on the conditional variances. The rate $n^{-1/2}$ was derived by \cite{Bolthausen82E} under bounded conditional moments of order 4, and by \cite{Renz96N} for moment conditions of order between 2 and 3. 
	As for the convergence rate for non-Gaussian ensembles, \cite{BaoH23Q} estimated a near-optimal convergence rate of $n^{-1/2+\kappa}$ for the CLT for LSS of Wigner matrices.
	\end{remark}
	
	\subsection{Literature review on convergence rates of CLT and Stein's method}
	The establishment of convergence rates often involves fundamental inequalities between distribution functions or functions of bounded variation, frequently through suitable transformations. Seminal early contributions include those of \cite{Berry41A} and \cite{Esseen45F}, who derived convergence rates for normal approximation using Fourier transforms or characteristic functions. Subsequently, \cite{Bai93Ca} established convergence rates for the ESDs of large-dimensional random matrices via Stieltjes transforms. A distinct and potent alternative is Stein's method, pioneered by Charles Stein.

	Charles Stein introduced his innovative method in the seminal paper \cite{Stein72Ba}, aiming to provide an abstract normal approximation theorem, which yielded Berry-Esseen bounds for the sum of a stationary sequence of random variables. He further developed a complete theoretical framework in his monograph \cite{Stein86Aa}, which laid the foundation for what is now known as Stein's method. A key concept introduced in \cite{Stein86Aa} is that of an exchangeable pair of random variables. The method has since evolved through several significant innovations: the diffusion approach (\cite{Barbour90S}), size-biased coupling (\cite{GoldsteinR96M}), and zero-biased coupling (\cite{GoldsteinR97S}). These techniques rely on specific structural assumptions. Addressing more general scenarios,  Chatterjee proposed the generalized perturbative approach for discrete systems in \cite{Chatterjee08N}, and later extended it to a continuous framework in \cite{Chatterjee09F}. Stein's method is a versatile tool applicable not only to normal approximation, but also to approximations by other distributions (e.g. \cite{Chen75P,Luk94S,GauntP17C,ShaoZ19Ba}). For further discussion, see the review papers  \cite{Ross11F,Chatterjee14S,Chen21S}.

	In the field of RMT, proving CLTs can be particularly arduous due to complex computations. In recent years, Stein's method has emerged as a versatile and effective tool for tackling these challenges. Notably, \cite{Chatterjee09F} applied Stein's method to prove second-order Poincare inequalities, which were utilized to derive an upper bound on the total variation distance between a normal random variable and a class of random matrices expressed as smooth functions of independent random variables. Stein's method has also been successfully applied to LSS of classical compact groups and circular $\beta$-ensembles (e.g. \cite{DoblerS11Sb,DoblerS14Q,Fulman12S,Webb16L}), where the exchangeable pairs are derived from circular Dyson-Brownian motion. More recently, \cite{LambertL19Q} studied the CLT for LSS of $\beta$-ensembles in the one-cut regime,  providing a rate of convergence in the quadratic Kantorovich or Wasserstein-2 distance by using Stein's method; \cite{LandonL24S} employed Stein's method to establish the CLT for the LSS of general Wigner-type matrices. In summary, Stein's method serves as a crucial theoretical tool in the field of RMT. Motivated by its proven utility, the present work aims to introduce Stein's method into the study of convergence rate in the CLT for LSS of large-dimensional sample covariance matrices.

	Beyond the aforementioned works, several other significant studies have focused on convergence rates in RMT. The earliest contributions in this domain pertained to random matrices on compact group, as discussed in \cite{DiaconisS94E,Johansson97R}. Subsequently, \cite{BerezinB21R} investigated the convergence rate of the CLT for LSS of the Gaussian unitary ensembles (GUE), Laguerre unitary ensembles  (LUE), and Jacobi unitary ensembles (JUE) in Kolmogorov-Smirnov distance. More recently, \cite{BaoH23Q} estimated a near-optimal convergence rate for the CLT for LSS of $n\times n$ Wigner matrices, in Kolmogorov-Smirnov distance, which is either $n^{-1/2+\kappa}$ or $n^{-1+\kappa}$ depending on the first Chebyshev coefficient of test function $f$ and the third moment of the diagonal matrix entries; \cite{SchnelliX23C} provided a quantitative version of the Tracy-Widom law for the largest eigenvalue of large-dimensional sample covariance matrices, demonstrating convergence at the rate $O(n^{-1/3+\kappa})$; \cite{BaoG24U} used the upper bounds of the ultra high order cumulants to derive the quantitative versions of the CLT for the trace of any given self-adjoint polynomial in generally distributed Wigner matrices.

	\subsection{Organization}
	The remainder of this paper is structured as follows. Section \ref{Preliminaries} outlines essential preliminaries, including truncation, normalization, contour selection, which form the basis for the subsequent analysis. Section \ref{Random part} introduces two key propositions for handling the random components in the proof of the main theorem. The proofs of these propositions are shown in Appendix \ref{Proof of Proposition Simplification} and Section \ref{proof of the main term}, respectively. The analysis of the nonrandom components is presented in Section \ref{Nonrandom part}. Appendix \ref{Technical_lemmas} contains several technical lemmas utilized throughout the proofs, and Appendix \ref{Stein's Equation and its Solution} presents results pertinent to Stein's method. 
	
	\subsection{Notation and conventions}\label{Notation and conventions}
	Throughout this paper, $n$ is considered the fundamental large parameter, and $K$ denotes a positive constant that does  not depend on $n$. Multiple occurrences of $K$, even within a single formula, do not necessarily represent the same value. For a parameter $a$, $K_a$ signifies a constant that may depend on $a$. The notation $A=O(B)$ indicates $\lvert A\rvert\leq KB$, while $A_n=o(B_n)$ means $A_n/B_n \to 0$ as $n\to\infty$. 
	We use $\lVert \bbA\rVert$ for the spectral norm of a matrix $\bbA$, and $\lVert \bbu\rVert$ for the $L^2$-norm of a vector $\bbu$. To facilitate notation, we will let $\bbT=\bbT_p$. Under Assumption \ref{assumption on population matrix}, we may assume $\lVert \bbT\rVert\leq 1$ for all $p$. Assumptions \ref{assumption about RG case} and \ref{assumption about CG case} require $x_{11}$ to match the first, second and forth moments of either a real or complex Gaussian variable, the latter having real and imaginary parts that are i.i.d. $\mathcal{N}(0,1/2)$. We refer to these conditions as "$\textit{RG}$" and "$\textit{CG}$", respectively. 
	 	Let $\mathbb{E}_{0}(\cdot)$ be the expectation, $\mathbb{E}_{j}(\cdot)$ be the conditional expectation with respect to the $\sigma$-field generated by $\bbr_{1}, \dots, \bbr_{j}$ and $\mathbb{E}_{n}^{(j)}(\cdot)$ be the conditional expectation with respect to the $\sigma$-field generated by $\bbr_{1}, \dots, \bbr_{j-1}, \bbr_{j+1}, \dots, \bbr_n$.

	\subsection{Contribution and proof strategy}

	In the remainder of this subsection, we sketch the main arguments of our proofs. Our starting point is Lemma \ref{Chen} and Corollary \ref{Chen_2}, which address the invariance of the convergence rate under disturbance induced by linear transformations. These results facilitate the simplification of terms converging to the normal distribution.

	\textit{$\bullet$ Truncation and normalization.} The initial step involves subjecting the random variables to truncation at $\eta_nn^{1/4}$, followed by normalization, where $\eta_n\to 0$ as $n\to\infty$. The choice of truncation at $n^{1/4}$ is motivated by the objective to achieve the optimal rate of $O(n^{-k/2})$ for any $k$-th moment of the centered quadratic form (see Lemma \ref{QMTr_1}). As discussed in Section \ref{Truncation and normalization}, the truncation and normalization procedures do not affect the convergence rate we seek to establish.

	\textit{$\bullet$ Relation between the centering of LSS and the sum of martingale difference sequences.} Cauchy's integral formula allows the centered LSS to be expressed as 
	\begin{align*}
		\int f(x)d[G_n(x)-\mathbb{E}G_n(x)]=-\frac{p}{2\pi i}\oint_{\mathcal{C}}f(z)[s_{F^{\bbB_n}}(z)-\mathbb{E}s_{F^{\bbB_n}}(z)]dz,
	\end{align*}
	where the contour $\mathcal{C}$ is closed, oriented in the positive direction in the complex plane, and encloses the support set \eqref{supportset}. Details regarding the construction of $\mathcal{C}$ are provided in Section \ref{Stieltjes transform}. Consistent with the approach based on martingale decomposition and integration by parts adopted in Section 2 of \cite{BaiS04C}, the centered Stieltjes transform can be represented as a sum of martingale difference sequences. Further details are elaborated at the beginning of Section \ref{Random part}.

	\textit{$\bullet$ Simplification of the sum of martingale difference sequences.} We primarily focus on simplifying the remainder term in the Taylor expansion of the logarithmic function. This simplification, guided by Corollary \ref{Chen_2}, aims to render the analysis of the convergence rate to the normal distribution more tractable. As demonstrated in Proposition \ref{simplification of the random part}, this process culminates in the term $W_n$ (defined in \eqref{W_n in stein equation}), which does not affect the convergence rate of $O(n^{-1/2+\kappa})$. 
	\color{black}
	
	\textit{$\bullet$ Convergence rate to the normal distribution.} Following Proposition \ref{simplification of the random part}, the principal variable under investigation, $W_n$, is expressed as a sum of martingale difference sequences with unit variance. The subsequent and most critical step involves estimating $\mathbb{K}(W_n,Z)$ using Stein's method. To proceed, we introduce an auxiliary random variable $W_n^{(j)}$ (see \eqref{W_n^j}), constructed analogously to  $W_n$, so as to be independent of the $j$-th column of $\bbX_n$. As a consequence, the convergence rate $O(n^{-1/2+\kappa})$ is achieved, as will be demonstrated in Proposition \ref{Propostition of using stein method estimate rate}.

	During the application of Stein's method, obtaining a convergence rate of $n^{-1/4}$ is relatively straightforward, by analogy with the example of the sum of independent random variables (see Theorem 3.3 in \cite{ChenG11N}). However, improving this rate presents considerable challenges. Since the auxiliary variable $W_n^{(j)}$ is not independent of $Y_j$ (shown in \eqref{Y_j in stein equation}), we cannot directly employ the concentration inequality approach (see Section 3.4 in \cite{ChenG11N}) to enhance the rate. Instead, our strategy involves a careful treatment of the error term in the Stein equation. Specifically, we analyze the difference involving the smoothed $h$ (see \eqref{hw_0 theta_n}), aiming for an upper bound that depends on the Kolmogorov-Smirnov distance and $\lVert h^{\prime}\rVert^{-1}$, rather than $\lVert h^{\prime}\rVert$ directly, requiring only that the coefficient of Kolmogorov-Smirnov distance in the upper bound is much smaller than 1. The specific details are presented in Section \ref{Proof of Lemma sumEYj2gh'}.

	\textit{$\bullet$ Convergence rate of the nonrandom part.} The nonrandom component comprises two parts. The first pertains to the variance term, addressed in the proof of Proposition \ref{simplification of the random part}. In light of Corollary \ref{Chen_2}, we estimate $\lvert \sqrt{\sigma_n^0(f)}/\sqrt{\sigma_n(f)}-1\rvert$ to be  $O(n^{-1})$, as shown in Lemma \ref{convergence of variance}. The second part relates to the mean \eqref{the nonrandom part}; Lemma \ref{convergence of mean} demonstrates that the corresponding term is $O(n^{-1})$.
	\color{black}

	\section{Preliminaries}\label{Preliminaries}
	
	\subsection{A key lemma}
	We begin this subsection by presenting a key lemma, an extension of Lemma 9 in \cite{Chen81B}, which will serve as a fundamental tool for simplifying terms related to  convergence to the normal distribution. 
	
	\begin{lemma} \label{Chen}
		Let $Z_n=X_n+Y_n$, $n=1,2,\dots$, be a sequence of random variables, the distribution functions                                                                                                                                                                                                                                                                                                                                                                                                                                                                                                                                                                                                                                                                                                                                                                                                                                                                                                                                                                                                                                                                                                                                                                                                                                                                                                        of $Z_n$, $X_n$ be $F_n(x)$ and $G_n(x)$, respectively. Suppose there exists a sequence $\epsilon_n$ such that
		$$\mathbb{K}(X_n,Z)= \sup_{x\in\mathbb{R}}\lvert G_n(x)-\Phi(x)\rvert\leq K\epsilon_n,\quad \mathbb{P}\left(\lvert Y_n\rvert\geq \frac{K}{\epsilon_n}\right)\leq K\epsilon_n, \quad n=1,2,\dots,$$
		then 
		$$\mathbb{K}(Z_n,Z)=\sup_{x\in\mathbb{R}}\lvert F_n(x)-\Phi(x)\rvert\leq K\epsilon_n,\quad n=1,2,\dots.$$ 
	\end{lemma}
	\begin{proof}
		On one hand, we write
		\begin{align*}
			F_n(x) &= \mathbb{P}(Z_n\leq x)=\mathbb{P}(X_n+Y_n\leq x)\leq \mathbb{P}(X_n+Y_n\leq x,\lvert Y_n\rvert \textless K\epsilon_n)+\mathbb{P}(\lvert Y_n\rvert\geq K\epsilon_n) \\
			&\leq \mathbb{P}(X_n\leq x+K\epsilon_n) + \mathbb{P}(\lvert Y_n\rvert\geq K\epsilon_n).
		\end{align*}
		On the other hand,
		\begin{align*}
			F_n(x) &= \mathbb{P}(X_n+Y_n\leq x)\geq \mathbb{P}(X_n+Y_n\leq x,\lvert Y_n\rvert\textless K\epsilon_n)-\mathbb{P}(\lvert Y_n\rvert\geq K\epsilon_n) \\
			&\geq \mathbb{P}(X_n\leq x-K\epsilon_n)-\mathbb{P}(\lvert Y_n\rvert\geq K\epsilon_n).
		\end{align*}
		Noting that $\Phi(x)$ is a Lipschitz continuous function, we can get
		\begin{align*}
			F_n(x)-\Phi(x) &\leq \mathbb{P}(X_n\leq x+K\epsilon_n)-\Phi(x+K\epsilon_n)+\Phi(x+K\epsilon_n)-\Phi(x) + \mathbb{P}(\lvert Y_n\rvert\geq K\epsilon_n) \\
			&= G_n(x+K\epsilon_n)-\Phi(x+K\epsilon_n)+\Phi(x+K\epsilon_n)-\Phi(x) + \mathbb{P}(\lvert Z_n-X_n\rvert\geq K\epsilon_n) \\
			&\leq K\epsilon_n,
		\end{align*}
	and
		\begin{align*}
			F_n(x)-\Phi(x) &\geq \mathbb{P}(X_n\leq x-K\epsilon_n)-\Phi(x-K\epsilon_n)+\Phi(x-K\epsilon_n)-\Phi(x)-\mathbb{P}(\lvert Y_n\rvert\geq K\epsilon_n) \\
			&= G_n(x-2K\epsilon_n)-\Phi(x-2K\epsilon_n)+\Phi(x-2K\epsilon_n)-\Phi(x)  -\mathbb{P}(\lvert Z_n-X_n\rvert\geq K\epsilon_n)  \\
			&\geq -K\epsilon_n.
		\end{align*}
		Since the above two bounds are independent of $x$, we complete the proof.
	\end{proof}
	
	\begin{remark}\label{about K-S distance equivalent}
	By Markov's inequality, for any fixed $s>0$, if $\mathbb{E}\lvert Z_n-X_n \rvert^s\leq K\epsilon_n^{1+s}$, then 
		\begin{align}\label{ZsimX}
			\mathbb{P}\left(\lvert Z_n-X_n\rvert\geq \frac{K}{\epsilon_n}\right)\leq K\epsilon_n.
		\end{align}
	In the sequel, we denote $Z_n\stackrel{\epsilon_n}{\sim}X_n$  if \eqref{ZsimX} holds.
	\end{remark}
	
	The following corollary will be instrumental in the proof of our main result.
	
	\begin{corollary}\label{Chen_2}
		Let $X_n$, $n=1,2,\dots$,  be a sequence of random variables,   $a_n>0$ and $b_n$ be nonrandm sequences. Let $Z_n=a_nX_n+b_n$ and denote the distribution functions of $Z_n$ and $X_n$ are $F_n(x)$ and $G_n(x)$, respectively. 
		If there exists $\epsilon_n$, such that for all $n\geq1$
		\begin{align*}
			\lvert a_n-1\rvert\leq K\epsilon_n , \quad \lvert b_n\rvert\leq K\epsilon_n , \quad \mathbb{K}(X_n,Z)\leq K\epsilon_n,
		\end{align*}
		then 
		\begin{align*}
			\mathbb{K}(Z_n,Z)\leq K\epsilon_n, \quad n=1,2,\dots.
		\end{align*}
	\end{corollary}
	
	\subsection{Truncation and normalization}\label{Truncation and normalization}
	This subsection is devoted to the initialization procedure, involving truncation and normalization of the random variables. The subsequent arguments will demonstrate that these truncation and normalization do not affect the established convergence rates.

	Under Assumption \ref{the moment assumption}, for any $\eta>0$, we have  
	\begin{align*}
		\eta^{-10}\mathbb{E}\lvert x_{11}\rvert^{10} I\left\lbrace \lvert x_{11}\rvert\geq \eta n^{1/4}\right\rbrace \to 0.
	\end{align*}
	Following Lemma 15 in \cite{LiB16C}, we can select a slowly decreasing sequence of constants $\eta_n \to 0$ such that 
	\begin{align}\label{truncation_condition}
		\eta_n^{-10}\mathbb{E}\lvert x_{11}\rvert^{10} I\left\lbrace \lvert x_{11}\rvert\geq \eta_n n^{1/4}\right\rbrace \to 0.
	\end{align}
	Let $\widehat{\bbB}_n=(1/n)\bbT^{1/2}\widehat{\bbX}_n\widehat{\bbX}_n^{*}\bbT^{1/2}$, where $\widehat{\bbX}_n$ $p\times n$ with $(i,j)$-th entry $\widehat{x}_{ij}=x_{ij}I\lbrace \lvert x_{ij}\rvert<\eta_n n^{1/4}\rbrace $. We use $\widehat{G}_n(x)$ to denote the analogues of $G_n(x)$ with $\bbB_n$ replaced by $\widehat{\bbB}_n$. Then,
	\begin{align}\label{truncation_beforeandafter}
		& \mathbb{K}\left(\int f(x)dG_n(x),\int f(x)d\widehat{G}_n(x)\right) \\
		\nonumber =& \sup_{x\in\mathbb{R}} \left| \mathbb{P}\left(\int f(x)dG_n(x)\leq x\right)-\mathbb{P}\left(\int f(x)d\widehat{G}_n(x)\leq x\right)\right|\\
		\nonumber =& \sup_{x\in\mathbb{R}} \left| \mathbb{P}\left( \left\lbrace  \int f(x)dG_n(x) \leq x \right\rbrace  \cap \left\lbrace  \int f(x)dG_n(x) \neq \int f(x)d\widehat{G}_n(x) \right\rbrace  \right)  \right. \\
		\nonumber  &\quad\quad \left. - \mathbb{P}\left( \left\lbrace  \int f(x)d\widehat{G}_n(x) \leq x \right\rbrace  \cap \left\lbrace  \int f(x)dG_n(x) \neq \int f(x)d\widehat{G}_n(x) \right\rbrace  \right) \right|  \\
		\nonumber \leq & 2\mathbb{P}\left(\int f(x)dG_n(x)\neq\int f(x)d\widehat{G}_n(x)\right) = 2\mathbb{P}\left(\bbB_n\neq \widehat{\bbB}_n\right) \\
		\nonumber \leq & 2np\mathbb{P}\left(\lvert x_{11}\rvert\geq \eta_n n^{1/4}\right) \leq Kn^{-1/2} \eta_n^{-10} \mathbb{E} \left[ \lvert x_{11}\rvert^{10} I\lbrace \lvert x_{11}\rvert\geq \eta_n n^{1/4}\rbrace \right] = o(n^{-1/2}).
	\end{align} 
%
	Define $\widetilde{\bbB}_n = (1/n)\bbT^{1/2}\widetilde{\bbX}_n\widetilde{\bbX}_n^{*}\bbT^{1/2}$ with $\widetilde{\bbX}_n$ $p\times n$ having $(i,j)$-th entry $\widetilde{x}_{ij}=(\widehat{x}_{ij}-\mathbb{E}\widehat{x}_{ij})/\sigma_{ij}$, where $\sigma_{ij}^2=\mathbb{E}\lvert \widehat{x}_{ij}-\mathbb{E}\widehat{x}_{ij} \rvert^2$. We use $\widetilde{G}_n(x)$ to denote the analogues of $G_n(x)$ with $\bbB_n$ replaced by $\widetilde{\bbB}_n$.
	By Lemma \ref{Chen} and Remark \ref{about K-S distance equivalent}, it suffices to consider  $\mathbb{E}\lvert \int f(x)d\widehat{G}_n(x)-\int f(x)d\widetilde{G}_n(x)\rvert$. Following the methodology and bounds from the proof of Lemma 2.7 in \cite{Bai99M}, we have  
	\begin{align}\label{normalization_beforeandafter}
		&  \mathbb{E} \left| \int f(x)d\widehat{G}_n(x)-\int f(x)d\widetilde{G}_n(x)\right|  \\
		\nonumber =& \mathbb{E}\left| \sum_{k=1}^{p} [ f(\lambda_k^{\hat{\bbB}_n})-f(\lambda_k^{\tilde{\bbB}_n})] \right| \leq K \mathbb{E} \left| \sum_{k=1}^{p}(\lambda_k^{\hat{\bbB}_n}-\lambda_k^{\tilde{\bbB}_n}) \right|  \\
		\nonumber \leq& K_{f^{\prime}} \left[\sum_{k=1}^{p}\mathbb{E}\left(\sqrt{\lambda_k^{\hat{\bbB}_n}}+\sqrt{\lambda_k^{\tilde{\bbB}_n}}\right)^2\right]^{1/2}\left[\sum_{k=1}^{p} \mathbb{E}\bigg\lvert \sqrt{\lambda_k^{\hat{\bbB}_n}}-\sqrt{\lambda_k^{\tilde{\bbB}_n}}\bigg\rvert^2\right]^{1/2}  \\
		\nonumber  \leq& K\left[n^{-1}\mathbb{E} \operatorname{tr}\bbT^{1/2}(\widehat{\bbX}_n-\widetilde{\bbX}_n)(\widehat{\bbX}_n-\widetilde{\bbX}_n)^{*}\bbT^{1/2}\right]^{1/2} \left[2n^{-1}\mathbb{E} \operatorname{tr} \bbT^{1/2}(\widehat{\bbX}_n\widehat{\bbX}_n^{*})\bbT^{1/2}\right]^{1/2}.
	\end{align}
	According to \eqref{truncation_condition}, we have
	\begin{align*}
		\left| \mathbb{E} \widehat{x}_{ij} \right| &= \left|\mathbb{E} x_{ij} I\left\lbrace \lvert x_{ij}\rvert \geq \eta_n n^{1/4} \right\rbrace \right| \leq \frac{1}{(\eta_n n^{1/4})^{9}} \mathbb{E}\left[\lvert x_{ij}\rvert^{10} I\left\lbrace \lvert x_{ij}\rvert\geq \eta_n n^{1/4} \right\rbrace\right] = o(n^{-9/4}\eta_n) 
	\end{align*}
	and
	\begin{align*}
		(1-\sigma_{ij}^{-1})^2 &= \frac{(1-\sigma_{ij}^2)^2}{\sigma_{ij}^2(1+\sigma_{ij})^2}=\frac{(\mathbb{E}\lvert x_{ij}\rvert^2-\mathbb{E}\lvert\widehat{x}_{ij}-\mathbb{E}\widehat{x}_{ij}\rvert^2)^2}{\sigma_{ij}^2(1+\sigma_{ij})^2}  \\
		&\leq 2\frac{\left[ \mathbb{E}\lvert x_{ij}\rvert^2 I\left\lbrace \lvert x_{ij}\rvert \geq \eta_n n^{1/4}\right\rbrace \right]^2 + (\mathbb{E}\widehat{x}_{ij})^4 }{\sigma_{ij}^2(1+\sigma_{ij})^2}  \\
		&\leq K \left[ \frac{1}{(\eta_n n^{1/4})^{8}} \mathbb{E}\left[\lvert x_{ij}\rvert^{10} I\left\lbrace \lvert x_{ij}\rvert\geq\eta_n n^{1/4} \right\rbrace  \right] \right]^2 + o(n^{-9}\eta_n^4)   =o(n^{-4}\eta_n^{4}) .
	\end{align*}
Using the two conclusions above, we obtain that
	\begin{align*}
		& n^{-1}\mathbb{E} \operatorname{tr}\bbT^{1/2}(\widehat{\bbX}_n-\widetilde{\bbX}_n)(\widehat{\bbX}_n-\widetilde{\bbX}_n)^{*}\bbT^{1/2}
		\leq n^{-1} \lVert \bbT\rVert\sum_{i,j} \mathbb{E}\bigg\lvert (1-\sigma_{ij}^{-1})\widehat{x}_{ij}+\frac{\mathbb{E} \widehat{x}_{ij}}{\sigma_{ij}}\bigg\rvert^2 \\
		\leq& 2n^{-1}\sum_{i,j}\left[ (1-\sigma_{ij}^{-1})^2\mathbb{E}\lvert \widehat{x}_{ij}\rvert^2+\frac{1}{\sigma_{ij}^2} \mathbb{E}\lvert \widehat{x}_{ij}\rvert^2 \right]  
		\leq 2n^{-1} np \left[ (1-\sigma_{11}^{-1})^2+\sigma_{11}^{-2}\lvert\mathbb{E} \widehat{x}_{11}\rvert^2 \right]  \\
		= &o(n^{-3}\eta_n^{4}).
	\end{align*}
and
	\begin{align*}
		&  n^{-1}\mathbb{E} \operatorname{tr} \bbT^{1/2}(\widehat{\bbX}_n\widehat{\bbX}_n^{*})\bbT^{1/2} 
		\leq n^{-1}\lVert \bbT \rVert \sum_{i,j} \mathbb{E}\left[\lvert \widehat{x}_{ij}\rvert^2 +\lvert \widetilde{x}_{ij} \rvert^2 \right]  \\
		\leq& K n^{-1} np \left[ \mathbb{E}\lvert x_{11}\rvert^2+\frac{\mathbb{E}\lvert \widehat{x}_{11}-\mathbb{E}\widehat{x}_{11}\rvert^2}{\sigma_{11}} \right]  
		\leq K p\left[ \mathbb{E}\lvert x_{11}\rvert^2 + \frac{2\mathbb{E}\lvert \widehat{x}_{11}\rvert^2}{\sigma_{11}}\right] 
	= O(p). 
	\end{align*}
	Therefore, 
	\begin{align*}
		\mathbb{E}\left| \int f(x)d\widehat{G}_n(x)-\int f(x)d\widetilde{G}_n(x)\right|=o(n^{-1}\eta_n^{2}),
	\end{align*}
	which implies
	\begin{align}\label{normalise_before_and_after}
		 \int f(x)d\widehat{G}_n(x) \stackrel{n^{-\frac{1}{2}}}{\sim}  \int f(x)d\widetilde{G}_n(x).  
	\end{align}
	
	Henceforth, we assume the underlying variables are truncated at $\eta_n n^{1/4}$, centralized and normalized. For simplicity, we will suppress all sub- or superscripts on these variables and assume $\lvert x_{ij}\rvert<\eta_n n^{1/4}$, $\mathbb{E}x_{ij}=0$, $\mathbb{E}\lvert x_{ij}\rvert^2=1$, $\mathbb{E}\lvert x_{ij}\rvert^{10}<\infty$. For the RG case (see Assumption \ref{assumption about RG case}), $\mathbb{E}\lvert x_{ij}\rvert^4=3+o(n^{-3/2}\eta_n^{4})$; for the CG case (see Assumption \ref{assumption about CG case}), $\mathbb{E} x_{ij}^2=o(n^{-2}\eta_n^2)$ and $\mathbb{E}\lvert x_{ij}\rvert^4=2+o(n^{-3/2}\eta_n^4)$.
	From  \cite{BaiY93L} and \cite{YinB88L}, after truncation and normalization,  
	\begin{align}\label{outofsupport}
		\mathbb{P}(\lVert \bbB_n\rVert\geq \mu_1)=o(n^{-l})\quad \quad and \quad\quad \mathbb{P}(\lambda_{min}^{\bbB_n}\leq \mu_2)=o(n^{-l})
	\end{align}
	hold for any $\mu_1>\limsup_p\lVert \bbT_p\rVert(1+\sqrt{y})^2$,  $0<\mu_2<\liminf_p\lambda_{min}^{\bbT_p}I_{(0,1)}(y)(1-\sqrt{y})^2$ and $l>0$.

	\subsection{Relation between the LSS and the Stieltjes transform}\label{Stieltjes transform}
	Define
	\begin{align}\label{M_n}
		M_n(z)=p[s_{F^{\bbB_n}}(z)-s_{F^{y_n,H_p}}(z)]=n[s_{F^{\underline{\bbB}_n}}(z)-s_{\underline{F}^{y_n,H_p}}(z)].
	\end{align}
	Cauchy's integral formula enables us to rewrite the LSS into 
	\begin{align*}
		\int f(x)dG_n(x)=-\frac{1}{2\pi i}\oint_{\mathcal{C}}f(z)M_n(z)dz,
	\end{align*}
	where the contour $\mathcal{C}$ is closed and taken in the positive direction in the complex plane containing the support of $G_n$. In what follows, we specify the selection of a suitable contour $\mathcal{C}$.

	Let $v_0>0$ and $\epsilon>0$ be arbitrary. Let $x_r=\lambda_{max}^{\bbT}(1+\sqrt{y})^2+\epsilon$. Let $x_l$ be any negative number if the left point of \eqref{supportset} is zero. Otherwise, choose $x_l=\lambda_{min}^{\bbT}I_{(0,1)}(y)(1-\sqrt{y})^2-\epsilon$. Define
	$$\mathcal{C}_u=\left\lbrace x+iv_0:x\in[x_l,x_r]\right\rbrace, 
~~\mathcal{C}_l=\left\lbrace x_l+iv:v\in[0,v_0]\right\rbrace,$$
	$$\mathcal{C}_r= \left\lbrace x_r+iv:v\in[0,v_0]\right\rbrace,~~\mathcal{C}^{+}\equiv \mathcal{C}_l \cup \mathcal{C}_u \cup \mathcal{C}_r. $$
	Then  $\mathcal{C}=\mathcal{C}^{+}\cup \overline{\mathcal{C}^{+}}$.
	Let
	\begin{align}\label{Xi_n}
		\Xi_n=\left\lbrace \lambda_{min}^{B_n}\leq x_l-\epsilon/2\quad or\quad\lambda_{max}^{B_n}\geq x_r+\epsilon/2\right\rbrace.
	\end{align}
	Due to (\ref{outofsupport}), for any $l>0$, 	we have 
	\begin{align}\label{order of Xi_n}
		\mathbb{P}(\Xi_n)=o(n^{-l}).
	\end{align}
	Since the support of $F^{y_n,H_p}$ is contained within $[x_l-\epsilon/2,x_r+\epsilon/2]$, 
 hence it follows from the CLT for $	\int f(x) dG_n(x)$ established by \cite{BaiS04C}, 
	\begin{align*}
		\int f(x) dG_n(x)  \stackrel{n^{-\frac{l}{4}}}{\sim} I(\Xi_n^{c})\int f(x)dG_n(x).
	\end{align*}
	For simplicity, we will omit the notation  $I(\Xi_n^{c})$ in the sequel. 

	For $z\in\mathcal{C}$, we decompose $M_n(z)=M_n^1(z)+M_n^2(z)$, where
	\begin{align}\label{random part}
		M_n^1(z)=p[s_{F^{\bbB_n}}(z)-\mathbb{E}s_{F^{\bbB_n}}(z)]
	\end{align}
	and
	\begin{align}\label{nonrandom part}
		M_n^2(z)=p[\mathbb{E}s_{F^{\bbB_n}}(z)-s_{F^{y_n,H_p}}(z)].
	\end{align} 
	Accordingly, we divide 
	\begin{align*}
		\frac{\int f(x)dG_n(x)-\mu_n(f)}{\sqrt{\sigma_n(f)}}
	\end{align*}
	into the random part
	\begin{align}\label{the random part}
		& \frac{\int f(x)d[G_n(x)-\mathbb{E}G_n(x)]}{\sqrt{\sigma_n(f)}}
		=-\frac{1}{2\pi i\sqrt{\sigma_n(f)}}\oint_{\mathcal{C}} f(z)M_n^1(z)dz,
	\end{align}
	and the nonrandom part
	\begin{align}\label{the nonrandom part}
		& \frac{\int f(x)d\mathbb{E}G_n(x)-\mu_n(f)}{\sqrt{\sigma_n(f)}}=-\frac{1}{\sqrt{\sigma_n(f)}}\left[\frac{1}{2\pi i}\oint_{\mathcal{C}}f(z)M_n^2(z)dz-\mu_n(f)\right].
	\end{align} 
	The subsequent analysis will investigate the two components,  \eqref{the random part} and \eqref{the nonrandom part}, respectively.

	\section{Random part}\label{Random part}
Before giving the proof of the convergence rate of the random part, we introduce some notations adapted from \cite{BaiS04C}, which will be used extensively. Let $\bbr_{j}=(1 / \sqrt{n}) \bbT^{1/2} \bbX_{\cdot j}$, where $\bbX_{\cdot j}$ denotes the $j$-th column of $\bbX_n$. Denote $\bbD(z)=\bbB_{n}-z\bbI$, $ \bbD_{j}(z)=\bbD(z)-\bbr_{j}\bbr_{j}^{*}$, $\bbD_{ij}(z)=\bbD(z)-\bbr_i\bbr_i^{*}-\bbr_j\bbr_j^{*}$,
	\begin{align*}
		 \varepsilon_{j}(z)=&\bbr_{j}^{*} \bbD_{j}^{-1}(z) \bbr_{j}-\frac{1}{n} \operatorname{tr} \bbT \bbD_{j}^{-1}(z), ~~
		 \gamma_j(z)=\bbr_j^*\bbD_j^{-1}(z)\bbr_j-\frac{1}{n}\mathbb{E}\operatorname{tr} \bbT\bbD_j^{-1}(z)\\
		 \beta_{j}(z)=&\frac{1}{1+\bbr_{j}^{*} \bbD_{j}^{-1}(z) \bbr_{j}}, ~~ \widetilde{\beta}_{j}(z)=\frac{1}{1+n^{-1} \operatorname{tr} \bbT\bbD_{j}^{-1}(z)}, ~~ b_{j}(z)=\frac{1}{1+n^{-1} \mathbb{E} \operatorname{tr} \bbT \bbD_{j}^{-1}(z)} \\
		 \beta_{ij}(z)=&\frac{1}{1+\bbr_{i}^{*} \bbD_{ij}^{-1}(z) \bbr_{i}}, ~~ \widetilde{\beta}_{ij}(z)=\frac{1}{1+n^{-1}\operatorname{tr}\bbT\bbD_{ij}^{-1}(z)}, \\
		 \gamma_{ij}(z)=&\bbr_i^*\bbD_{ij}^{-1}(z)\bbr_i-\frac{1}{n}\mathbb{E}\operatorname{tr}\bbT\bbD_{ij}^{-1}(z), ~~ b_{ij}(z)=\frac{1}{1+n^{-1} \mathbb{E} \operatorname{tr} \bbT \bbD_{ij}^{-1}(z)} .
	\end{align*}

	Next, we will introduce some trivial identities, whose detailed proofs can be found in \cite{BaiS10S}. For any Hermitian matrix $\bbC$, 
	\begin{align}\label{r_i^*(C+r_ir_i^*)^-1}
		\bbr_i^{*}(\bbC+\bbr_i\bbr_i^{*})^{-1}=\frac{1}{1+\bbr_i^{*}\bbC^{-1}\bbr_i}\bbr_i^{*}\bbC^{-1}.
	\end{align}
It follows that 
	\begin{align} \label{D^{-1}-D_j^{-1}}
		&  \bbD^{-1}(z)-\bbD_j^{-1}(z)  =  -\frac{\bbD_j^{-1}(z)\bbr_j\bbr_j^{*}\bbD_j^{-1}(z)}{1+\bbr_j^{*}\bbD_j^{-1}(z)\bbr_j} =-\beta_{j}(z)\bbD_j^{-1}(z)\bbr_j\bbr_j^{*}\bbD_j^{-1}(z),
	\end{align}
and
	\begin{align} \label{D_j^{-1}-D_kj^{-1}}
		& \bbD_j^{-1}(z)-\bbD_{kj}^{-1}(z)=-\frac{\bbD_{kj}^{-1}(z)\bbr_k\bbr_k^{*}\bbD_{kj}^{-1}(z)}{1+\bbr_k^{*}\bbD_{kj}^{-1}(z)\bbr_k}=-\beta_{kj}(z)\bbD_{kj}^{-1}(z)\bbr_k\bbr_k^{*}\bbD_{kj}^{-1}(z).
	\end{align}
	In addition, the following relations hold:
	\begin{align}\label{beta_1}
		\beta_j=b_n-\beta_j b_n \gamma_j=b_n-b_n^2\gamma_j+\beta_j b_n^2\gamma_j^2 ,
	\end{align}
	\begin{align}\label{beta_ij decomposition}
		\beta_{ij}=b_{ij}-\beta_{ij}b_{ij}\gamma_{ij}=b_{ij}-b_{ij}^2\gamma_{ij}+\beta_{ij} b_{ij}^2\gamma_{ij}^2 ,
	\end{align}
	and
	\begin{align}\label{beta_j}
		\beta_{j}(z)=\widetilde{\beta}_j(z)-\beta_{j}(z)\widetilde{\beta}_j(z)\varepsilon_{j}(z).
	\end{align}

	Now we are in the position to consider the random part \eqref{the random part}. Employing martingale decomposition, integration by parts and \eqref{D^{-1}-D_j^{-1}}-\eqref{beta_j}, we have that 
	\begin{align}\label{decomposition of random part}
		&  p\int f(x)d[F^{\bbB_n}(x)-\mathbb{E}F^{\bbB_n}(x)] 
		=-\frac{p}{2\pi i}\oint_{\mathcal{C}}f(z)[s_{F^{\bbB_n}}(z)-\mathbb{E}s_{F^{\bbB_n}}(z)]dz  \\
		\nonumber 
		=& -\frac{1}{2\pi i}\sum_{j=1}^{n}\oint_{\mathcal{C}} f^{\prime}(z) (\mathbb{E}_j-\mathbb{E}_{j-1})\left[\varepsilon_j(z)b_j(z)+Q_j(z)\right]dz,
	\end{align}
	where $Q_j(z)=R_j(z)+\varepsilon_j(z)(\widetilde{\beta}_j(z)-b_j(z))$, $R_j(z)=\ln(1+\varepsilon_j(z)\widetilde{\beta}_j(z))-\varepsilon_j(z)\widetilde{\beta}_j(z)$.

	Denote 
	\begin{align}\label{Y_j in stein equation}
		Y_j=\oint_{\mathcal{C}}f^{\prime}(z) (\mathbb{E}_j-\mathbb{E}_{j-1})\varepsilon_j(z)b_j(z)dz,
	\end{align}
	\begin{align}\label{W_n in stein equation}
		W_n=\frac{1}{\sqrt{\sigma_n^0(f)}}\sum_{j=1}^{n}Y_j~~\mbox{where}~~\sigma_n^0(f)=\sum_{j=1}^{n}\mathbb{E}Y_j^2.
	\end{align}
	Based on Lemma \ref{Chen} and Corollary \ref{Chen_2}, we will introduce the following proposition to simplify the proof of the random part.
	\begin{proposition}\label{simplification of the random part}
		Under Assumptions \ref{the moment assumption}--\ref{assumption about test function}, and either \ref{assumption about RG case} or \ref{assumption about CG case}, for any fixed $\kappa>0$,
		\begin{align*}
			&\frac{p}{\sqrt{\sigma_n(f)}}\oint_{\mathcal{C}}f(z)[s_{F^{\bbB_n}}(z)-\mathbb{E}s_{F^{\bbB_n}}(z)]dz 
			\stackrel{n^{-\frac{1}{2}+\kappa}}{\sim}  W_n.
		\end{align*}
	\end{proposition}
	
	\begin{remark}
		Since the proof of this proposition is routine in random matrix theory, we postpone it to Appendix \ref{Proof of Proposition Simplification}.
	\end{remark}
	
	Notice that $W_n$ is a sum of martingale difference sequences with unit variance. The next proposition plays an important role in our paper which shows the convergence rate of $W_n$ to $Z$ using Stein's method.
		\begin{proposition}\label{Propostition of using stein method estimate rate}
		Under Assumptions \ref{the moment assumption}--\ref{assumption about test function}, and either  \ref{assumption about RG case} or \ref{assumption about CG case}, for any fixed $\kappa>0$,
		\begin{align*}
			\mathbb{K}(W_n,Z)\leq Kn^{-1/2+\kappa}.
		\end{align*}
	\end{proposition}

	The remainder of this section is dedicated to the proof of Proposition \ref{Propostition of using stein method estimate rate}.

	\subsection{Proof of Proposition $\ref{Propostition of using stein method estimate rate}$}\label{proof of the main term}
	In this subsection, we will prove Proposition \ref{Propostition of using stein method estimate rate} utilizing Stein's method.
	Revisit notations \eqref{Y_j in stein equation} and \eqref{W_n in stein equation}. Let
	\begin{align*}
		Y_{kj}=\oint_{\mathcal{C}}f^{\prime}(z)b_k(z)(\mathbb{E}_k-\mathbb{E}_{k-1})(\bbr_k^*\bbD_{kj}^{-1}(z)\bbr_k-n^{-1}\operatorname{tr}\bbT\bbD_{kj}^{-1}(z))dz.                    
	\end{align*} 
	Let 
	\begin{align}\label{W_n^j}
		W_n^{(j)} &=\frac{1}{\sqrt{\sigma_n^0}}\sum_{k>j}\oint_{\mathcal{C}}f^{\prime}(z) b_k(z)(\mathbb{E}_k-\mathbb{E}_{k-1})(\bbr_k^{*}\bbD_{kj}^{-1}(z)\bbr_k-n^{-1}\operatorname{tr}\bbT\bbD_{kj}^{-1}(z))dz \\
		\nonumber&\quad +\frac{1}{\sqrt{\sigma_n^0}}\sum_{k<j}\oint_{\mathcal{C}}f^{\prime}(z) b_k(z)(\mathbb{E}_k-\mathbb{E}_{k-1})(\bbr_k^{*}\bbD_{k}^{-1}(z)\bbr_k-n^{-1}\operatorname{tr}\bbT\bbD_{k}^{-1}(z))dz \\
		\nonumber&= \frac{1}{\sqrt{\sigma_n^0}}\sum_{k<j}Y_k+\frac{1}{\sqrt{\sigma_n^0}}\sum_{k>j}Y_{kj}.
	\end{align}
	Here, we abbreviate $\sigma_n^0(f)=\sum_{j=1}^{n}\mathbb{E}Y_j^2$ as $\sigma_n^0$. It's evident that $W_n^{(j)}$ is independent of $\bbr_j$.
	 
	By taking expectations on both sides of the Stein equation (\ref{stein_equation}), we arrive at the identity:
	\begin{align}\label{stein_identity}
		\mathbb{E}[g_h(W_n)W_n]-\mathbb{E}g_h^{\prime}(W_n)=Nh-\mathbb{E}h(W_n).
	\end{align}
	Our objective is to bound the difference between $\mathbb{E}\left[g_h(W_n)W_n\right]$ and $\mathbb{E}g_h^{\prime}(W_n)$ thereby bounding $Nh-\mathbb{E}h(W_n)$. Properties of the solution to the Stein equation are deferred to Appendix \ref{Stein's Equation and its Solution}.
	
It follows that for any real function $g$ with a second derivative, 
	\begin{align*}
		& \mathbb{E}[W_ng(W_n)-g^{\prime}(W_n)] \\
		=&\frac{1}{\sqrt{\sigma_n^0}}\sum_{j=1}^{n}\mathbb{E}[Y_jg(W_n)-Y_jg(W_n^{(j)})]-\mathbb{E}g^{\prime}(W_n) \\
		=&\frac{1}{\sqrt{\sigma_n^0}}\sum_{j=1}^{n}\mathbb{E}\left[Y_j(W_n-W_n^{(j)})\int_{0}^{1}g^{\prime}(W_n^{(j)}+t(W_n-W_n^{(j)}))dt\right]-\mathbb{E}g^{\prime}(W_n) \\
		=&\frac{1}{\sigma_n^0} \sum_{j=1}^{n}\mathbb{E}\left[Y_j^2\int_{0}^{1}g^{\prime}(W_n^{(j)}+t(W_n-W_n^{(j)}))dt\right]-\frac{1}{\sigma_n^0}\sum_{j=1}^{n}\mathbb{E}Y_j^2\mathbb{E}g^{\prime}(W_n) \\
		&  +\frac{1}{\sigma_n^0}\sum_{j=1}^{n}\mathbb{E}\left[Y_j\left(\sum_{k> j}(Y_k-Y_{kj})\right)\int_{0}^{1}g^{\prime}(W_n^{(j)}+t(W_n-W_n^{(j)}))dt\right] \\
		=&\frac{1}{\sigma_n^0}\sum_{j=1}^{n}\mathbb{E}\left[Y_j^2\left(\int_{0}^{1}(g^{\prime}(W_n^{(j)}+t(W_n-W_n^{(j)}))-g^{\prime}(W_n^{(j)}))dt\right)\right]\\
		&-\frac{1}{\sigma_n^0}\sum_{j=1}^{n}\mathbb{E}Y_j^2\mathbb{E}\left[\int_{0}^{1}(g^{\prime}(W_n^{(j)}+t(W_n-W_n^{(j)}))-g^{\prime}(W_n^{(j)}))dt\right]\\
		&+\frac{1}{\sigma_n^0}\sum_{j=1}^{n}\mathbb{E}Y_j^2 \mathbb{E}\left[\int_{0}^{1}(g^{\prime}(W_n^{(j)}+t(W_n-W_n^{(j)}))-g^{\prime}(W_n))dt\right] \\
		&+\frac{1}{\sigma_n^0}\sum_{j=1}^{n}\mathbb{E}\left[Y_j\left(\sum_{k> j}(Y_k-Y_{kj})\right)\int_{0}^{1}(g^{\prime}(W_n^{(j)}+t(W_n-W_n^{(j)}))-g^{\prime}(W_n^{(j)}))dt\right] \\
		&+\frac{1}{\sigma_n^0}\sum_{j=1}^{n}\mathbb{E}\left[(Y_j^2-\mathbb{E}Y_j^2)g^{\prime}(W_n^{(j)})\right]  +\frac{1}{\sigma_n^0}\sum_{j=1}^{n}\mathbb{E}\left[Y_j\left(\sum_{k> j}(Y_k-Y_{kj})\right)g^{\prime}(W_n^{(j)})\right].
	\end{align*}
	Replacing $\alpha$ with $\theta_n$  in equation \eqref{initial_smooth_indicator_function} obtains 
	\begin{align}\label{hw_0 theta_n}
		h_{w_0,\theta_n}(w)=\left\{
		\begin{array}{rcl}
			&1 ,\quad  & {w\leq w_0} \\
			&1+(w_0-w)/\theta_n,\quad  & {w_0< w\leq w_0+\theta_n} \\
			&0 ,\quad  & {w> w_0+\theta_n}.
		\end{array}
		\right.
	\end{align}
	Here, $\theta_n=n^{-1/2+\kappa}$ represents the order of magnitude of the bound  $\lvert \mathbb{P}(W_n\leq w_0)-\Phi(w_0)\rvert$. 
	Let $g$ be identified with $g_h$, where $g_h$ is given by \eqref{stein_solution} and $h$ is determined by \eqref{hw_0 theta_n}. The following lemmas are introduced to established the order of $\lvert \mathbb{E}\left[W_ng_h(W_n)-g_h^{\prime}(W_n)\right]\rvert$. 
	
	Let $M$ be a constant satisfying $M> K$. Throughout the following, $K$ and $M$ are independent of $n$ and $h$. Define
	\begin{align*}
		\rho=\int_{0}^{1}(g_h^{\prime}(W_n^{(j)}+t(W_n-W_n^{(j)}))-g_h^{\prime}(W_n^{(j)}))dt,
	\end{align*}
	and	
	\begin{align*}
		\widetilde{\rho}=\int_{0}^{1}(g_h^{\prime}(W_n^{(j)}+t(W_n-W_n^{(j)}))-g_h^{\prime}(W_n))dt.
	\end{align*}

	\begin{lemma}\label{lemma of sumEYj2gh'}
		Under Assumptions \ref{the moment assumption}--\ref{assumption about test function}, \ref{assumption about RG case} or \ref{assumption about CG case},
		\begin{align*}
			\left|\sum_{j=1}^{n}\mathbb{E}\left[Y_j^2\rho\right] \right| \leq Kn^{-1/2}+\frac{K}{M}\mathbb{K}(W_n,Z)+\frac{K\theta_n}{M}.
		\end{align*}
	\end{lemma}
	
	\begin{lemma}\label{lemma of sumYj2E(g'-g'(W^j))}
		Under Assumptions \ref{the moment assumption}--\ref{assumption about test function}, \ref{assumption about RG case} or \ref{assumption about CG case},
		\begin{align*}
			 \left|\sum_{j=1}^{n}\mathbb{E}Y_j^2\mathbb{E}\rho\right| \leq Kn^{-1/2}+\frac{K}{M}\mathbb{K}(W_n,Z)+\frac{K\theta_n}{M}.
		\end{align*}
	\end{lemma}
	
	\begin{lemma}\label{lemma of sumYj2E(g'-g'(W))}
		Under Assumptions \ref{the moment assumption}--\ref{assumption about test function}, \ref{assumption about RG case} or \ref{assumption about CG case},
		\begin{align*}
			 \left|\sum_{j=1}^{n}\mathbb{E}Y_j^2 \mathbb{E}\widetilde{\rho}\right| \leq Kn^{-1/2}+\frac{K}{M}\mathbb{K}(W_n,Z)+\frac{K\theta_n}{M}.
		\end{align*}
	\end{lemma}
	
	\begin{lemma}\label{lemma of sumEYjsum(Yk-Ykj)gh'}
		Under Assumptions \ref{the moment assumption}--\ref{assumption about test function}, \ref{assumption about RG case} or \ref{assumption about CG case},
		\begin{align*}
			&\left| \sum_{j=1}^{n}\mathbb{E}\left[Y_j\left(\sum_{k> j}(Y_k-Y_{kj})\right)\rho\right]\right| \leq Kn^{-1/2}+\frac{1}{M}\mathbb{K}(W_n,Z)+\frac{K\theta_n}{M}.
		\end{align*}
	\end{lemma}

	\begin{lemma}\label{lemma of sumEYjsumYk-Ykjgh'Wnj}
		Under Assumptions \ref{the moment assumption}--\ref{assumption about test function}, \ref{assumption about RG case} or \ref{assumption about CG case},
		\begin{align*}
			\left|\sum_{j=1}^{n}\mathbb{E}\left[ Y_j\left(\sum_{k>j}(Y_k-Y_{kj})\right)g_h^{\prime}(W_n^{(j)})\right]\right|\leq Kn^{-1/2}.
		\end{align*}
	\end{lemma}
	
	\begin{lemma}\label{sumE|Y_j^2-EY_j^2|}
		Under Assumptions \ref{the moment assumption}--\ref{assumption about test function}, \ref{assumption about RG case} or \ref{assumption about CG case},
		\begin{align*}
			\left|\sum_{j=1}^{n}\mathbb{E}\left[(Y_j^2-\mathbb{E}Y_j^2)g_h^{\prime}(W_n^{(j)})\right]\right|\leq Kn^{-1}.
		\end{align*}
	\end{lemma}
	
	\begin{remark}
		The proofs of Lemmas \ref{lemma of sumEYj2gh'}--\ref{lemma of sumEYjsum(Yk-Ykj)gh'} share a common structure: they commence with an analysis of $\rho$ or $\widetilde{\rho}$, then apply the Stein equation \eqref{stein_equation} to control the difference in $g_h^{\prime}$, and subsequently analyze the corresponding difference in $h$. The primary distinction between Lemmas $\ref{lemma of sumEYj2gh'}$ and $\ref{lemma of sumEYjsum(Yk-Ykj)gh'}$ lies in the multiplicative factor appearing before $\rho$. The proof of Lemma \ref{lemma of sumEYj2gh'} is presented in Section \ref{Proof of Lemma sumEYj2gh'}, while the proofs of the remaining lemmas are deferred to the Appendix \ref{the lemmas in Proposition of using stein method estimate rate}.
Moreover, since the proof strategies for Lemmas \ref{lemma of sumEYjsumYk-Ykjgh'Wnj} and \ref{sumE|Y_j^2-EY_j^2|} are analogous, we only focus on the proof of Lemma \ref{sumE|Y_j^2-EY_j^2|},  which is presented in Section \ref{Proof of Lemma sumE|Y_j^2-EY_j^2|}. The proof of Lemma \ref{lemma of sumEYjsumYk-Ykjgh'Wnj} is postponed in Appendix \ref{the lemmas in Proposition of using stein method estimate rate} for interested readers.	\end{remark}
	
		In view of \eqref{stein_identity} and Lemmas \ref{lemma of sumEYj2gh'}-\ref{sumE|Y_j^2-EY_j^2|}, 
	\begin{align}\label{bound Nh-Eh}
		&  \lvert\mathbb{E}g_{h}(W_n)W_n-\mathbb{E}g_{h}^{\prime}(W_n)\rvert = \lvert Nh_{w_0,\theta_n}-\mathbb{E}h_{w_0,\theta_n}(W_n)\rvert \leq Kn^{-1/2}+\frac{K}{M}\mathbb{K}(W_n,Z)+\frac{K\theta_n}{M}.
	\end{align}
	Analogously to \eqref{hw_0 theta_n}, we define $h_{w_0-\theta_n,\theta_n}(w)$, as shown in Fig \ref{about h}. Here, $h_{w_0}(w)$ denotes the indicator function of the interval $( -\infty,w_0] $. 
	
	\begin{figure}[htbp]
		\centering
		\begin{tikzpicture}
			\draw[->] (-2,0)--(8,0)node[below]{$w$};
			\draw[->] (0,-1)--(0,4);
			\node at (-.3, -.3) {$O$};
			\draw (0,3)--(.1,3);
			\node at (0,3.2)[left] {$1$};
			\draw (2,0)--(2,.1);
			\draw (4,0)--(4,.1);
			\draw (6,0)--(6,.1);
			\node at (2,-.3) {$w_0-\theta_n$};
			\node at (4,-.3) {$w_0$};
			\node at (6,-.3) {$w_0+\theta_n$}; 
			
			\node at (4,0.7)[right]{$h_{w_0}(w)$};
			\node at (5,1.6)[right]{$h_{w_0,\theta_n}(w)$};
			\node at (3.3,1)[left]{$h_{w_0-\theta_n,\theta_n}(w)$};
			
			\draw[thick, color=black] (-2,3)--(2,3);
			\draw[thick, color=black] (2,3)--(4,0);
			\draw[thick, color=black] (4,0)--(6,0);

			\draw [thick, color=black] (2,3)--(4,3);
			\draw [thick, color=black] (4,3)--(6,0);
			\draw [thick, color=black] (6,0)--(7,0);
			
			\draw [dashed, color=black] (4,0)--(4,3);
			
			\draw[thick, color=black] (-2,3)--(4,3);
			\draw[color=black] (4,0) circle (.1);
			\draw[thick, color=black] (4.1,0)--(7,0);
		\end{tikzpicture}
		\caption{The functions $h_{w_0-\theta_n,\theta_n}(w)$,$h_{w_0}(w)$ and $h_{w_0,\theta_n}(w)$.}
		\label{about h}
	\end{figure}
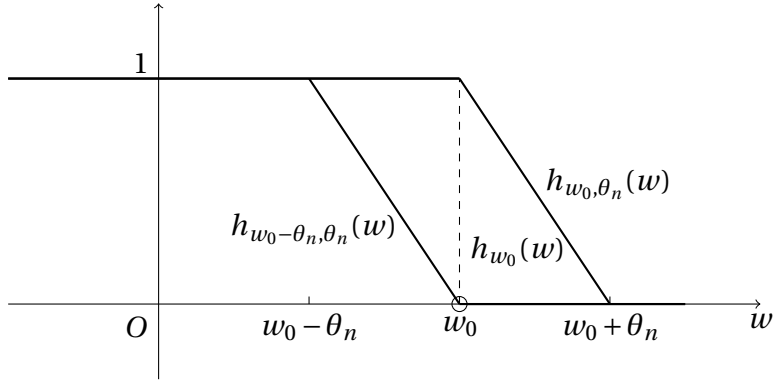

	Based on the definition of $h_{w_0-\theta_n,\theta_n}(w)$ and $h_{w_0,\theta_n}(w)$,
	\begin{align*}
		& \mathbb{E}h_{w_0-\theta_n,\theta_n}(W_n) \\
		=& \mathbb{P}(W_n\leq w_0-\theta_n)+\mathbb{E}\left[ \left(1-\frac{W_n+\theta_n-w_0}{\theta_n}\right)I\left\lbrace w_0-\theta_n\leq W_n\leq w_0\right\rbrace \right] \\
		=& \mathbb{P}(W_n\leq w_0-\theta_n)+\mathbb{E}\left[ \frac{w_0-W_n}{\theta_n}I\left\lbrace w_0-\theta_n\leq W_n\leq w_0\right\rbrace \right] \\
		\leq& \mathbb{P}(W_n\leq w_0-\theta_n)+\mathbb{P}(w_0-\theta_n\leq W_n\leq w_0) = \mathbb{P}(W_n\leq w_0) \\
		\leq& \mathbb{P}(W_n\leq w_0)+\mathbb{E}\left[\left(1-\frac{W_n-w_0}{\theta_n}\right)I\left\lbrace w_0\leq W_n\leq w_0+\theta_n\right\rbrace \right]=\mathbb{E}h_{w_0,\theta_n}(W_n).
	\end{align*}
	Furthermore, this implies
	\begin{align}\label{left hand of F_W_n-Phi}
		&  \mathbb{E}h_{w_0-\theta_n,\theta_n}(W_n)-Nh_{w_0-\theta_n,\theta_n}+Nh_{w_0-\theta_n,\theta_n}-\Phi(w_0) \leq \mathbb{P}(W_n\leq w_0)-\Phi(w_0)
	\end{align}
	and
	\begin{align}\label{right hand of F_W_n-Phi}
		& \mathbb{P}(W_n\leq w_0)-\Phi(w_0)
		\leq \mathbb{E}h_{w_0,\theta_n}(W_n)-Nh_{w_0,\theta_n}+Nh_{w_0,\theta_n}-\Phi(w_0).
	\end{align}
	
	We first consider \eqref{left hand of F_W_n-Phi}. By the mean value theorem, 
	\begin{align*}
		& Nh_{w_0-\theta_n,\theta_n}-\Phi(w_0)=\frac{1}{\sqrt{2\pi}}\int_{-\infty}^{+\infty}h_{w_0-\theta_n,\theta_n}(x)e^{-x^2/2}dx-\Phi(w_0) \\
		=& \frac{1}{\sqrt{2\pi}}\int_{-\infty}^{w_0-\theta_n}e^{-x^2/2}dx + \frac{1}{\sqrt{2\pi}}\int_{w_0-\theta_n}^{w_0}\frac{w_0-x}{\theta_n}e^{-x^2/2}dx-\Phi(w_0) \\
		\geq& \frac{1}{\sqrt{2\pi}} \int_{-\infty}^{w_0-\theta_n}e^{-x^2/2}dx-\Phi(w_0)=\Phi(w_0-\theta_n)-\Phi(w_0) = -\Phi^{\prime}(w_1)\theta_n \geq -\frac{\theta_n}{\sqrt{2\pi}},
	\end{align*}
	where $w_1\in\left[w_0-\theta_n,w_0\right]$. With the aid of the upper bound from the Stein equation \eqref{bound Nh-Eh}, we have 
	\begin{align*}
		\mathbb{E}h_{w_0-\theta_n,\theta_n}(W_n)-Nh_{w_0-\theta_n,\theta_n}\geq -Kn^{-1/2}-\frac{K}{M}\mathbb{K}(W_n,Z)-\frac{K\theta_n}{M}.
	\end{align*}
	Thus, combining with \eqref{left hand of F_W_n-Phi}, we obtain the lower bound
	\begin{align}\label{lower bound of F_W_n-Phi}
		\mathbb{P}(W_n\leq w_0)-\Phi(w_0) \geq -Kn^{-1/2}-\frac{K}{M}\mathbb{K}(W_n,Z)-\frac{K\theta_n}{M}.
	\end{align}

	Next, we address \eqref{right hand of F_W_n-Phi}. Analogously to the preceding argument, 
	\begin{align*}
		& Nh_{w_0,\theta_n}-\Phi(w_0) = \frac{1}{\sqrt{2\pi}}\int_{-\infty}^{\infty}h_{w_0,\theta_n}(x)e^{-x^2/2}dx-\Phi(w_0) \\
		=& \frac{1}{\sqrt{2\pi}}\int_{-\infty}^{w_0}e^{-x^2/2}dx + \frac{1}{\sqrt{2\pi}}\int_{w_0}^{w_0+\theta_n}(1-\frac{x-w_0}{\theta_n})e^{-x^2/2}dx-\Phi(w_0) \\
		\leq& \frac{1}{\sqrt{2\pi}}\int_{-\infty}^{w_0}e^{-x^2/2}dx + \frac{1}{\sqrt{2\pi}}\int_{w_0}^{w_0+\theta_n}e^{-x^2/2}dx-\Phi(w_0) \\
		=& \Phi(w_0+\theta_n)-\Phi(w_0)=\Phi^{\prime}(w_2)\theta_n\leq \frac{\theta_n}{\sqrt{2\pi}} 
	\end{align*} 
	where $w_2\in\left[w_0,w_0+\theta_n\right]$. By utilizing the upper bound \eqref{bound Nh-Eh},
	\begin{align*}
		\mathbb{E}h_{w_0,\theta_n}(W_n)-Nh_{w_0,\theta_n}\leq Kn^{-1/2}+\frac{K}{M}\mathbb{K}(W_n,Z)+\frac{K\theta_n}{M}.
	\end{align*}
	Therefore, by combining with  \eqref{right hand of F_W_n-Phi}, we derive the upper bound
	\begin{align}\label{upper bound of F_W_n-Phi}
		\mathbb{P}(W_n\leq w_0)-\Phi(w_0)\leq Kn^{-1/2}+\frac{K}{M}\mathbb{K}(W_n,Z)+\frac{K\theta_n}{M}.
	\end{align}
	
	In summary, from \eqref{lower bound of F_W_n-Phi} and \eqref{upper bound of F_W_n-Phi}, we can get the bound
	\begin{align*}
		\lvert \mathbb{P}(W_n\leq w_0)-\Phi(w_0)\rvert \leq Kn^{-1/2}+\frac{K}{M}\mathbb{K}(W_n,Z)+\frac{K\theta_n}{M}.
	\end{align*}
	Since the right-hand side of the preceding inequality is independent of $w_0$, we can further obtain 
	\begin{align*}
		\mathbb{K}(W_n,Z)=\sup_{w_0\in\mathbb{R}}\lvert \mathbb{P}(W_n\leq w_0)-\Phi(w_0)\rvert\leq Kn^{-1/2}+\frac{K}{M}\mathbb{K}(W_n,Z)+\frac{Kn^{-1/2+\kappa}}{M}.
	\end{align*}
	Since $M>K$, it follows that
	\begin{align*}
		\mathbb{K}(W_n,Z)\leq Kn^{-1/2+\kappa}.
	\end{align*}
	The proof of Proposition \ref{Propostition of using stein method estimate rate} is complete.
	
	\subsection{Proof of Lemma \ref{lemma of sumEYj2gh'}}\label{Proof of Lemma sumEYj2gh'}
		We begin by decomposition of $\rho$ using the Stein equation \eqref{stein_equation} for $g_h$:
		\begin{align}\label{decomposition of rho}
			 \rho =& \int_{0}^{1}(g_h^{\prime}(W_n^{(j)}+t(W_n-W_n^{(j)}))-g_h^{\prime}(W_n^{(j)}))dt \\
			\nonumber =&\int_{0}^{1}((W_n^{(j)}+t(W_n-W_n^{(j)}))g_h(W_n^{(j)}+t(W_n-W_n^{(j)}))-W_n^{(j)}g_h(W_n^{(j)}))dt \\
			\nonumber &+\int_{0}^{1}(h_{w_0,\theta_n}(W_n^{(j)}+t(W_n-W_n^{(j)}))-h_{w_0,\theta_n}(W_n^{(j)}))dt.
		\end{align}
		Due to \eqref{decomposition of rho} and Lemma \ref{properties_of_the_smoothed_stein_solution},
		\begin{align}\label{sumEYj2gh'}
			\nonumber&\left|\sum_{j=1}^{n}\mathbb{E}\left[Y_j^2\rho\right]\right|  \\
			 \nonumber \leq & \sum_{j=1}^{n}\mathbb{E}\left[\lvert Y_j\rvert^2\int_{0}^{1}\lvert t(W_n-W_n^{(j)})\rvert\lvert g_h(W_n^{(j)}+ts(W_n-W_n^{(j)}))\rvert dt\right] \\
			\nonumber &+\sum_{j=1}^{n}\mathbb{E}\left[\lvert Y_j\rvert^2\int_{0}^{1}\lvert t(W_n-W_n^{(j)})\rvert\lvert (W_n^{(j)}+ts(W_n-W_n^{(j)}))g_h^{\prime}(W_n^{(j)}+ts(W_n-W_n^{(j)}))\rvert)dt\right] \\
			\nonumber &+\left|\sum_{j=1}^{n}\mathbb{E}\left[Y_j^2\left(\int_{0}^{1}(h_{w_0,\theta_n}(W_n^{(j)}+t(W_n-W_n^{(j)}))-h_{w_0,\theta_n}(W_n^{(j)}))dt\right)\right]\right|  \\
			\leq & K\sum_{j=1}^{n}\mathbb{E}\left[\lvert Y_j\rvert^2\lvert W_n-W_n^{(j)}\rvert\right]+K\sum_{j=1}^{n}\mathbb{E}\left[\lvert Y_j\rvert^2\lvert W_n-W_n^{(j)}\rvert\lvert W_n^{(j)}\rvert\right] \\
			\nonumber &+K\sum_{j=1}^{n}\mathbb{E}\left[\lvert Y_j\rvert^2\lvert W_n-W_n^{(j)}\rvert^2\right] \\
			\nonumber &+\left|\sum_{j=1}^{n}\mathbb{E}\left[Y_j^2\left(\int_{0}^{1}(h_{w_0,\theta_n}(W_n^{(j)}+t(W_n-W_n^{(j)}))-h_{w_0,\theta_n}(W_n^{(j)}))dt\right)\right]\right|,
		\end{align}
		where $s\in\left[0,1\right]$. In order to estimate the order of \eqref{sumEYj2gh'}, we first need to estimate the following two terms. Analogous to the proof of Lemma \ref{Taylor_expansion_2}, due to \eqref{estimation formulas for complex integrals}, Lemmas \ref{bound moment of b_j and beta_j}, \ref{QMTr_1} and  the Jensen's inequality, for any $p\geq 1$,
		\begin{align}\label{Yjs}
			\mathbb{E}\lvert Y_j^p\rvert\leq \mathbb{E}\lvert Y_j\rvert^p  \leq &K\left[ (x_l-x_r)^{p-1} \int_{x_l}^{x_r} \mathbb{E}\lvert \varepsilon_j(u+iv_0)b_j(u+iv_0)\rvert^p du \right.\\
			\nonumber & \quad+(x_l-x_r)^{p-1} \int_{x_l}^{x_r} \mathbb{E}\lvert \varepsilon_j(u-iv_0)b_j(u-iv_0)\rvert^p du  \\
			\nonumber & \quad+(2v_0)^{p-1}\int_{-v_0}^{v_0} \mathbb{E}\lvert \varepsilon_j(x_r+iv)b_j(x_r+iv)\rvert^p dv  \\
			\nonumber & \quad \left. +(2v_0)^{p-1}\int_{-v_0}^{v_0} \mathbb{E}\lvert \varepsilon_j(x_l+iv)b_j(x_l+iv)\rvert^p dv  \right]\\
			\nonumber \leq & Kn^{-p/2},
		\end{align}
		With regards to $\mathbb{E}\lvert \sum_{k>j}(Y_k-Y_{kj})\rvert^t$, $t\geq 1$, we need to estimate, for $z\in\mathcal{C}$, due to \eqref{D_j^{-1}-D_kj^{-1}}, Lemmas \ref{Burkholder_2}, \ref{bound moment of b_j and beta_j}, \ref{QMTr} and \ref{quadratic form minus trace of qudratic form},
		\begin{align*}
			& \mathbb{E}\lvert\sum_{k>j}(\mathbb{E}_k-\mathbb{E}_{k-1})[\bbr_k^*(\bbD_k^{-1}(z)-\bbD_{kj}^{-1}(z))\bbr_k-n^{-1}\operatorname{tr}\bbT(\bbD_k^{-1}(z)-\bbD_{kj}^{-1}(z))]\rvert^t \\
			\leq & K(\sum_{k> j}\mathbb{E}\lvert(\mathbb{E}_k-\mathbb{E}_{k-1})[\bbr_k^*(\bbD_k^{-1}(z)-\bbD_{kj}^{-1}(z))\bbr_k-n^{-1}\operatorname{tr}\bbT(\bbD_k^{-1}(z)-\bbD_{kj}^{-1})]\rvert^2)^{t/2} \\
			& + K\sum_{k>j}\mathbb{E}\lvert (\mathbb{E}_k-\mathbb{E}_{k-1})[\bbr_k^*(\bbD_k^{-1}(z)-\bbD_{kj}^{-1}(z))\bbr_k-n^{-1}\operatorname{tr}\bbT(\bbD_k^{-1}(z)-\bbD_{kj}^{-1})]\rvert^t \\
			=&K(\sum_{k> j}\mathbb{E}\lvert (\mathbb{E}_k-\mathbb{E}_{k-1}) \beta_{kj}(z)[\bbr_k^{*}\bbD_{kj}^{-1}(z)\bbr_j\bbr_j^{*}\bbD_{kj}^{-1}(z)\bbr_k-n^{-1}\operatorname{tr}\bbT\bbD_{kj}^{-1}(z)\bbr_j\bbr_j^{*}\bbD_{kj}^{-1}(z)]\rvert^2)^{t/2} \\
			&+K\sum_{k>j}\mathbb{E}\lvert(\mathbb{E}_k-\mathbb{E}_{k-1})\beta_{kj}(z)[\bbr_k^{*}\bbD_{kj}^{-1}(z)\bbr_j\bbr_j^{*}\bbD_{kj}^{-1}(z)\bbr_k -n^{-1}\operatorname{tr}\bbT\bbD_{kj}^{-1}(z)\bbr_j\bbr_j^{*}\bbD_{kj}^{-1}(z)]\rvert^t \\
			\leq &  K(\sum_{k>j}\mathbb{E}\lvert (\beta_{kj}-\widetilde{\beta}_{kj})(\bbr_k^{*}\bbD_{kj}^{-1}(z)\bbr_j\bbr_j^{*}\bbD_{kj}^{-1}(z)\bbr_k-n^{-1}\operatorname{tr}\bbT\bbD_{kj}^{-1}(z)\bbr_j\bbr_j^{*}\bbD_{kj}^{-1}(z))\rvert^2)^{t/2} \\
			& +K(\sum_{k>j}\mathbb{E}\lvert \widetilde{\beta}_{kj}(\bbr_k^{*}\bbD_{kj}^{-1}(z)\bbr_j\bbr_j^{*}\bbD_{kj}^{-1}(z)\bbr_k-n^{-1}\operatorname{tr}\bbT\bbD_{kj}^{-1}(z)\bbr_j\bbr_j^{*}\bbD_{kj}^{-1}(z))\rvert^2)^{t/2} \\
			& +K \sum_{k>j}\mathbb{E}\lvert(\beta_{kj}-\widetilde{\beta}_{kj})(\bbr_k^{*}\bbD_{kj}^{-1}(z)\bbr_j\bbr_j^{*}\bbD_{kj}^{-1}(z)\bbr_k-n^{-1}\operatorname{tr}\bbT\bbD_{kj}^{-1}(z)\bbr_j\bbr_j^{*}\bbD_{kj}^{-1}(z))\rvert^t \\
			&+K \sum_{k>j}\mathbb{E}\lvert\widetilde{\beta}_{kj}(\bbr_k^{*}\bbD_{kj}^{-1}(z)\bbr_j\bbr_j^{*}\bbD_{kj}^{-1}(z)\bbr_k-n^{-1}\operatorname{tr}\bbT\bbD_{kj}^{-1}(z)\bbr_j\bbr_j^{*}\bbD_{kj}^{-1}(z))\rvert^t \\
			\leq& Kn^{-t/2}+Kn^{-t+1+((t/2-5/2)\vee 0)}\eta_n^{(2t-10)\vee 0} \\
			&+K\sum_{k>j}\sqrt{\mathbb{E}\lvert\beta_{kj}-\widetilde{\beta}_{kj} \rvert^{2t}} \times  \sqrt{\mathbb{E}\lvert \bbr_k^{*}\bbD_{kj}^{-1}(z)\bbr_j\bbr_j^{*}\bbD_{kj}^{-1}(z)\bbr_k-n^{-1}\operatorname{tr}\bbT\bbD_{kj}^{-1}(z)\bbr_j\bbr_j^{*}\bbD_{kj}^{-1}(z) \rvert^{2t}} \\
			\leq& Kn^{-t/2}.
		\end{align*}
		Thus, for $t\geq 1$, 
		\begin{align}\label{(sum k>j(Yk-Ykj))t}
			\mathbb{E}\lvert \sum_{k>j}(Y_k-Y_{kj})\rvert^t\leq Kn^{-t/2}.
		\end{align}
		
		After the preparation outlined above, we now consider the order of \eqref{sumEYj2gh'}. From \eqref{Yjs}, \eqref{(sum k>j(Yk-Ykj))t} and Lemma \ref{properties_of_the_smoothed_stein_solution},
		\begin{align*}
			 \eqref{sumEYj2gh'} &\leq K\sum_{j=1}^{n}\mathbb{E}\left[\lvert Y_j\rvert^2\lvert W_n-W_n^{(j)}\rvert\right]+K\sum_{j=1}^{n}\mathbb{E}\left[\lvert Y_j\rvert^2\lvert W_n-W_n^{(j)}\rvert\lvert W_n^{(j)}\rvert\right] \\
			 \nonumber&\quad +K\sum_{j=1}^{n}\mathbb{E}\left[\lvert Y_j\rvert^2\lvert W_n-W_n^{(j)}\rvert^2\right] \\
			 \nonumber&\quad +\left|\sum_{j=1}^{n}\mathbb{E}\left[Y_j^2\left(\int_{0}^{1}(h_{w_0,\theta_n}(W_n^{(j)}+t(W_n-W_n^{(j)}))-h_{w_0,\theta_n}(W_n^{(j)}))dt\right)\right]\right| \\
			&\leq Kn^{-1/2}+Kn^{-1} \\
			&\quad +\left|\sum_{j=1}^{n}\mathbb{E}\left[Y_j^2\left(\int_{0}^{1}(h_{w_0,\theta_n}(W_n^{(j)}+t(W_n-W_n^{(j)}))-h_{w_0,\theta_n}(W_n^{(j)}))dt\right)\right]\right|.
		\end{align*}
		Next, our focus shifts to estimating the order of 
		\begin{align*}
			\left|\sum_{j=1}^{n}\mathbb{E}\left[Y_j^2\left(\int_{0}^{1}(h_{w_0,\theta_n}(W_n^{(j)}+t(W_n-W_n^{(j)}))-h_{w_0,\theta_n}(W_n^{(j)}))dt\right)\right]\right|.
		\end{align*}
		Recalling the definition of $h_{w_0,\theta_n}$ \eqref{hw_0 theta_n}, $h_{w_0,\theta_n}^{\prime}$ is equal to $1/\theta_n$ over a small interval and zero elsewhere. Therefore, due to \eqref{Yjs}, \eqref{(sum k>j(Yk-Ykj))t}, Lemmas \ref{bound moment of b_j and beta_j} and \ref{QMTr_1} ,
		\begin{align}\label{sumEYj2h'}
			&	\left|\sum_{j=1}^{n}\mathbb{E}\left[Y_j^2\left(\int_{0}^{1}(h_{w_0,\theta_n}(W_n^{(j)}+t(W_n-W_n^{(j)}))-h_{w_0,\theta_n}(W_n^{(j)}))dt\right)\right]\right| \\
			\nonumber\leq&\left|\sum_{j=1}^{n}\mathbb{E}\left[Y_j^2\left(\int_{0}^{1}(h_{w_0,\theta_n}(W_n^{(j)}+t(W_n-W_n^{(j)}))-h_{w_0,\theta_n}(W_n^{(j)}))dt\right)I\left\lbrace \lvert W_n-W_n^{(j)}\rvert\geq \frac{\theta_n}{M}\right\rbrace \right]\right| \\
			\nonumber&+\left|\sum_{j=1}^{n}\mathbb{E}\left[Y_j^2\left(\int_{0}^{1}(h_{w_0,\theta_n}(W_n^{(j)}+t(W_n-W_n^{(j)}))-h_{w_0,\theta_n}(W_n^{(j)}))dt\right)\right.\right. \\
			\nonumber&\left.\left. \quad\quad\quad\quad\quad\quad\quad I\left\lbrace \lvert W_n-W_n^{(j)}\rvert< \frac{\theta_n}{M},W_n^{(j)}\in\left[w_0-\theta_n,w_0+2\theta_n\right]\right\rbrace \right]\right| \\
			\nonumber&+\left|\sum_{j=1}^{n}\mathbb{E}\left[Y_j^2\left(\int_{0}^{1}(h_{w_0,\theta_n}(W_n^{(j)}+t(W_n-W_n^{(j)}))-h_{w_0,\theta_n}(W_n^{(j)}))dt\right)\right.\right. \\
			\nonumber&\left.\left. \quad\quad\quad\quad\quad\quad\quad I\left\lbrace \lvert W_n-W_n^{(j)}\rvert< \frac{\theta_n}{M},W_n^{(j)}\notin\left[w_0-\theta_n,w_0+2\theta_n\right]\right\rbrace \right]\right| \\
			\nonumber \leq&\sum_{j=1}^{n}\left(\mathbb{E}\left| Y_j\right|^4\right)^{1/2}\left(\mathbb{P}\left(\lvert W_n-W_n^{(j)}\rvert\geq \frac{\theta_n}{M}\right)\right)^{1/2} \\
			\nonumber &+\sum_{j=1}^{n}\frac{1}{2\theta_n}\mathbb{E}\left[\lvert Y_j\rvert^2\lvert W_n-W_n^{(j)}\rvert   I\left\lbrace \lvert W_n-W_n^{(j)}\rvert< \frac{\theta_n}{M},W_n^{(j)}\in\left[w_0-\theta_n,w_0+2\theta_n\right]\right\rbrace\right] \\
			\nonumber \leq& \sum_{j=1}^{n} Kn^{-1}\left(\frac{M^m \mathbb{E}\lvert W_n-W_n^{(j)}\rvert^m}{\theta_n^m}\right)^{1/2} \\
			\nonumber &+\sum_{j=1}^{n}\frac{1}{2M}\mathbb{E}\left[\lvert Y_j\rvert^2 I\left\lbrace W_n^{(j)}\in\left[w_0-\theta_n,w_0+2\theta_n\right]\right\rbrace\right] \\
			\nonumber \leq&K\left(\frac{M^m n^{-m/2}}{n^{-m/2}n^{\kappa m}}\right)^{1/2}+\sum_{j=1}^{n}\frac{1}{2M}\mathbb{E}\left[I\left\lbrace W_n^{(j)}\in\left[w_0-\theta_n,w_0+2\theta_n\right]\right\rbrace \mathbb{E}_n^{(j)}\left[\lvert Y_j\rvert^2\right]\right] \\
			\nonumber \leq &K M^{m/2}n^{-\kappa m/2}+\frac{Kn^{-1}}{2M}\sum_{j=1}^{n}\mathbb{E}\left[I\left\lbrace W_n^{(j)}\in\left[w_0-\theta_n,w_0+2\theta_n\right],\lvert W_n-W_n^{(j)}\rvert\geq \theta_n\right\rbrace\right] \\
			\nonumber &+\frac{Kn^{-1}}{2M}\sum_{j=1}^{n}\mathbb{E}\left[I\left\lbrace W_n^{(j)}\in\left[w_0-\theta_n,w_0+2\theta_n\right],\lvert W_n-W_n^{(j)}\rvert<\theta_n\right\rbrace\right] \\
			\nonumber \leq&K M^{m/2}n^{-\kappa m/2}+\frac{Kn^{-1}}{M}\sum_{j=1}^{n}\frac{\mathbb{E}\lvert W_n-W_n^{(j)}\rvert^m}{\theta_n^m} +\frac{K}{M}\mathbb{P}\left(W_n\in\left[w_0-2\theta_n,w_0+3\theta_n\right]\right) \\
			\nonumber \leq&K M^{m/2}n^{-\kappa m/2}+\frac{K}{M}\left|\Phi(w_0+3\theta_n)-\Phi(w_0-2\theta_n)\right| \\
			\nonumber &+\frac{K}{M}\left| \mathbb{P}\left(W_n\in\left[w_0-2\theta_n,w_0+3\theta_n\right]\right)-\mathbb{P}\left(Z\in\left[w_0-2\theta_n,w_0+3\theta_n\right]\right)\right| \\
			\nonumber \leq&K_mn^{-\kappa m/2}+\frac{K\theta_n}{M}+\frac{K}{M}\mathbb{K}(W_n,Z).
		\end{align}
		Choose an appropriate $m$ such that $K_mn^{-\kappa m/2}\leq Kn^{-1/2}$. Thus, the proof of Lemma \ref{lemma of sumEYj2gh'} is complete.
	
	\subsection{Proof of Lemma \ref{sumE|Y_j^2-EY_j^2|}}\label{Proof of Lemma sumE|Y_j^2-EY_j^2|} 
	At the beginning of the proof, due to Lemma \ref{properties_of_the_smoothed_stein_solution}, we have
		\begin{align*}
			&\left|\sum_{j=1}^{n}\mathbb{E}\left[(Y_j^2-\mathbb{E}Y_j^2)g_h^{\prime}(W_n^{(j)})\right]\right| \\
			\leq&  \sum_{j=1}^{n}\mathbb{E}\lvert \mathbb{E}[Y_j^2|\sigma(\bbr_1,\dots,\bbr_{j-1},\bbr_{j+1},\dots,\bbr_n)]-\mathbb{E}Y_j^2\rvert \\
			=& \sum_{j=1}^{n}\mathbb{E}\lvert \mathbb{E}_{n}^{(j)}[Y_j^2]-\mathbb{E}[\mathbb{E}_{n}^{(j)}(Y_j^2)]\rvert \\
			=& \frac{1}{4\pi^2}\sum_{j=1}^{n}\mathbb{E}\left| \oint_{\mathcal{C}_1}\oint_{\mathcal{C}_2}f^{\prime}(z_1)f^{\prime}(z_2)b_j(z_1)b_j(z_2)\mathbb{E}_{n}^{(j)}[\mathbb{E}_j \varepsilon_j(z_1)\mathbb{E}_j \varepsilon_j(z_2)]dz_1 dz_2 \right. \\
			&\left. \quad\quad\quad\quad -\oint_{\mathcal{C}_1}\oint_{\mathcal{C}_2}f^{\prime}(z_1)f^{\prime}(z_2)b_j(z_1)b_j(z_2)\mathbb{E}\left[ \mathbb{E}_{n}^{(j)}[\mathbb{E}_j \varepsilon_j(z_1)\mathbb{E}_j \varepsilon_j(z_2)]\right] dz_1dz_2 \right|
		\end{align*}
		The contour $\mathcal{C}_1$ refers to the $ \mathcal{C}$ stated in Section \ref{Stieltjes transform}. The contour $\mathcal{C}_2$ is also a rectangle which intersects the real axis at $x_l-\epsilon$, $x_r+\epsilon$ (points where $f$ remains analytic), with height $2v_0$ (see Fig \ref{contour 1 and 2} for illustration).

		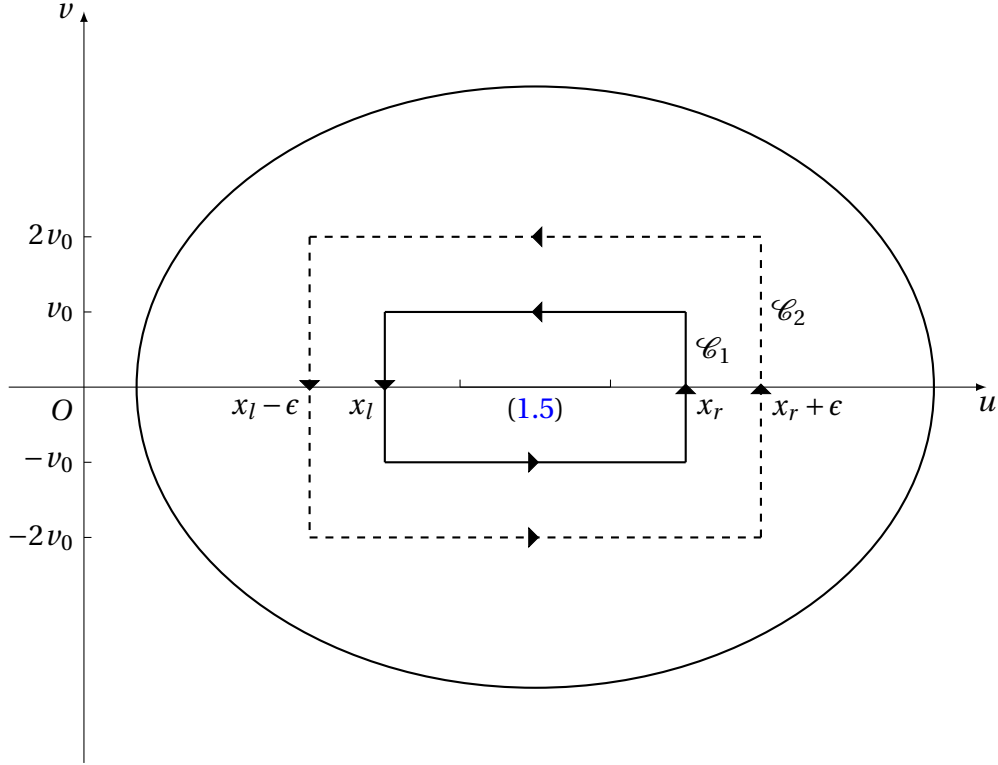
\begin{figure}[htbp]
			\centering
			\begin{tikzpicture}[>=latex]
				\draw[->] (-1,0)--(12,0)node[below]{$u$};
				\draw[->] (0,-5)--(0,5)node[left]{$v$};
				\node at (-.3, -.3) {$O$};
				\draw (0,1)node[left]{$v_0$}--(.1,1);
				\draw (0,2)node[left]{$2v_0$}--(.1,2);
				\draw (0,-1)node[left]{$-v_0$}--(.1,-1);
				\draw (0,-2)node[left]{$-2v_0$}--(.1,-2);
				
				\node at (3,-.3)[left] {$x_l-\epsilon$};
				\node at (4,-.3)[left] {$x_l$};
				\draw (5,0)--(5,.1);
				\draw (7,0)--(7,.1);
				\draw (5,0)--node[below]{$\eqref{supportset}$}(7,0);
				\node at (8,-.3) [right] {$x_r$};
				\node at (9,-.3) [right] {$x_r+\epsilon$};
				\node at (9,1) [right] {$\mathcal{C}_2$};
				\node at (8,.5) [right] {$\mathcal{C}_1$};
				
				\begin{scope}[color=black, thick, every node/.style={sloped,allow upside down}]
					\draw[dashed] (3,2)-- node {\midarrow} (3,-2);
					\draw[dashed] (3,-2)-- node {\midarrow} (9,-2);
					\draw[dashed] (9,-2)-- node {\midarrow} (9,2);
					\draw[dashed] (9,2)-- node {\midarrow} (3,2);
				\end{scope}

				\begin{scope}[color=black, thick, every node/.style={sloped,allow upside down}]
					\draw (4,1)-- node {\midarrow} (4,-1);
					\draw (4,-1)-- node {\midarrow} (8,-1);
					\draw (8,-1)-- node {\midarrow} (8,1);
					\draw (8,1)-- node {\midarrow} (4,1);
				\end{scope}
				
				\draw[color=black,thick] (6,0) ellipse (5.3 and 4);
				
			\end{tikzpicture}
			\caption{Contours $\mathcal{C}_1$ and $\mathcal{C}_2$. The solid contour on the inside represents $\mathcal{C}_1$, while the dashed contour on the outside corresponds to $\mathcal{C}_2$. The inner region enclosed by the outer elliptical curve represents the domain of analyticity of the function $f$.}
			\label{contour 1 and 2}
		\end{figure}
		
		Next, we only need to consider, for $z_1\in\mathcal{C}_2$ and $z_2\in\mathcal{C}_1$, due to Lemma \ref{product of two quadratic minus trace of two matrices},
		\begin{align}\label{E|Y_j^2-EY_j^2|}
			&  \mathbb{E}\left| \mathbb{E}_{n}^{(j)}[\mathbb{E}_j \varepsilon_j(z_1)\mathbb{E}_j \varepsilon_j(z_2)] -\mathbb{E}\left[ \mathbb{E}_{n}^{(j)}(\mathbb{E}_j \varepsilon_j(z_1)\mathbb{E}_j \varepsilon_j(z_2))\right] \right| \\
			\nonumber=& n^{-2}\mathbb{E}\left| (\mathbb{E}\lvert x_{1j}\rvert^4-\lvert \mathbb{E}x_{1j}^2\rvert^2-2)\operatorname{tr}(\bbT\mathbb{E}_j(\bbD_j^{-1}(z_1)))\circ(\bbT \mathbb{E}_j(\bbD_j^{-1}(z_2))) \right. \\
			\nonumber &\quad\quad\quad + \lvert \mathbb{E}x_{1j}^2 \rvert^2 \operatorname{tr} \bbT\mathbb{E}_j(\bbD_j^{-1}(z_1))(\bbT\mathbb{E}_j(\bbD_j^{-1}(z_2)))^{T} +\operatorname{tr}\bbT\mathbb{E}_j(\bbD_j^{-1}(z_1))\bbT\mathbb{E}_j(\bbD_j^{-1}(z_2)) \\
			\nonumber &\quad\quad\quad -(\mathbb{E}\lvert x_{1j}\rvert^4-\lvert \mathbb{E}x_{1j}^2\rvert^2-2) \mathbb{E}\operatorname{tr}(\bbT\mathbb{E}_{j}(\bbD_j^{-1}(z_1)))\circ (\bbT\mathbb{E}_j(\bbD_j^{-1}(z_2))) \\
			\nonumber &\quad\quad\quad \left. -\lvert \mathbb{E}x_{1j}^2 \rvert^2 \mathbb{E}\operatorname{tr} \bbT\mathbb{E}_j(\bbD_j^{-1}(z_1))(\bbT\mathbb{E}_j(\bbD_j^{-1}(z_2)))^{T} -\mathbb{E}\operatorname{tr} \bbT\mathbb{E}_j(\bbD_j^{-1}(z_1))\bbT\mathbb{E}_j(\bbD_j^{-1}(z_2))\right|
		\end{align}
		Under the RG case,
		\begin{align*}
			\eqref{E|Y_j^2-EY_j^2|} = &2n^{-2}\mathbb{E}\left| \operatorname{tr} \bbT\mathbb{E}_j(\bbD_j^{-1}(z_1))\bbT\mathbb{E}_j(\bbD_j^{-1}(z_2))-\mathbb{E}\operatorname{tr} \bbT\mathbb{E}_j(\bbD_j^{-1}(z_1))\bbT\mathbb{E}_j(\bbD_j^{-1}(z_2)) \right| +o(n^{-5/2}\eta_n^{4}).
		\end{align*} 
		Under the CG case,
		\begin{align*}
			\eqref{E|Y_j^2-EY_j^2|} =& n^{-2}\mathbb{E}\left| \operatorname{tr} \bbT\mathbb{E}_j(\bbD_j^{-1}(z_1))\bbT\mathbb{E}_j(\bbD_j^{-1}(z_2))-\mathbb{E}\operatorname{tr} \bbT\mathbb{E}_j(\bbD_j^{-1}(z_1))\bbT\mathbb{E}_j(\bbD_j^{-1}(z_2)) \right| +o(n^{-5/2}\eta_n^{4}).
		\end{align*} 
		Thus, we only need to estimate the order of
		\begin{align}\label{E|trTD_jTD_j-EtrTD_jTD_j|}
			\mathbb{E}\left| \operatorname{tr} \bbT\mathbb{E}_j(\bbD_j^{-1}(z_1))\bbT\mathbb{E}_j(\bbD_j^{-1}(z_2))-\mathbb{E}\operatorname{tr} \bbT\mathbb{E}_j(\bbD_j^{-1}(z_1))\bbT\mathbb{E}_j(\bbD_j^{-1}(z_2)) \right| .
		\end{align}
		Define $\breve{\bbD}_j(z)$ to have the same structure as the matrix $\bbD_j$ with random vectors $\bbr_{j+1},\dots,\bbr_n$ replaced by their i.i.d. copies $\breve{\bbr}_{j+1},\dots,\breve{\bbr}_n$. Define $\breve{\bbD}_{ij}$,  $\breve{\beta}_{ij}$ and $\breve{\gamma}_{ij}$ accordingly. 
		Then \eqref{E|trTD_jTD_j-EtrTD_jTD_j|} becomes
		\begin{align}
			\mathbb{E}\left| \operatorname{tr} \bbT\mathbb{E}_j(\bbD_j^{-1}(z_1)\bbT\breve{\bbD}_j^{-1}(z_2))-\mathbb{E}\operatorname{tr} \bbT\mathbb{E}_j(\bbD_j^{-1}(z_1)\bbT\breve{\bbD}_j^{-1}(z_2))\right| .
		\end{align}
		Due to the martingale decomposition and \eqref{D_j^{-1}-D_kj^{-1}}, we have
		\begin{align}
			\nonumber&  \mathbb{E}\left| \operatorname{tr} \bbT\mathbb{E}_j(\bbD_j^{-1}(z_1)\bbT\breve{\bbD}_j^{-1}(z_2))-\mathbb{E}\operatorname{tr} \bbT\mathbb{E}_j(\bbD_j^{-1}(z_1)\bbT\breve{\bbD}_j^{-1}(z_2)) \right| \\
			\nonumber=&\mathbb{E} \left| \sum_{i<j}(\mathbb{E}_i-\mathbb{E}_{i-1})\operatorname{tr}\bbT\mathbb{E}_j(\bbD_j^{-1}(z_1)\bbT\breve{\bbD}_j^{-1}(z_2))\right| \\
			\nonumber=& \mathbb{E}\left| \sum_{i<j}(\mathbb{E}_i-\mathbb{E}_{i-1})\operatorname{tr}\bbT\mathbb{E}_j[(\bbD_j^{-1}(z_1)-\bbD_{ij}^{-1}(z_1)+\bbD_{ij}^{-1}(z_1)) \bbT(\breve{\bbD}_j^{-1}(z_2)-\breve{\bbD}_{ij}^{-1}(z_2)+\breve{\bbD}_{ij}^{-1}(z_2))]\right| \\
			\nonumber=& \mathbb{E}\left| \sum_{i<j}(\mathbb{E}_i-\mathbb{E}_{i-1})\operatorname{tr}\bbT\mathbb{E}_j[(\bbD_j^{-1}(z_1)-\bbD_{ij}^{-1}(z_1))\bbT(\breve{\bbD}_j^{-1}(z_2)-\breve{\bbD}_{ij}^{-1}(z_2))] \right. \\
			\nonumber &\quad + \sum_{i<j}(\mathbb{E}_i-\mathbb{E}_{i-1})\operatorname{tr}\bbT\mathbb{E}_j[(\bbD_j^{-1}(z_1)-\bbD_{ij}^{-1}(z_1))\bbT\breve{\bbD}_{ij}^{-1}(z_2)] \\
			\nonumber &\quad+\sum_{i<j}(\mathbb{E}_i-\mathbb{E}_{i-1})\operatorname{tr}\bbT\mathbb{E}_j[\bbD_{ij}^{-1}(z_1)\bbT(\breve{\bbD}_j^{-1}(z_2)-\breve{\bbD}_{ij}^{-1}(z_2))]\\
			\nonumber &\quad \left. +\sum_{i<j}(\mathbb{E}_i-\mathbb{E}_{i-1})\operatorname{tr}\bbT\mathbb{E}_j[\bbD_{ij}^{-1}(z_1)\bbT\breve{\bbD}_{ij}^{-1}(z_2)]\right| \\
			\leq &\mathbb{E} \left| 	\sum_{i<j}(\mathbb{E}_i-\mathbb{E}_{i-1})\beta_{ij}(z_1)\breve{\beta}_{ij}(z_2)  \operatorname{tr}[\bbT\bbD_{ij}^{-1}(z_1)\bbr_i\bbr_i^{*}\bbD_{ij}^{-1}(z_1)\bbT\breve{\bbD}_{ij}^{-1}(z_2)\bbr_i\bbr_i^{*}\breve{\bbD}_{ij}^{-1}(z_2)]\right| \label{decomposition of E|trTD_jTD_j-EtrTD_jTD_j|1} \\
			& +\mathbb{E}\left| \sum_{i<j}(\mathbb{E}_i-\mathbb{E}_{i-1})\beta_{ij}(z_1)\operatorname{tr}\bbT\bbD_{ij}^{-1}(z_1)\bbr_i\bbr_i^{*}\bbD_{ij}^{-1}(z_1)\bbT\breve{\bbD}_{ij}^{-1}(z_2) \right|  \label{decomposition of E|trTD_jTD_j-EtrTD_jTD_j|2}\\
			& +\mathbb{E}\left| \sum_{i<j}(\mathbb{E}_i-\mathbb{E}_{i-1})\breve{\beta}_{ij}(z_2)\operatorname{tr}\bbT\bbD_{ij}^{-1}(z_1)\bbT\breve{\bbD}_{ij}^{-1}(z_2)\bbr_i\bbr_i^{*}\breve{\bbD}_{ij}^{-1}(z_2)\right|. \label{decomposition of E|trTD_jTD_j-EtrTD_jTD_j|3}
		\end{align}
		In terms of \eqref{decomposition of E|trTD_jTD_j-EtrTD_jTD_j|2}, due to  Lemmas \ref{Burkholder_1} and \ref{a(v)_quadratic_minus_trace},
		\begin{align}\label{the second term of decomposition of E|trTD_jTD_j-EtrTD_jTD_j|}
			&  \mathbb{E}\left| \sum_{i<j}(\mathbb{E}_i-\mathbb{E}_{i-1})\beta_{ij}(z_1)\operatorname{tr}\bbT\bbD_{ij}^{-1}(z_1)\bbr_i\bbr_i^{*}\bbD_{ij}^{-1}(z_1)\bbT\breve{\bbD}_{ij}^{-1}(z_2)\right| \\
			\nonumber=& \mathbb{E}\left| \sum_{i<j}(\mathbb{E}_i-\mathbb{E}_{i-1})\widetilde{\beta}_{ij}(z_1)\bbr_i^{*}\bbD_{ij}^{-1}(z_1)\bbT\breve{\bbD}_{ij}^{-1}(z_2)\bbT\bbD_{ij}^{-1}(z_1)\bbr_i \right. \\
			\nonumber & \left.+\sum_{i<j}(\mathbb{E}_i-\mathbb{E}_{i-1})(\beta_{ij}(z_1)-\widetilde{\beta}_{ij}(z_1))\bbr_i^{*}\bbD_{ij}^{-1}(z_1)\bbT\breve{\bbD}_{ij}^{-1}(z_2)\bbT\bbD_{ij}^{-1}(z_1)\bbr_i\right| \\
			\nonumber =&\mathbb{E}\left| \sum_{i<j} (\mathbb{E}_i-\mathbb{E}_{i-1})\widetilde{\beta}_{ij}(z_1)(\bbr_i^{*}\bbD_{ij}^{-1}(z_1)\bbT\breve{\bbD}_{ij}^{-1}(z_2)\bbT\bbD_{ij}^{-1}(z_1)\bbr_i \right.\\
			\nonumber&\quad\quad\quad\quad\quad\quad\quad\quad\quad\quad\quad -n^{-1}\operatorname{tr} \bbT\bbD_{ij}^{-1}(z_1)\bbT\breve{\bbD}_{ij}^{-1}(z_2)\bbT\bbD_{ij}^{-1}(z_1))\\
			\nonumber &+\sum_{i<j}(\mathbb{E}_i-\mathbb{E}_{i-1})\beta_{ij}(z_1)\widetilde{\beta}_{ij}(z_1)(\bbr_i^{*}\bbD_{ij}^{-1}(z_1)\bbr_i-n^{-1}\operatorname{tr}\bbT\bbD_{ij}^{-1}(z_1))\\
			\nonumber &\quad\quad\quad\quad\quad\quad\quad\quad\quad\quad \quad\quad\quad\quad\quad \left. (\bbr_i^{*}\bbD_{ij}^{-1}(z_1)\bbT\breve{\bbD}_{ij}^{-1}(z_2)\bbT\bbD_{ij}^{-1}(z_1)\bbr_i)\right| \\
			\nonumber\leq& \left[K\sum_{i<j}\mathbb{E}\lvert \bbr_i^{*}\bbD_{ij}^{-1}(z_1)\bbT\breve{\bbD}_{ij}^{-1}(z_2)\bbT\bbD_{ij}^{-1}(z_1)\bbr_i  -n^{-1}\operatorname{tr} \bbT\bbD_{ij}^{-1}(z_1)\bbT\breve{\bbD}_{ij}^{-1}(z_2)\bbT\bbD_{ij}^{-1}(z_1)\rvert^2\right]^{1/2} \\
			\nonumber&+\left[ K\sum_{i<j}\mathbb{E}\lvert(\bbr_i^{*}\bbD_{ij}^{-1}(z_1)\bbr_i-n^{-1}\operatorname{tr}\bbT\bbD_{ij}^{-1}(z_1)) (\bbr_i^{*}\bbD_{ij}^{-1}(z_1)\bbT\breve{\bbD}_{ij}^{-1}(z_2)\bbT\bbD_{ij}^{-1}(z_1)\bbr_i)\rvert^2 \right]^{1/2}\\
			\nonumber\leq&K\sqrt{n\times n^{-1}}+K\sqrt{n\times n^{-1}} \leq K. 
		\end{align}
		For \eqref{decomposition of E|trTD_jTD_j-EtrTD_jTD_j|3}, by a similar calculation, we can get
		\begin{align}\label{the third term of decomposition of E|trTD_jTD_j-EtrTD_jTD_j|}
			\mathbb{E} \left| \sum_{i<j}(\mathbb{E}_i-\mathbb{E}_{i-1})\breve{\beta}_{ij}(z_2)\operatorname{tr}\bbT\bbD_{ij}^{-1}(z_1)\bbT\breve{\bbD}_{ij}^{-1}(z_2)\bbr_i\bbr_i^{*}\breve{\bbD}_{ij}^{-1}(z_2)\right| \leq K.
		\end{align}
		
		Finally, we only need to estimate \eqref{decomposition of E|trTD_jTD_j-EtrTD_jTD_j|1}, due to Lemmas \ref{Burkholder_1}, \ref{bound moment of b_j and beta_j}, \ref{a(v)_quadratic_minus_trace} and \ref{E|beta_j-b_j|},
		\begin{align}\label{the first term of decomposition of E|trTD_jTD_j-EtrTD_jTD_j|}
			  \eqref{decomposition of E|trTD_jTD_j-EtrTD_jTD_j|1}=&\mathbb{E}\left|\sum_{i<j}(\mathbb{E}_i-\mathbb{E}_{i-1})(\beta_{ij}(z_1)-b_{ij}(z_1))(\breve{\beta}_{ij}(z_2)-b_{ij}(z_2)) \right. \\
			\nonumber & \quad\quad\quad\quad\quad\quad\quad\quad (\bbr_i^{*}\bbD_{ij}^{-1}(z_1)\bbT\breve{\bbD}_{ij}^{-1}(z_2)\bbr_i \bbr_i^{*}\breve{\bbD}_{ij}^{-1}(z_2)\bbT\bbD_{ij}^{-1}(z_1)\bbr_i) \\
			\nonumber &\quad +\sum_{i<j}(\mathbb{E}_i-\mathbb{E}_{i-1})b_{ij}(z_1)(\breve{\beta}_{ij}(z_2)-b_{ij}(z_2)) \\
			\nonumber &\quad\quad\quad\quad\quad\quad\quad\quad\quad (\bbr_i^{*}\bbD_{ij}^{-1}(z_1)\bbT\breve{\bbD}_{ij}^{-1}(z_2)\bbr_i \bbr_i^{*}\breve{\bbD}_{ij}^{-1}(z_2)\bbT\bbD_{ij}^{-1}(z_1)\bbr_i) \\
			\nonumber &\quad +\sum_{i<j}(\mathbb{E}_i-\mathbb{E}_{i-1})(\beta_{ij}(z_1)-b_{ij}(z_1))b_{ij}(z_2) \\
			\nonumber &\quad\quad\quad\quad\quad\quad\quad\quad\quad (\bbr_i^{*}\bbD_{ij}^{-1}(z_1)\bbT\breve{\bbD}_{ij}^{-1}(z_2)\bbr_i \bbr_i^{*}\breve{\bbD}_{ij}^{-1}(z_2)\bbT\bbD_{ij}^{-1}(z_1)\bbr_i) \\
			\nonumber &\quad +\sum_{i<j}b_{ij}(z_1)b_{ij}(z_2)(\mathbb{E}_i-\mathbb{E}_{i-1}) \\ 
			\nonumber &\quad\quad\quad\quad\quad\quad\quad\quad\quad\quad \left. (\bbr_i^{*}\bbD_{ij}^{-1}(z_1)\bbT\breve{\bbD}_{ij}^{-1}(z_2)\bbr_i \bbr_i^{*}\breve{\bbD}_{ij}^{-1}(z_2)\bbT\bbD_{ij}^{-1}(z_1)\bbr_i)\right| \\
			\nonumber \leq & K\sum_{i<j}\mathbb{E}\lvert \gamma_{ij}(z_1)\breve{\gamma}_{ij}(z_2) \beta_{ij}(z_1) \breve{\beta}_{ij}(z_2) \\
			\nonumber &\quad\quad\quad\quad\quad\quad\quad\quad  (\bbr_i^{*}\bbD_{ij}^{-1}(z_1)\bbT\breve{\bbD}_{ij}^{-1}(z_2)\bbr_i \bbr_i^{*}\breve{\bbD}_{ij}^{-1}(z_2)\bbT\bbD_{ij}^{-1}(z_1)\bbr_i)\rvert \\
			\nonumber &+K\left[\sum_{i<j}\mathbb{E}\lvert  \breve{\gamma}_{ij}(z_2) \breve{\beta}_{ij}(z_2) (\bbr_i^{*}\bbD_{ij}^{-1}(z_1)\bbT\breve{\bbD}_{ij}^{-1}(z_2)\bbr_i) \right. \\
			\nonumber &\quad\quad\quad\quad\quad\quad\quad\quad\quad\quad\quad\quad\quad\quad\quad\quad\quad\quad \left. (\bbr_i^{*}\breve{\bbD}_{ij}^{-1}(z_2)\bbT\bbD_{ij}^{-1}(z_1)\bbr_i)\rvert^2\right]^{1/2} \\
			\nonumber &+ K\left[ \sum_{i<j}\mathbb{E}\lvert \gamma_{ij}(z_1)\beta_{ij}(z_1)(\bbr_i^{*}\bbD_{ij}^{-1}(z_1)\bbT\breve{\bbD}_{ij}^{-1}(z_2)\bbr_i) \right. \\
			\nonumber &\quad\quad\quad\quad\quad\quad\quad\quad\quad\quad\quad\quad\quad\quad\quad\quad\quad\quad \left.(\bbr_i^{*}\breve{\bbD}_{ij}^{-1}(z_2)\bbT\bbD_{ij}^{-1}(z_1)\bbr_i)\rvert^2 \right]^{1/2} \\
			\nonumber&+\mathbb{E}\left| \sum_{i<j} (\mathbb{E}_i-\mathbb{E}_{i-1})b_{ij}(z_1)(\bbr_i^{*}\bbD_{ij}^{-1}(z_1)\bbT\breve{\bbD}_{ij}^{-1}(z_2)\bbr_i-n^{-1}\operatorname{tr}\bbT\bbD_{ij}^{-1}(z_1)\bbT\breve{\bbD}_{ij}^{-1}(z_2))  \right.\\
			\nonumber&\quad\quad\quad\quad\quad\quad\quad \times b_{ij}(z_2)( \bbr_i^{*}\breve{\bbD}_{ij}^{-1}(z_2)\bbT\bbD_{ij}^{-1}(z_1)\bbr_i-n^{-1}\operatorname{tr}\bbT\breve{\bbD}_{ij}^{-1}(z_2)\bbT\bbD_{ij}^{-1}(z_1)) \\
			\nonumber &\quad+\sum_{i<j}(\mathbb{E}_i-\mathbb{E}_{i-1})b_{ij}(z_1)(\bbr_i^{*}\bbD_{ij}^{-1}(z_1)\bbT\breve{\bbD}_{ij}^{-1}(z_2)\bbr_i-n^{-1}\operatorname{tr}\bbT\bbD_{ij}^{-1}(z_1)\bbT\breve{\bbD}_{ij}^{-1}(z_2)) \\
			\nonumber&\quad\quad\quad\quad\quad\quad\quad\quad \times b_{ij}(z_2)(n^{-1}\operatorname{tr}\bbT\breve{\bbD}_{ij}^{-1}(z_2)\bbT\bbD_{ij}^{-1}(z_1)) \\
			\nonumber &\quad+\sum_{i<j}(\mathbb{E}_i-\mathbb{E}_{i-1})b_{ij}(z_2)( \bbr_i^{*}\breve{\bbD}_{ij}^{-1}(z_2)\bbT\bbD_{ij}^{-1}(z_1)\bbr_i-n^{-1}\operatorname{tr}\bbT\breve{\bbD}_{ij}^{-1}(z_2)\bbT\bbD_{ij}^{-1}(z_1)) \\
			\nonumber &\quad\quad\quad\quad\quad\quad\quad\quad \left. \times b_{ij}(z_1)(n^{-1}\operatorname{tr}\bbT\bbD_{ij}^{-1}(z_1)\bbT\breve{\bbD}_{ij}^{-1}(z_2))\right| \\
			\nonumber \leq & Kn\times n^{-1}+K\sqrt{n\times n^{-1}}+K\sqrt{n\times n^{-1}} \\
			\nonumber & +K\sum_{i<j}\mathbb{E}\lvert(\bbr_i^{*}\bbD_{ij}^{-1}(z_1)\bbT\breve{\bbD}_{ij}^{-1}(z_2)\bbr_i-n^{-1}\operatorname{tr}\bbT\bbD_{ij}^{-1}(z_1)\bbT\breve{\bbD}_{ij}^{-1}(z_2)) \\
			\nonumber &\quad\quad\quad\quad\quad\quad (\bbr_i^{*}\breve{\bbD}_{ij}^{-1}(z_2)\bbT\bbD_{ij}^{-1}(z_1)\bbr_i-n^{-1}\operatorname{tr}\bbT\breve{\bbD}_{ij}^{-1}(z_2)\bbT\bbD_{ij}^{-1}(z_1))\rvert \\
			\nonumber &+K\left[\sum_{i<j}\mathbb{E}\lvert \bbr_i^{*}\bbD_{ij}^{-1}(z_1)\bbT\breve{\bbD}_{ij}^{-1}(z_2)\bbr_i-n^{-1}\operatorname{tr}\bbT\bbD_{ij}^{-1}(z_1)\bbT\breve{\bbD}_{ij}^{-1}(z_2) \rvert^2\right]^{1/2} \\
			\nonumber&+K\left[\sum_{i<j}\mathbb{E}\lvert \bbr_i^{*}\breve{\bbD}_{ij}^{-1}(z_2)\bbT\bbD_{ij}^{-1}(z_1)\bbr_i-n^{-1}\operatorname{tr}\bbT\breve{\bbD}_{ij}^{-1}(z_2)\bbT\bbD_{ij}^{-1}(z_1)\rvert^2 \right]^{1/2} \\
			\nonumber \leq& K+Kn\times n^{-1}+K\sqrt{n\times n^{-1}}+K\sqrt{n\times n^{-1}} \leq K.
		\end{align}
		From \eqref{the second term of decomposition of E|trTD_jTD_j-EtrTD_jTD_j|}, \eqref{the third term of decomposition of E|trTD_jTD_j-EtrTD_jTD_j|} and \eqref{the first term of decomposition of E|trTD_jTD_j-EtrTD_jTD_j|}, we can get the order of \eqref{E|trTD_jTD_j-EtrTD_jTD_j|}, which is
		\begin{align*}
			\mathbb{E}\lvert \operatorname{tr} \bbT\mathbb{E}_j(\bbD_j^{-1}(z_1))\bbT\mathbb{E}_j(\bbD_j^{-1}(z_2))-\mathbb{E}\operatorname{tr} \bbT\mathbb{E}_j(\bbD_j^{-1}(z_1))\bbT\mathbb{E}_j(\bbD_j^{-1}(z_2)) \rvert \leq K.
		\end{align*}
		Thus, we can obtain the conclusion of this lemma.

	\section{Nonrandom part}\label{Nonrandom part}
	The non-random component consists of two parts. The first part is associated with the variance as detailed in Lemma  \ref{convergence of variance}. The proof of Lemma  \ref{convergence of variance} which will be seen in Section \ref{proof of convergence of variance} is borrowed from  Section 2 of \cite{BaiS04C}. However, in this paper, we derive a more precise rate and extend the analysis to the entire contour, rather than the horizontal parts. Furthermore, compared to the decomposition in \cite{BaiS04C}, our proof introduces novel elements through  equations \eqref{construct an equation 1} and \eqref{construct an equation 2}.

	 The second part concerns the mean term \eqref{the nonrandom part}, addressed in the following lemma.
	 
	 \begin{lemma}\label{convergence of mean}
	 	Under Assumptions \ref{the moment assumption}--\ref{assumption about RG case},
	 	\begin{align}\label{nonrandom part under RG case}
	 		&\left| \int f(x)d\mathbb{E}G_n(x)-\mu_n(f)\right|\\ 
	 		\nonumber=&\left| \int f(x)d\mathbb{E}G_n(x)-\frac{1}{2\pi i}\oint_{\mathcal{C}} f(z) \frac{y_n\int \underline{s}_n^0(z)^3t^2(1+t\underline{s}_n^0(z))^{-3}dH_p(t)}{(1-y_n\int \underline{s}_n^0(z)^2t^2(1+t\underline{s}_n^0(z))^{-2}dH_p(t))^2}dz\right|\\
	 		\nonumber\leq&Kn^{-1}.
	 	\end{align}
	 	Under Assumptions \ref{the moment assumption}--\ref{assumption about test function} and \ref{assumption about CG case}, 
	 	\begin{align}\label{nonrandom part under CG case}
	 		\left|\int f(x)d\mathbb{E}G_n(x)-\mu_n(f)\right|=\left|\int f(x)d\mathbb{E}G_n(x)\right|\leq Kn^{-1}.
	 	\end{align}
	 \end{lemma}
	 \begin{remark}
	 	The proof of this lemma, presented in Section \ref{proof of convergence of mean}, is mainly referred to Section 4 in \cite{BaiS04C}. However, our conclusions are more precise, mainly by Lemma \ref{the lemma of the three product of quadratic minus trace}.
	 \end{remark}
	 
	 The remainder of this section is dedicated to the separate proofs of Lemmas \ref{convergence of variance} and \ref{convergence of mean}.
	
	\subsection{Proof of Lemma $\ref{convergence of variance}$}\label{proof of convergence of variance}

	It is known that $\sigma_n(f)$ is order 1 and  
	$$\lvert\sigma_n^0(f)\rvert=\lvert\sum_{j=1}^{n}\mathbb{E}Y_j^2\rvert\leq\sum_{j=1}^{n}\mathbb{E}\lvert Y_j^2\rvert=O(1).$$
	Since
	\begin{align*}
		\left|\frac{\sqrt{\sigma_n^0(f)}}{\sqrt{\sigma_n(f)}}-1\right|\leq 	\left|\frac{\sigma_n^0(f)}{\sigma_n(f)}-1\right|,
	\end{align*}
	it is sufficient to estimate the order of $\lvert \sigma_n^0(f)-\sigma_n(f)\rvert$. Under the CG case \color{black},
	\begin{align*}
		& \lvert \sigma_n^0(f)-\sigma_n(f)\rvert \\
		=& \left| \sum_{j=1}^{n}\mathbb{E}Y_j^2+\frac{1}{4\pi^2}\oint_{\mathcal{C}_1}\oint_{\mathcal{C}_2}f^{\prime}(z_1)f^{\prime}(z_2)a_n(z_1,z_2)\int_{0}^{1}\frac{1}{1-ta_n(z_1,z_2)}dtdz_1dz_2\right| \\
		=& \frac{1}{4\pi^2}\left|\sum_{j=1}^{n}\mathbb{E}\oint_{\mathcal{C}_{1}}\oint_{\mathcal{C}_{2}}f^{\prime}(z_1)f^{\prime}(z_2)\mathbb{E}_j(\varepsilon_j(z_1)b_j(z_1))\mathbb{E}_j(\varepsilon_j(z_2)b_j(z_2))dz_1dz_2\right. \\
		& \quad\quad \left.-\oint_{\mathcal{C}_1}\oint_{\mathcal{C}_2}f^{\prime}(z_1)f^{\prime}(z_2)a_n(z_1,z_2)\int_{0}^{1}\frac{1}{1-ta_n(z_1,z_2)}dtdz_1dz_2  \right|. 
	\end{align*}
	Under the RG case,
	\begin{align*}
		& \lvert \sigma_n^0(f)-\sigma_n(f)\rvert \\
		=& \frac{1}{4\pi^2}\left|\sum_{j=1}^{n}\mathbb{E}\oint_{\mathcal{C}_{1}}\oint_{\mathcal{C}_{2}}f^{\prime}(z_1)f^{\prime}(z_2)\mathbb{E}_j(\varepsilon_j(z_1)b_j(z_1))\mathbb{E}_j(\varepsilon_j(z_2)b_j(z_2))dz_1dz_2 \right. \\
		& \quad\quad \left.-2\oint_{\mathcal{C}_1}\oint_{\mathcal{C}_2}f^{\prime}(z_1)f^{\prime}(z_2)a_n(z_1,z_2)\int_{0}^{1}\frac{1}{1-ta_n(z_1,z_2)}dtdz_1dz_2  \right|.
	\end{align*} 
	The contours $\mathcal{C}_1$ and $\mathcal{C}_2$ are shown in the proof of  Lemma \ref{sumE|Y_j^2-EY_j^2|} in Fig \ref{contour 1 and 2}.
	
	Throughout the following, all bounds, including $O(\cdot)$ and $o(\cdot)$ expressions hold uniformly for $z\in\mathcal{C}$. The positive constant $K$ is  independent on $z$.
	 Next, we need only to estimate, for $z_1\in\mathcal{C}_2$, $z_2\in\mathcal{C}_1$, under the CG case,
	\begin{align*}
		\left|\sum_{j=1}^{n}b_j(z_1)b_j(z_2)\mathbb{E}\left[\mathbb{E}_j(\varepsilon_j(z_1))\mathbb{E}_j(\varepsilon_j(z_2))\right]-a_n(z_1,z_2)\int_{0}^{1}\frac{1}{1-ta_n(z_1,z_2)}dt\right|,
	\end{align*}
	under the RG case,
	\begin{align*}
		\left|\sum_{j=1}^{n}b_j(z_1)b_j(z_2)\mathbb{E}\left[\mathbb{E}_j(\varepsilon_j(z_1))\mathbb{E}_j(\varepsilon_j(z_2))\right]-2a_n(z_1,z_2)\int_{0}^{1}\frac{1}{1-ta_n(z_1,z_2)}dt\right|.
	\end{align*}
	
	From \eqref{beta_j}, Lemmas \ref{bound moment of b_j and beta_j}, \ref{a(v)_quadratic_minus_trace}, \ref{E|trTDj-1-EtrTDj-1|},  we have
	\begin{align}\label{|b_j-Ebeta_j|}
		&\lvert b_j(z_i)-\mathbb{E}\beta_{j}(z_i)\rvert =\lvert \mathbb{E}(b_j(z_i)-\beta_{j}(z_i))\rvert  \\
		\nonumber\quad =&\lvert b_j(z_i)\mathbb{E}[ \beta_{j}(z_i)(\bbr_j^{*}\bbD_j^{-1}(z_i)\bbr_j-n^{-1}\mathbb{E}\operatorname{tr}\bbT\bbD_j^{-1}(z_i))]\rvert \\
		\nonumber \leq& \lvert b_j(z_i)\mathbb{E}[\widetilde{\beta}_{j}(z_i)(\bbr_j^{*}\bbD_j^{-1}(z_i)\bbr_j-n^{-1}\mathbb{E}\operatorname{tr}\bbT\bbD_j^{-1}(z_i))]\rvert \\
		\nonumber & +\lvert b_j(z_i)\mathbb{E}[\beta_{j}(z_i)\widetilde{\beta}_j(z_i)\varepsilon_j(z_i)(\bbr_j^{*}\bbD_j^{-1}(z_i)\bbr_j-n^{-1}\mathbb{E}\operatorname{tr}\bbT\bbD_j^{-1}(z_i))]\rvert \\
		\nonumber  \leq& \lvert b_j(z_i)\mathbb{E} \widetilde{\beta}_{j}(z_i)\varepsilon_j(z_i) \rvert + \mathbb{E}\lvert b_j(z_i) \beta_{j}(z_i)\widetilde{\beta}_j(z_i)\varepsilon_j^2(z_i) \rvert \\
		\nonumber  &+ \mathbb{E}\lvert\widetilde{\beta}_{j}(z_i)b_j(z_i)(n^{-1}\operatorname{tr}\bbT\bbD_j^{-1}(z_i)-n^{-1}\mathbb{E}\operatorname{tr}\bbT\bbD_j^{-1}(z_i)) \rvert \\
		\nonumber &+\mathbb{E}\lvert b_j(z_i) \beta_{j}(z_i)\widetilde{\beta}_j(z_i)\varepsilon_j(z_i)(n^{-1}\operatorname{tr}\bbT\bbD_j^{-1}(z_i)-n^{-1}\mathbb{E}\operatorname{tr}\bbT\bbD_j^{-1}(z_i)) \rvert  \\
		\nonumber \leq& \sqrt{\mathbb{E}\lvert b_j(z_i)\beta_j(z_i)\widetilde{\beta}_j(z_i)\varepsilon_j(z_i)\rvert^2}\sqrt{\mathbb{E}\lvert n^{-1}\operatorname{tr}\bbT\bbD_j^{-1}(z_i)-n^{-1}\mathbb{E}\operatorname{tr}\bbT\bbD_j^{-1}(z_i) \rvert^2} \\
		\nonumber &+Kn^{-1} \\
		\nonumber \leq& Kn^{-1/2}n^{-1} +Kn^{-1}\leq Kn^{-1}
	\end{align}
	Therefore, from Lemma \ref{|b_j-b|},
	\begin{align}\label{|b-Ebeta_j|}
		&\lvert b(z_i)-\mathbb{E}\beta_j(z_i)\rvert=\lvert b(z_i)-b_j(z_i)+b_j(z_i)-\mathbb{E}\beta_j(z_i)\rvert \\
		\nonumber \leq& \lvert b(z_i)-b_j(z_i)\rvert+\lvert b_j(z_i)-\mathbb{E}\beta_j(z_i)\rvert \leq Kn^{-1}
	\end{align}
	From the formula 
	\begin{align*}
		\underline{s}_n(z_i)=-\frac{1}{z_in}\sum_{j=1}^{n}\beta_{j}(z_i)
	\end{align*}
	(see (2.2) in \cite{Silverstein95S}), we get $(1/n)\sum_{j=1}^{n}\mathbb{E}\beta_j(z_i)=-z_i\mathbb{E}\underline{s}_n(z_i)$.
	Due to \eqref{|b-Ebeta_j|} and Lemma  \ref{|Es_n-s_n^0|}, we get
	\begin{align}\label{b(z)+zs_n^0(z)}
		& \lvert b(z_i)+z_i\underline{s}_n^0(z_i)\rvert=\lvert b(z_i)+z_i\mathbb{E}\underline{s}_n(z_i)+z_i\underline{s}_n^0(z_i)-z_i\mathbb{E}\underline{s}_n(z_i)\rvert \\
		\nonumber\leq& \lvert b(z_i)+z_i\mathbb{E}\underline{s}_n(z_i)\rvert+\lvert z_i\mathbb{E}\underline{s}_n(z_i)-z_i\underline{s}_n^0(z_i)\rvert \\
		\nonumber\leq& \lvert\frac{1}{n}\sum_{j=1}^{n}(b(z_i)-\mathbb{E}\beta_j(z_i))\rvert+Kn^{-1}\leq Kn^{-1}.
	\end{align}
	Due to \eqref{b(z)+zs_n^0(z)} and Lemmas \ref{E|beta_j-b_j|} and \ref{|b_j-b|}, we conclude that 
	\begin{align}\label{b_j(z)+zs_n^0(z)}
		& \lvert b_j(z_i)+z_i\underline{s}_n^0(z_i)\rvert=\lvert b_j(z_i)-b(z_i)+b(z_i)+z_i\underline{s}_n^0(z_i)\rvert \\ 
		\nonumber\leq& \lvert b_j(z_i)-b(z_i)\vert + \lvert b(z_i)+z_i\underline{s}_n^0(z_i)\rvert \leq Kn^{-1},
	\end{align}
	and
	\begin{align}\label{E|tilde beta_j(z)+zs_n^0(z)|}
		& \mathbb{E}\lvert\widetilde{\beta}_j(z_i)+z_i\underline{s}_n^0(z_i)\rvert=\mathbb{E}\lvert\widetilde{\beta}_j(z_i)-b_j(z_i)+b_j(z_i)+z_i\underline{s}_n^0(z_i)\rvert \\
		\nonumber\leq& \mathbb{E}\lvert\widetilde{\beta}_j(z_i)-b_j(z_i)\rvert+\lvert b_j(z_i)+z_i\underline{s}_n^0(z_i)\rvert \leq Kn^{-1}.
	\end{align}
	Using \eqref{b_j(z)+zs_n^0(z)} and Lemma \ref{Q_QMTr}, we have
	\begin{align*}
		&\lvert (b_j(z_1)+z_1\underline{s}_n^0(z_1))(b_j(z_2)+z_2\underline{s}_n^0(z_2))\mathbb{E}[\mathbb{E}_j(\varepsilon_j(z_1))\mathbb{E}_j(\varepsilon_j(z_2))]\rvert \\
		\nonumber\leq& Kn^{-2} (\mathbb{E}\lvert \mathbb{E}_j(\varepsilon_j(z_1))\rvert^2)^{1/2}(\mathbb{E}\lvert \mathbb{E}_j(\varepsilon_j(z_1))\rvert^2)^{1/2} \leq Kn^{-3}.
	\end{align*}
	Similarly
	\begin{align*}
		\lvert (b_j(z_1)+z_1\underline{s}_n^0(z_1))(-z_2\underline{s}_n^0(z_2))\mathbb{E}(\mathbb{E}_j\varepsilon_j(z_1)\mathbb{E}_j(\varepsilon_j(z_2)))\rvert\leq Kn^{-2}.
	\end{align*}
	This implies
	\begin{align}\label{b_j(z_1)b_j(z_2)-z_1z_2s_n^0(z_1)s_n^0(z_2)}
		&\left| \sum_{j=1}^{n}b_j(z_1)b_j(z_2)\mathbb{E}[\mathbb{E}_j(\varepsilon_j(z_1))\mathbb{E}_j(\varepsilon_j(z_2))] \right. \\
		\nonumber &~~~~~~~~ \left. -z_1\underline{s}_n^0(z_1)z_2\underline{s}_n^0(z_2)\sum_{j=1}^{n}\mathbb{E}[\mathbb{E}_j(\varepsilon_j(z_1))\mathbb{E}_j(\varepsilon_j(z_2))]\right| \\
		\nonumber&=O(n^{-1}).
	\end{align}
	Thus, our goal changes to 
	\begin{align}
		\sum_{j=1}^{n}z_1z_2\underline{s}_n^0(z_1)\underline{s}_n^0(z_2)\mathbb{E} [\mathbb{E}_j(\varepsilon_j(z_1))\mathbb{E}_j(\varepsilon_j(z_2))].
	\end{align}
	Due to \eqref{product of two quadratic minus trace of two matrices},
	\begin{align*}
		&  \sum_{j=1}^{n}z_1z_2\underline{s}_n^0(z_1)\underline{s}_n^0(z_2)\mathbb{E} [\mathbb{E}_j(\varepsilon_j(z_1))\mathbb{E}_j(\varepsilon_j(z_2))] \\
		=& \sum_{j=1}^{n}z_1z_2\underline{s}_n^0(z_1)\underline{s}_n^0(z_2)\mathbb{E}[(\bbr_j^{*}\mathbb{E}_j(\bbD_j^{-1}(z_1))\bbr_j-n^{-1}\operatorname{tr}\bbT\mathbb{E}_j(\bbD_j^{-1}(z_1)))  \\
		&\quad\quad\quad\quad\quad\quad\quad\quad\quad\quad\quad (\bbr_j^{*}\mathbb{E}_j(\bbD_j^{-1}(z_2))\bbr_j-n^{-1}\operatorname{tr}\bbT\mathbb{E}_j(\bbD_j^{-1}(z_2)))] \\
		=& n^{-2}\sum_{j=1}^{n}z_1z_2\underline{s}_n^0(z_1)\underline{s}_n^0(z_2)\left[  (\mathbb{E}\lvert x_{1j}\rvert^4-\lvert\mathbb{E} x_{1j}^2\rvert^2-2) \mathbb{E}\operatorname{tr}[ (\bbT\mathbb{E}_j\bbD_j^{-1}(z_1))\circ(\bbT\mathbb{E}_j\bbD_j^{-1}(z_2))]\right. \\
		& ~~~ \left. +\lvert\mathbb{E}x_{1j}^2\rvert^2 \mathbb{E}\operatorname{tr} \bbT\mathbb{E}_j(\bbD_j^{-1}(z_1))(\bbT\mathbb{E}_j(\bbD_j^{-1}(z_2)))^{T}+\mathbb{E} \operatorname{tr} \bbT\mathbb{E}_j(\bbD_j^{-1}(z_1))\bbT\mathbb{E}_j(\bbD_j^{-1}(z_2)) \right]
	\end{align*}
	Under the RG case $\mathbb{E}\lvert x_{ij}\rvert^4=3+o(n^{-3/2}\eta_n^{4})$, we have
	\begin{align*}
		&  \sum_{j=1}^{n}z_1z_2\underline{s}_n^0(z_1)\underline{s}_n^0(z_2)\mathbb{E}\left[\mathbb{E}_j(\varepsilon_j(z_1))\mathbb{E}_j(\varepsilon_j(z_2))\right] \\
		=&   2 n^{-2}\sum_{j=1}^{n}z_1z_2\underline{s}_n^0(z_1)\underline{s}_n^0(z_2)\mathbb{E} \operatorname{tr} \bbT\mathbb{E}_j(\bbD_j^{-1}(z_1))\bbT\mathbb{E}_j(\bbD_j^{-1}(z_2))+ o(n^{-3/2}\eta_n^4).
	\end{align*}
	Under CG case $\mathbb{E} x_{ij}^2=o(n^{-2}\eta_n^2)$ and $\mathbb{E}\lvert x_{ij}\rvert^4=2+o(n^{-3/2}\eta_n^4)$, we have
	\begin{align*}
		& \sum_{j=1}^{n}z_1z_2\underline{s}_n^0(z_1)\underline{s}_n^0(z_2)\mathbb{E}\left[\mathbb{E}_j(\varepsilon_j(z_1))\mathbb{E}_j(\varepsilon_j(z_2))\right] \\
		=& n^{-2}\sum_{j=1}^{n}z_1z_2\underline{s}_n^0(z_1)\underline{s}_n^0(z_2)\mathbb{E} \operatorname{tr} \bbT\mathbb{E}_j(\bbD_j^{-1}(z_1))\bbT\mathbb{E}_j(\bbD_j^{-1}(z_2))+o(n^{-3/2}\eta_n^4).
	\end{align*}
	Thus, we only need consider, under the CG case,
	\begin{align}\label{EtrTE_jD_jTE_jD_j}
		&\left|\frac{1}{n^2}\sum_{j=1}^{n}z_1z_2\underline{s}_n^0(z_1)\underline{s}_n^0(z_2)\mathbb{E} \operatorname{tr} \bbT\mathbb{E}_j(\bbD_j^{-1}(z_1))\bbT\mathbb{E}_j(\bbD_j^{-1}(z_2)) \right.\\
		\nonumber&\left.\quad-a_n(z_1,z_2)\int_{0}^{1}\frac{1}{1-ta_n(z_1,z_2)}dt\right|.
	\end{align}
	
	Using \eqref{D_j^{-1}-D_kj^{-1}} and Lemma \ref{E|trTDj-1-EtrTDj-1|}, we get  
	\begin{align}
		& \mathbb{E}\lvert n^{-1}\mathbb{E}_j\operatorname{tr}\bbT\bbD_j^{-1}(z_i)-n^{-1}\operatorname{tr}\bbT\bbD_j^{-1}(z_i)\rvert  =\mathbb{E}\lvert n^{-1}\mathbb{E}_j\operatorname{tr}\bbT\bbD_j^{-1}(z_i)-(\widetilde{\beta}_j(z_i)^{-1}-1)\rvert \\
		\nonumber=&n^{-1}\mathbb{E}\left| \sum_{k=j+1}^{n}(\mathbb{E}_k-\mathbb{E}_{k-1})\operatorname{tr}\bbT\bbD_j^{-1}(z_i) \right| =n^{-1}\mathbb{E}\left| \sum_{k=j+1}^{n}(\mathbb{E}_k-\mathbb{E}_{k-1})\operatorname{tr}\bbT(\bbD_j^{-1}(z_i)-\bbD_{kj}^{-1}(z_i)) \right| \\
		\nonumber\leq&  Kn^{-1}, 
	\end{align}
	so from \eqref{E|tilde beta_j(z)+zs_n^0(z)|}, we get
	\begin{align}\label{E|n^-1E_jtrTD_j^-1-(-(z_is_n^0)^-1-1)|}
		\mathbb{E}\lvert n^{-1}\mathbb{E}_j\operatorname{tr}\bbT\bbD_j^{-1}(z_i)-(-(z_i\underline{s}_n^0(z_i))^{-1}-1)\rvert\leq Kn^{-1}.
	\end{align}
	Recall the definition $\bbD_{ij}(z)=\bbD(z)-\bbr_i\bbr_i^{*}-\bbr_j\bbr_j^{*}$,$\beta_{ij}(z)=(1+\bbr_i^{*}\bbD_{ij}^{-1}(z)\bbr_i)^{-1}$ and $b_{ij}(z)=(1+n^{-1}\mathbb{E}\operatorname{tr}\bbT\bbD_{ij}^{-1}(z))^{-1}$.
	Similar to Lemma \ref{|b_j-b|} we have
	\begin{align}\label{b_j-b_ij}
		\lvert b_j(z_i)-b_{ij}(z_i)\rvert\leq Kn^{-1}.
	\end{align}
	Using Lemma \ref{Q_QMTr} and a similar argument resulting in Lemma \ref{E|beta_j-b_j|}, we find
	\begin{align}\label{E|beta_ij-b_ij|}
		\mathbb{E}\lvert \beta_{ij}(z_i)-b_{ij}(z_i)\rvert^2\leq Kn^{-1}.
	\end{align}
	Recall the definition of $\breve{\bbD}_j$, $\breve{\bbD}_{ij}$ and $\breve{\beta}_{ij}$ given in the proof of Lemma \ref{sumE|Y_j^2-EY_j^2|}, where the random vectors $\bbr_{j+1},\dots,\bbr_n$ in matrix $\bbD_j$ or $\bbD_{ij}$ are replaced by their i.i.d. copies $\breve{\bbr}_{j+1},\dots,\breve{\bbr}_n$.  Then part of \eqref{EtrTE_jD_jTE_jD_j} becomes
	\begin{align}
		\frac{1}{n^2}\sum_{j=1}^{n}z_1z_2\underline{s}_n^0(z_1)\underline{s}_n^0(z_2)\mathbb{E}\operatorname{tr} \bbT\mathbb{E}_j(\bbD_j^{-1}(z_1)\bbT\breve{\bbD}_j^{-1}(z_2)).
	\end{align}
	We write
	\begin{align}\label{construct an equation 1}
		&\mathbb{E}\left[\frac{z_1}{n}\mathbb{E}_j \operatorname{tr}\bbT\bbD_j^{-1}(z_1)-\frac{z_2}{n}\mathbb{E}_j\operatorname{tr} \bbT\breve{\bbD}_j^{-1}(z_2)\right] =z_2-z_1+(\underline{s}_n^0(z_2))^{-1}-(\underline{s}_n^0(z_1))^{-1}+A_{1,n},
	\end{align}
	where from \eqref{E|n^-1E_jtrTD_j^-1-(-(z_is_n^0)^-1-1)|} we have $\lvert A_{1,n}\rvert\leq Kn^{-1}$. On the other hand, using \eqref{r_i^*(C+r_ir_i^*)^-1}
	\begin{align}\label{construct an equation 2}
		& \mathbb{E}\left[\frac{z_1}{n}\mathbb{E}_j\operatorname{tr}\bbT\bbD_j^{-1}(z_1)-\frac{z_2}{n}\mathbb{E}_j\operatorname{tr} \bbT\breve{\bbD}_j^{-1}(z_2)\right] \\
		\nonumber=&\frac{1}{n}\mathbb{E}[\mathbb{E}_j\operatorname{tr} \bbT \bbD_j^{-1}(z_1)(z_1\breve{\bbD}_j(z_2)-z_2\bbD_j(z_1))\breve{\bbD}_j^{-1}(z_2)] \\
		\nonumber=&\frac{1}{n}\sum_{i=1}^{j-1}(z_1-z_2)\mathbb{E}[\mathbb{E}_j \bbr_i^{*}\breve{\bbD}_{ij}^{-1}(z_2)\bbT\bbD_{ij}^{-1}(z_1)\bbr_i\beta_{ij}(z_1)\breve{\beta}_{ij}(z_2)] \\
		\nonumber&+ \frac{1}{n}\sum_{i=j+1}^{n}z_1\mathbb{E} [\mathbb{E}_j\breve{\bbr}_i^{*}\breve{\bbD}_{ij}^{-1}(z_2)\bbT\bbD_j^{-1}(z_1)\breve{\bbr}_i\breve{\beta}_{ij}(z_2)] \\
		\nonumber&-\frac{1}{n}\sum_{i=j+1}^{n} z_2\mathbb{E}[\mathbb{E}_j\bbr_i^{*}\breve{\bbD}_j^{-1}(z_2)\bbT\bbD_{ij}^{-1}(z_1)\bbr_i\beta_{ij}(z_1)] \\
		\nonumber=&\frac{1}{n^2}\sum_{i=1}^{j-1}(z_1-z_2)b_{ij}(z_1)b_{ij}(z_2)\mathbb{E}[\operatorname{tr}\mathbb{E}_j\bbT\breve{\bbD}_{ij}^{-1}(z_2)\bbT\bbD_{ij}^{-1}(z_1)] \\
		\nonumber&+\frac{1}{n^2}\sum_{i=j+1}^{n}z_1b_{ij}(z_2)\mathbb{E}[\operatorname{tr}\mathbb{E}_j\bbT\breve{\bbD}_{ij}^{-1}(z_2)\bbT\bbD_{j}^{-1}(z_1)] \\
		\nonumber&-\frac{1}{n^2}\sum_{i=j+1}^{n}z_2b_{ij}(z_1)\mathbb{E}[\operatorname{tr}\mathbb{E}_j\bbT\breve{\bbD}_j^{-1}(z_2)\bbT\bbD_{ij}^{-1}(z_1)]+A_{2,n} \\
		\nonumber=&\frac{1}{n^2}\sum_{i=1}^{j-1}(z_1-z_2)z_1z_2\underline{s}_n^0(z_1)\underline{s}_n^0(z_2)\mathbb{E}[\operatorname{tr}\mathbb{E}_j\bbT\breve{\bbD}_{ij}^{-1}(z_2)\bbT\bbD_{ij}^{-1}(z_1)] \\
		\nonumber &-\frac{1}{n^2}\sum_{i=j+1}^{n}z_1z_2\underline{s}_n^0(z_2)\mathbb{E}[\operatorname{tr}\mathbb{E}_j\bbT\breve{\bbD}_{ij}^{-1}(z_2)\bbT\bbD_{j}^{-1}(z_1)] \\
		\nonumber&+\frac{1}{n^2}\sum_{i=j+1}^{n}z_1z_2\underline{s}_n^0(z_1)\mathbb{E}[\operatorname{tr}\mathbb{E}_j\bbT\breve{\bbD}_j^{-1}(z_2)\bbT\bbD_{ij}^{-1}(z_1)]+A_{3,n}.
	\end{align}
	From \eqref{beta_ij decomposition},  \eqref{E|beta_ij-b_ij|}, Lemmas \ref{bound moment of b_j and beta_j} and  \ref{a(v)_quadratic_minus_trace}, we have $\lvert A_{2,n}\rvert\leq Kn^{-1}$. Further, from \eqref{b_j(z)+zs_n^0(z)} and \eqref{b_j-b_ij}, we have $\lvert A_{3.n}\rvert\leq Kn^{-1}$.
	Using \eqref{D_j^{-1}-D_kj^{-1}}, we find 
	\begin{align}\label{construct an equation 3}
		\eqref{construct an equation 2} & =\frac{j-1}{n^2}(z_1-z_2)z_1 z_2\underline{s}_n^0(z_1)\underline{s}_n^0(z_2)\mathbb{E}[\operatorname{tr} \bbT\mathbb{E}_j(\bbD_j^{-1}(z_1))\bbT\mathbb{E}_j(\breve{\bbD}_j^{-1}(z_2))] \\
		\nonumber&\quad - \frac{n-j+1}{n^2}z_1 z_2\underline{s}_n^0(z_2)\mathbb{E}[\operatorname{tr}\bbT\mathbb{E}_j(\bbD_j^{-1}(z_1))\bbT\mathbb{E}_j(\breve{\bbD}_j^{-1}(z_2))] \\
		\nonumber&\quad +\frac{n-j+1}{n^2}z_1 z_2\underline{s}_n^0(z_1)\mathbb{E}[\operatorname{tr}\bbT\mathbb{E}_j(\bbD_j^{-1}(z_1))\bbT\mathbb{E}_j(\breve{\bbD}_j^{-1}(z_2))]+A_{4,n},
	\end{align}
	where $\lvert A_{4,n}\rvert\leq Kn^{-1}$.
	Combining \eqref{construct an equation 1} and \eqref{construct an equation 3},  we have
	\begin{align}
		& n^{-1}z_1z_2\underline{s}_n^0(z_1)\underline{s}_n^0(z_2)\mathbb{E}[\operatorname{tr}\bbT\mathbb{E}_j(\bbD_j^{-1}(z_1))\bbT\mathbb{E}_j(\breve{\bbD}_j^{-1}(z_2))] \\
		\nonumber=&\frac{z_1 z_2 \underline{s}_n^0(z_2) \underline{s}_n^0(z_1)(z_2-z_1+\underline{s}_n^{0-1}(z_2)-\underline{s}_n^{0-1}(z_1))}{\frac{j-1}{n}(z_1-z_2) z_1 z_2 \underline{s}_n^0(z_1) \underline{s}_n^0(z_2)+\frac{n-(j-1)}{n} z_1 z_2(\underline{s}_n^0(z_1)-\underline{s}_n^0(z_2))}  +A_{5,n} \\
		\nonumber =& \frac{a_n(z_1, z_2)}{1-\frac{j-1}{n} a_n(z_1, z_2)}+A_{5, n},
	\end{align}
	where 
	$$a_n(z_1,z_2)=1+\frac{\underline{s}_n^0(z_1)\underline{s}_n^0(z_2)(z_1-z_2)}{\underline{s}_n^0(z_2)-\underline{s}_n^0(z_1)}.$$
	From \eqref{equation of s^0}, we can get the inverse of $\underline{s}^0$, which is 
	\begin{align}\label{inverse of s^0}
		z(\underline{s}^0)=-\frac{1}{\underline{s}^0}+y\int\frac{t}{1+t\underline{s}^0}dH(t).
	\end{align}
	and the inverse of $\underline{s}_n^0$, denoted $z_n^0$, is given by
	\begin{align}\label{inverse of s_n^0}
		z_n^0(\underline{s}_n^0)=-\frac{1}{\underline{s}_n^0}+y_n\int\frac{tdH_p(t)}{1+t\underline{s}_n^0}.
	\end{align}
	Notice that
	\begin{align}\label{Im of s_n^0}
		\Im\underline{s}_n^0=\frac{v+\underline{s}_n^0y_n\int\frac{t^2dH_p(t)}{\lvert 1+t\underline{s}_n^0\rvert^2}}{\left|-z+y_n\int\frac{tdH_p(t)}{1+t\underline{s}_n^0}\right|^2},
	\end{align}
	where $v=\Im z$.

	Because of the inverse relation between $z$ and $\underline{s}_n^0(z)$, namely \eqref{inverse of s_n^0}, we find
	\begin{align}
		a_n(z_1,z_2)=y_n\underline{s}_n^0(z_1)\underline{s}_n^0(z_2)\int\frac{t^2dH_p(t)}{(1+t\underline{s}_n^0(z_1))(1+t\underline{s}_n^0(z_2))}.
	\end{align}
	Using the Cauchy-Schwarz's inequality and \eqref{Im of s_n^0},
	\begin{align}\label{|a_n(z_1,z_2)|}
		& \lvert a_n(z_1,z_2)\rvert \\
		\nonumber =& \left| y_n \frac{\int \frac{t^2 d H_p(t)}{\left(1+t \underline{s}_n^0\left(z_1\right)\right)\left(1+t \underline{s}_n^0\left(z_2\right)\right)}}{(-z_1+y_n \int \frac{t d H_p(t)}{1+t \underline{s}_n^0\left(z_1\right)})(-z_2+y_n \int \frac{t d H_p(t)}{1+t\underline{s}_n^0\left(z_2\right)})} \right|  \\
		\nonumber \leq& \left(y_n \frac{\int \frac{t^2 d H_p(t)}{\left|1+t \underline{s}_n^0\left(z_1\right)\right|^2}}{\left|-z_1+y_n \int \frac{t d H_p(t)}{1+t \underline{s}_n^0\left(z_1\right)}\right|^2}\right)^{1 / 2}\left(y_n \frac{\int \frac{t^2 d H_p(t)}{\left|1+t \underline{s}_n^0\left(z_2\right)\right|^2}}{\left|-z_2+y_n \int \frac{t d H_p(t)}{1+t \underline{s}_n^0\left(z_2\right)}\right|^2}\right)^{1/2} \\
		\nonumber =& \left(\frac{\Im \underline{s}_n^0\left(z_1\right) y_n \int \frac{t^2 d H_p(t)}{\left|1+t \underline{s}_n^0\left(z_1\right)\right|^2}}{\Im z_1+\Im \underline{s}_n^0\left(z_1\right) y_n \int \frac{t^2 d H_p(t)}{\left|1+t \underline{s}_n^0\left(z_1\right)\right|^2}}\right)^{1/2} \left(\frac{\Im \underline{s}_n^0\left(z_2\right) y_n \int \frac{t^2 d H_p(t)}{\left|1+t \underline{s}_n^0\left(z_2\right)\right|^2}}{\Im z_2+\Im \underline{s}_n^0\left(z_2\right) y_n \int \frac{t^2 d H_p(t)}{\left|1+t \underline{s}_n^0\left(z_2\right)\right|^2}}\right)^{1/2}, 
	\end{align}
	When $\Im z_i\neq 0$, according to Lemma 2.3 in \cite{Silverstein95S}, from
	\begin{align*}
		& \Im \underline{s}_n^0(z_i)y_n\int \frac{t^2 d H_p(t)}{\lvert1+t \underline{s}_n^0(z_2)\rvert^2} = y_n\Im\left(\int\frac{tdH_p(t)}{1+t\underline{s}_n^0}\right) 
		\leq y_n \lVert \bbT(\bbI+\bbT\underline{s}_n^0)^{-1}\rVert\leq 4y_n/\Im z_i,
	\end{align*}
	we conclude that for all $n$ there exists an $\varrho>0$ such that $\lvert a_n(z_1,z_2)\rvert\leq1-\varrho$. 
	From
	\begin{align*}
		\lim_{\Im z_i\to 0} \frac{\Im\underline{s}_n^0(z_i)}{\Im z_i}=\int \frac{dF^{y_n,H_p}(\lambda)}{(\lambda-\Re z_i)^2}>0,
	\end{align*}
	we have that for all $n$ there exists an $\varrho>0$ such that $\lvert a_n(z_1,z_2)\rvert\leq1-\varrho$ when $\Im z_i=0$. 
	Thus, for $z_1\in\mathcal{C}_2$ and $z_2\in\mathcal{C}_1$, we have that for all $n$ there exists an $\varrho>0$ such that $\lvert a_n(z_1,z_2)\rvert\leq1-\varrho$. From \eqref{inverse of s_n^0}, we must have $\inf_{n}\lvert\underline{s}_n^0(z_1)-\underline{s}_n^0(z_2)\rvert>0$ for $z_1\in\mathcal{C}_2$ and $z_2\in\mathcal{C}_1$. As a result, we find 
	\begin{align}\label{a_n(z_1,z_2)+A_5,n}
		\eqref{EtrTE_jD_jTE_jD_j}=&\left| a_n(z_1,z_2)\frac{1}{n}\sum_{j=1}^{n}\frac{1}{1-\frac{j-1}{n}a_n(z_1,z_2)}+ \frac{1}{n}\sum_{j=1}^{n}A_{5,n} -a_n(z_1,z_2)\int_{0}^{1}\frac{1}{1-ta_n(z_1,z_2)}dt\right|
	\end{align}
	where $\lvert n^{-1}\sum_{j=1}^{n}A_{5,n}\rvert\leq Kn^{-1}$.
	 Let
	\begin{align*}
		L_n(t)=\frac{a_n(z_1,z_2)}{1-ta_n(z_1,z_2)},
	\end{align*}
	and when $u\in [\frac{j-1}{n},\frac{j}{n})$, $j=1,\dots,n$,
	\begin{align*}
		\widetilde{L}_n(u)=\frac{a_n(z_1,z_2)}{1-\frac{j-1}{n}a_n(z_1,z_2)}.
	\end{align*}
	We have, due to $\sup_{z_1\in\mathcal{C}_2,z_2\in\mathcal{C}_1}\lvert a_n(z_1,z_2)\rvert<1$,
	\begin{align}\label{a_n(z_1,z_2)-integral}
		&\sup_{z_1\in\mathcal{C}_2,z_2\in\mathcal{C}_1}\left\lvert a_n(z_1,z_2)\frac{1}{n}\sum_{j=1}^{n}\frac{1}{1-\frac{j-1}{n}a_n(z_1,z_2)}-a_n(z_1,z_2)\int_{0}^{1}\frac{1}{1-ta_n(z_1,z_2)}dt\right\rvert \\
		\nonumber =&\sup_{z_1\in\mathcal{C}_2,z_2\in\mathcal{C}_1}\left\lvert \int_{0}^{1}[\widetilde{L}_n(t)-L_n(t)] dt\right\rvert\leq \int_{0}^{1}\sup_{z_1\in\mathcal{C}_2,z_2\in\mathcal{C}_1}\lvert \widetilde{L}_n(t)-L_n(t)\rvert dt \leq K n^{-1}.
	\end{align}
	From \eqref{a_n(z_1,z_2)+A_5,n} and \eqref{a_n(z_1,z_2)-integral}, we have
	\begin{align*}
		& \left| \frac{1}{n^2}\sum_{j=1}^{n}z_1z_2\underline{s}_n^0(z_1)\underline{s}_n^0(z_2)\mathbb{E} \operatorname{tr} \bbT\mathbb{E}_j(\bbD_j^{-1}(z_1))\bbT\mathbb{E}_j(\bbD_j^{-1}(z_2)) -a_n(z_1,z_2)\int_{0}^{1}\frac{1}{1-ta_n(z_1,z_2)}dt \right| \\
		&\leq Kn^{-1}.
	\end{align*}
	From \eqref{b_j(z_1)b_j(z_2)-z_1z_2s_n^0(z_1)s_n^0(z_2)}, under the CG case, for $z_1\in\mathcal{C}_2$ and $z_2\in\mathcal{C}_1$,
	\begin{align*}
		& \left| \sum_{j=1}^{n}b_j(z_1)b_j(z_2)\mathbb{E}\left[\mathbb{E}_j(\varepsilon_j(z_1))\mathbb{E}_j(\varepsilon_j(z_2))\right] -a_n(z_1,z_2)\int_{0}^{1}\frac{1}{1-ta_n(z_1,z_2)}dt\right| \\
		&\leq Kn^{-1}.
	\end{align*}
	Thus we can get
	\begin{align*}
		& \lvert \sigma_n(f)-\sigma_n^0(f)\rvert \\
		=& \frac{1}{4\pi^2}\left|\sum_{j=1}^{n}\mathbb{E}\oint_{\mathcal{C}_{1}}\oint_{\mathcal{C}_{2}}f^{\prime}(z_1)f^{\prime}(z_2)\mathbb{E}_j(\varepsilon_j(z_1)b_j(z_1))\mathbb{E}_j(\varepsilon_j(z_2)b_j(z_2))dz_1dz_2 \right. \\
		&  \quad\quad\quad \left. -\oint_{\mathcal{C}_1}\oint_{\mathcal{C}_2}f^{\prime}(z_1)f^{\prime}(z_2)a_n(z_1,z_2)\int_{0}^{1}\frac{1}{1-ta_n(z_1,z_2)}dtdz_1dz_2  \right| \\
		\leq & Kn^{-1}. 
	\end{align*}
	There is a similar discussion under the RG case. Further, we can get the final result of Lemma \ref{convergence of variance}.

	\subsection{Proof of Lemma $\ref{convergence of mean}$}\label{proof of convergence of mean}

	Due to the statement in Section \ref{Stieltjes transform}, 
	\begin{align*}
		\int f(x)d\mathbb{E}G_n(x)= p\int f(x)d[\mathbb{E}F^{B_n}(x)-F^{y_n,H_p}(x)]=-\frac{1}{2\pi i}\oint_{\mathcal{C}}f(z)M_n^2(z)dz.
	\end{align*}
	We need to estimate, under the RG case,
	\begin{align*}
		\sup_{z\in\mathcal{C}}\left| M_n^2(z)-\frac{y_n\int \underline{s}_n^0(z)^3t^2(1+t\underline{s}_n^0(z))^{-3}dH_p(t)}{(1-y_n\int \underline{s}_n^0(z)^2t^2(1+t\underline{s}_n^0(z))^{-2}dH_p(t))^2}\right|,
	\end{align*}
	and under the CG case, $\sup_{z\in\mathcal{C}}\lvert M_n^2(z)\rvert$.
	Let $A_n(z)=y_n\int \frac{dH_p(t)}{1+t\mathbb{E}\underline{s}_n(z)}+zy_n\mathbb{E}s_n(z)$. Using the identity
	\begin{align*}
		\mathbb{E}\underline{s}_n(z)=-\frac{1-y_n}{z}+y_n\mathbb{E}s_n
	\end{align*}
	we have 
	\begin{align*}
		A_n(z) = y_n\int \frac{dH_p(t)}{1+t\mathbb{E}\underline{s}_n(z)}-y_n+z\mathbb{E}\underline{s}_n(z)+1=-\mathbb{E}\underline{s}_n(z)\left(-z-\frac{1}{\mathbb{E}\underline{s}_n(z)}+y_n\int \frac{tdH_p(t)}{1+t\mathbb{E}\underline{s}_n(z)}\right).
	\end{align*}
	It follows that
	\begin{align}\label{Es_n(z)}
		\mathbb{E}\underline{s}_n(z)=\left[ -z+y_n\int \frac{tdH_p(t)}{1+t\mathbb{E}\underline{s}_n(z)}+A_n/\mathbb{E}\underline{s}_n(z)\right]^{-1}.
	\end{align} 
	From \eqref{equation of s_n^0} and \eqref{Es_n(z)}, we write
	\begin{align}\label{M_n^2}
		& M_n^2(z)=p[\mathbb{E}s_{F^{B_n}}(z)-s_{F^{y_n,H_p}}(z)]= n[\mathbb{E}\underline{s}_n(z)-\underline{s}_n^0(z)] \\
		\nonumber=& n\left[\frac{1}{-z+y_n\int \frac{tdH_p(t)}{1+t\mathbb{E}\underline{s}_n(z)}+A_n/\mathbb{E}\underline{s}_n(z)}-\frac{1}{-z+y_n\int \frac{tdH_p(t)}{1+t\underline{s}_n^0(z)}}\right] \\
		\nonumber=& -n\underline{s}_n^0 A_n\left[1-y_n\mathbb{E}\underline{s}_n \underline{s}_n^0\int\frac{t^2dH_p(t)}{(1+t\mathbb{E}\underline{s}_n)(1+t\underline{s}_n^0)}\right]^{-1}.
	\end{align}
	Due to (4.4) in \cite{BaiS04C}, we can get the existence of $\xi\in(0,1)$ such that for all $n$ large 
	\begin{align*}
		\sup_{z\in\mathcal{C}}\bigg\lvert y_n\mathbb{E}\underline{s}_n \underline{s}_n^0\int\frac{t^2dH_p(t)}{(1+t\mathbb{E}\underline{s}_n)(1+t\underline{s}_n^0)}\bigg\rvert<\xi.
	\end{align*}
	Thus the denominator of \eqref{M_n^2} is bounded away from zero.
	Our next task is to investigate the limiting behavior of
	\begin{align*}
		& nA_n=n\left(y_n\int \frac{dH_p(t)}{1+t\mathbb{E}\underline{s}_n(z)}+zy_n\mathbb{E}s_n(z)\right) =n\mathbb{E}\beta_1[\bbr_1^{*}\bbD_1^{-1}(\mathbb{E}\underline{s}_n\bbT+\bbI)^{-1}\bbr_1-n^{-1}\mathbb{E}\operatorname{tr}(\mathbb{E}\underline{s}_n\bbT+\bbI)^{-1}\bbT\bbD^{-1}],
	\end{align*}
	for $z\in\mathcal{C}$[see (5.2) in \cite{BaiS98N}]. Throughout the following, all bounds, including $O(\cdot)$ and $o(\cdot)$ expressions hold uniformly for $z\in\mathcal{C}$. The positive constant $K$ is independent of $z$.  
	According to \eqref{D^{-1}-D_j^{-1}}, \eqref{beta_1}, Lemmas \ref{a(v)_quadratic_minus_trace} and \ref{E|trTDj-1-EtrTDj-1|}, we have 
	\begin{align}
		& \lvert\mathbb{E}\operatorname{tr}(\mathbb{E}\underline{s}_n\bbT+\bbI)^{-1}\bbT\bbD_1^{-1}-\mathbb{E}\operatorname{tr}(\mathbb{E}\underline{s}_n\bbT+\bbI)^{-1}\bbT\bbD^{-1}-b_n\mathbb{E}\bbr_1^{*}\bbD_1^{-1}(\mathbb{E}\underline{s}_n\bbT+\bbI)^{-1}\bbT\bbD_1^{-1}\bbr_1\rvert \\
		\nonumber =&\lvert\mathbb{E}\beta_1 \operatorname{tr}(\mathbb{E}\underline{s}_n\bbT+\bbI)^{-1}\bbT\bbD_1^{-1}\bbr_1\bbr_1^{*}\bbD_1^{-1}-b_n\mathbb{E}\bbr_1^{*}\bbD_1^{-1}(\mathbb{E}\underline{s}_n\bbT+\bbI)^{-1}\bbT\bbD_1^{-1}\bbr_1\rvert \\
		\nonumber =&\lvert b_n\mathbb{E}(1-\beta_1\gamma_1)\bbr_1^{*}\bbD_1^{-1}(\mathbb{E}\underline{s}_n\bbT+\bbI)^{-1}\bbT\bbD_1^{-1}\bbr_1 -b_n\mathbb{E}\bbr_1^{*}\bbD_1^{-1}(\mathbb{E}\underline{s}_n\bbT+\bbI)^{-1}\bbT\bbD_1^{-1}\bbr_1\rvert \\
		\nonumber =&\lvert b_n\mathbb{E}\beta_1\gamma_1\bbr_1^{*}\bbD_1^{-1}(\mathbb{E}\underline{s}_n\bbT+\bbI)^{-1}\bbT\bbD_1^{-1}\bbr_1\rvert \leq Kn^{-1}. 
	\end{align}
	Since \eqref{beta_1}, we have
	\begin{align*}
		& n\mathbb{E}\beta_1\bbr_1^{*}\bbD_1^{-1}(\mathbb{E}\underline{s}_n\bbT+\bbI)^{-1}\bbr_1-\mathbb{E}\beta_1\mathbb{E}\operatorname{tr}(\mathbb{E}\underline{s}_n\bbT+\bbI)^{-1}\bbT\bbD_1^{-1} \\
		=&-b_n^2 n\mathbb{E}\gamma_1\bbr_1^{*}\bbD_1^{-1}(\mathbb{E}\underline{s}_n\bbT+\bbI)^{-1}\bbr_1  +b_n^2\left[n\mathbb{E}\beta_1\gamma_1^2\bbr_1^{*}\bbD_1^{-1}(\mathbb{E}\underline{s}_n\bbT+\bbI)^{-1}\bbr_1 \right. \\
		& \quad\quad\quad\quad \left.-(\mathbb{E}\beta_1\gamma_1^2)\mathbb{E}\operatorname{tr}(\mathbb{E}\underline{s}_n\bbT+\bbI)^{-1}\bbT\bbD_1^{-1}\right] \\
		=&-b_n^2n\mathbb{E}\gamma_1\bbr_1^{*}\bbD_1^{-1}(\mathbb{E}\underline{s}_n\bbT+\bbI)^{-1}\bbr_1+b_n^2 \operatorname{Cov}(\beta_1\gamma_1^2,\operatorname{tr}\bbD_1^{-1}(\mathbb{E}\underline{s}_n\bbT+\bbI)^{-1}\bbT) \\
		& +b_n^2\mathbb{E}\left[n\beta_1\gamma_1^2\bbr_1^{*}\bbD_1^{-1}(\mathbb{E}\underline{s}_n\bbT+\bbI)^{-1}\bbr_1-\beta_1\gamma_1^2\operatorname{tr}\bbD_1^{-1}(\mathbb{E}\underline{s}_n\bbT+\bbI)^{-1}\bbT\right].
	\end{align*}
	Here $\operatorname{Cov}(X,Y)=\mathbb{E}XY-\mathbb{E}X\mathbb{E}Y$.
	Using (4.3) in \cite{BaiS04C},  \eqref{beta_1}, Lemmas \ref{bound moment of b_j and beta_j}, \ref{a(v)_quadratic_minus_trace}, \ref{the lemma of the three product of quadratic minus trace} and \ref{E|trTDj-1-EtrTDj-1|},  we have 
	\begin{align*}
		& \left|\mathbb{E}\left[n\beta_1\gamma_1^2\bbr_1^{*}\bbD_1^{-1}(\mathbb{E}\underline{s}_n\bbT+\bbI)^{-1}\bbr_1-\beta_1\gamma_1^2\operatorname{tr}\bbD_1^{-1}(\mathbb{E}\underline{s}_n\bbT+\bbI)^{-1}\bbT\right]\right| \\
		=&\left| nb_n\mathbb{E}\left[\gamma_1^2(\bbr_1^{*}\bbD_1^{-1}(\mathbb{E}\underline{s}_n\bbT+\bbI)^{-1}\bbr_1-n^{-1}\operatorname{tr}\bbD_1^{-1}(\mathbb{E}\underline{s}_n\bbT+\bbI)^{-1}\bbT)\right] \right. \\
		&\quad- nb_n^2\mathbb{E}\left[\gamma_1^3(\bbr_1^{*}\bbD_1^{-1}(\mathbb{E}\underline{s}_n\bbT+\bbI)^{-1}\bbr_1-n^{-1}\operatorname{tr}\bbD_1^{-1}(\mathbb{E}\underline{s}_n\bbT+\bbI)^{-1}\bbT)\right] \\
		&\quad \left. +b_n^2\mathbb{E}\left[\beta_1\gamma_1^4(\bbr_1^{*}\bbD_1^{-1}(\mathbb{E}\underline{s}_n\bbT+\bbI)^{-1}\bbr_1-n^{-1}\operatorname{tr}\bbD_1^{-1}(\mathbb{E}\underline{s}_n\bbT+\bbI)^{-1}\bbT)\right]\right| \\
		\leq& Kn\left| \mathbb{E}(\bbr_1^{*}\bbD_1^{-1}\bbr_1-n^{-1}\operatorname{tr}\bbT\bbD_1^{-1})^2(\bbr_1^{*}\bbD_1^{-1}(\mathbb{E}\underline{s}_n\bbT+\bbI)^{-1}\bbr_1-n^{-1}\operatorname{tr}\bbD_1^{-1}(\mathbb{E}\underline{s}_n\bbT+\bbI)^{-1}\bbT)\right| \\
		&\quad +Kn\sqrt{\mathbb{E}\lvert\gamma_1\rvert^6}\sqrt{\mathbb{E}\lvert \bbr_1^{*}\bbD_1^{-1}(\mathbb{E}\underline{s}_n\bbT+\bbI)^{-1}\bbr_1-n^{-1}\operatorname{tr}\bbD_1^{-1}(\mathbb{E}\underline{s}_n\bbT+\bbI)^{-1}\bbT\rvert^2}  \\
		&\quad +Kn\sqrt{\mathbb{E}\lvert \gamma_1\rvert^8}\sqrt{\mathbb{E}\lvert \beta_1(\bbr_1^{*}\bbD_1^{-1}(\mathbb{E}\underline{s}_n\bbT+\bbI)^{-1}\bbr_1-n^{-1}\operatorname{tr}\bbD_1^{-1}(\mathbb{E}\underline{s}_n\bbT+\bbI)^{-1}\bbT)\rvert^2}  \\
		&\leq Kn^{-1}+Kn^{-1}+Kn^{-3/2} \leq Kn^{-1}.
	\end{align*} 
	Using (4.3) in \cite{BaiS04C}, Lemmas \ref{bound moment of b_j and beta_j}, \ref{a(v)_quadratic_minus_trace} and \ref{E|trTDj-1-EtrTDj-1|}, we have
	\begin{align*}
		&  \left|\operatorname{Cov}(\beta_1\gamma_1^2,\operatorname{tr}\bbD_1^{-1}(\mathbb{E}\underline{s}_n\bbT+\bbI)^{-1}\bbT)\right| \\
		 \leq & (\mathbb{E}\lvert\beta_1\rvert^4)^{1/4}(\mathbb{E}\lvert\gamma_1\rvert^8)^{1/4} (\mathbb{E}\lvert \operatorname{tr}\bbD_1^{-1}(\mathbb{E}\underline{s}_n\bbT+\bbI)^{-1}\bbT-\mathbb{E}\operatorname{tr}\bbD_1^{-1}(\mathbb{E}\underline{s}_n\bbT+\bbI)^{-1}\bbT\rvert^2)^{1/2} \\
		\leq & Kn^{-1}.
	\end{align*}
	Write
	\begin{align*}
		&  n\mathbb{E}\gamma_1\bbr_1^{*}\bbD_1^{-1}(\mathbb{E}\underline{s}_n\bbT+\bbI)^{-1}\bbr_1 \\
		=& n\mathbb{E}[(\bbr_1^{*}(\mathbb{E}\underline{s}_n\bbT+\bbI)^{-1}\bbr_1-n^{-1}\operatorname{tr}\bbD_1^{-1}(\mathbb{E}\underline{s}_n\bbT+\bbI)^{-1}\bbT)  (\bbr_1^{*}\bbD_1^{-1}\bbr_1-n^{-1}\operatorname{tr}\bbD_1^{-1}\bbT)] \\
		&\quad +n^{-1}\operatorname{Cov}(\operatorname{tr}\bbD_1^{-1}\bbT,\operatorname{tr}\bbD_1^{-1}(\mathbb{E}\underline{s}_n\bbT+\bbI)^{-1}\bbT).
	\end{align*}
	From Lemma \ref{E|trTDj-1-EtrTDj-1|} we see the second term above is $O(n^{-1})$. Since \eqref{|b_j-Ebeta_j|}, we get $\mathbb{E}\beta_1=b_n+O(n^{-1})$.
	Therefore, we arrive at
	\begin{align}
		& nA_n=n\left(y_n\int \frac{dH_p(t)}{1+t\mathbb{E}\underline{s}_n(z)}+zy_n\mathbb{E}s_n(z)\right) \\
		\nonumber=& b_n^2 n^{-1}\mathbb{E}\operatorname{tr}\bbD_1^{-1}(\mathbb{E}\underline{s}_n\bbT+\bbI)^{-1}\bbT\bbD_1^{-1}\bbT \\
		\nonumber & -b_n^2n\mathbb{E}[(\bbr_1^{*}(\mathbb{E}\underline{s}_n\bbT+\bbI)^{-1}\bbr_1-n^{-1}\operatorname{tr}\bbD_1^{-1}(\mathbb{E}\underline{s}_n\bbT+\bbI)^{-1}\bbT) (\bbr_1^{*}\bbD_1^{-1}\bbr_1-n^{-1}\operatorname{tr}\bbD_1^{-1}\bbT)]+O(n^{-1}).
	\end{align}
	Due to Lemma \ref{product of two quadratic minus trace of two matrices}, we see that under the CG case 
	\begin{align}\label{nA_n in CG case}
		nA_n=O(n^{-1}),
	\end{align} 
	while under the RG case
	\begin{align}
		nA_n=-b_n^2 n^{-1}\mathbb{E}\operatorname{tr}\bbD_1^{-1}(\mathbb{E}\underline{s}_n\bbT+\bbI)^{-1}\bbT\bbD_1^{-1}\bbT+O(n^{-1}).
	\end{align}
	Therefore from \eqref{nA_n in CG case} and the corresponding bound involving $\underline{s}_n^0(z)$, combining \eqref{M_n^2},  under the CG case,
	\begin{align}\label{sup M_n^2 in CG case}
		\sup_{z\in\mathcal{C}}\lvert M_n^2(z)\rvert=O(n^{-1}).
	\end{align}
	We now find the limit of $n^{-1}\mathbb{E}\operatorname{tr}\bbD_1^{-1}(\mathbb{E}\underline{s}_n\bbT+\bbI)^{-1}\bbT\bbD_1^{-1}\bbT$. Applications of \eqref{D^{-1}-D_j^{-1}}, (4.3) in \cite{BaiS04C}, Lemmas \ref{bound moment of b_j and beta_j}, \ref{a(v)_quadratic_minus_trace} and \ref{E|trTDj-1-EtrTDj-1|} show that both
	\begin{align*}
		\mathbb{E}\operatorname{tr}\bbD_1^{-1}(\mathbb{E}\underline{s}_n\bbT+\bbI)^{-1}\bbT\bbD_1^{-1}\bbT-\mathbb{E}\operatorname{tr}\bbD^{-1}(\mathbb{E}\underline{s}_n\bbT+\bbI)^{-1}\bbT\bbD_1^{-1}\bbT
	\end{align*}
	and
	\begin{align*}
		\mathbb{E}\operatorname{tr}\bbD^{-1}(\mathbb{E}\underline{s}_n\bbT+\bbI)^{-1}\bbT\bbD_1^{-1}\bbT-\mathbb{E}\operatorname{tr}\bbD^{-1}(\mathbb{E}\underline{s}_n\bbT+\bbI)^{-1}\bbT\bbD^{-1}\bbT
	\end{align*}
	are bounded. Therefore it is sufficient to consider 
	\begin{align*}
		n^{-1}\mathbb{E}\operatorname{tr}\bbD^{-1}(\mathbb{E}\underline{s}_n\bbT+\bbI)^{-1}\bbT\bbD^{-1}\bbT.
	\end{align*}
	Write
	\begin{align*}
		\bbD(z)+z\bbI-b_n(z)\bbT=\sum_{j=1}^{n}\bbr_j\bbr_j^{*}-b_n(z)\bbT.
	\end{align*}
	It is straightforward to verify that $z\bbI-b_n(z)\bbT$ is non-singular. Multiplying by $(z\bbI-b_n(z)\bbT)^{-1}$ on the left-hand side, $\bbD_j^{-1}(z)$ on the right-hand side and using \eqref{r_i^*(C+r_ir_i^*)^-1}, we get
	\begin{align}\label{D^-1(z)}
		\bbD^{-1}(z)= & -(z \bbI-b_n(z) \bbT)^{-1} +\sum_{j=1}^n \beta_j(z)(z \bbI-b_n(z) \bbT)^{-1} \bbr_j \bbr_j^* \bbD_j^{-1}(z) \\
		\nonumber& -b_n(z)(z \bbI-b_n(z) \bbT)^{-1} \bbT \bbD^{-1}(z) \\
		\nonumber= & -(z \bbI-b_n(z) \bbT)^{-1}+b_n(z) \bbA(z)+\bbB(z)+\bbC(z)
	\end{align}
	where
	\begin{align*}
		& \bbA(z)=\sum_{j=1}^n(z \bbI-b_n(z) \bbT)^{-1}(\bbr_j \bbr_j^*-n^{-1} \bbT) \bbD_j^{-1}(z), \\
		& \bbB(z)=\sum_{j=1}^n(\beta_j(z)-b_n(z))(z \bbI-b_n(z) \bbT)^{-1} \bbr_j \bbr_j^* \bbD_j^{-1}(z)
	\end{align*}
	
	and
	\begin{align*}
		\bbC(z) & =n^{-1} b_n(z)(z \bbI-b_n(z) \bbT)^{-1} \bbT \sum_{j=1}^n(\bbD_j^{-1}(z)-\bbD^{-1}(z)) \\
		& =n^{-1} b_n(z)(z \bbI-b_n(z) \bbT)^{-1} \bbT \sum_{j=1}^n \beta_j(z) \bbD_j^{-1}(z) \bbr_j \bbr_j^* \bbD_j^{-1}(z) .
	\end{align*}
	From Lemma \ref{bound spectral norm of zI-bj(z)T} , it follows that $\lVert (z\bbI-b_n(z)\bbT)^{-1}\rVert$ is bounded. We have by \eqref{beta_1}, Lemmas \ref{a(v)_quadratic_minus_trace} and \ref{E|trTDj-1-EtrTDj-1|}
	\begin{align}\label{E|b_n-beta_1|}
		\mathbb{E}\lvert\beta_1-b_n\rvert^2=\lvert b_n\rvert^2\mathbb{E}\lvert\beta_1\gamma_1\rvert^2\leq Kn^{-1}.
	\end{align}
	Let $\bbM$ be $n\times n$ matrix. From Lemmas 
	\ref{bound moment of b_j and beta_j},  \ref{a(v)_quadratic_minus_trace}, \eqref{|b_j-Ebeta_j|} and \eqref{E|b_n-beta_1|} we get
	\begin{align}\label{tr B(z)M}
		& \lvert n^{-1} \mathbb{E}\operatorname{tr} \bbB(z) \bbM\rvert \\
		\nonumber &=\lvert n^{-1}\sum_{j=1}^{n}\mathbb{E}[(\beta_j-b_n) (\bbr_j^*\bbD_j^{-1}\bbM(z\bbI-b_n(z)\bbT)^{-1}\bbr_j-n^{-1}\operatorname{tr}\bbT\bbD_j^{-1}\bbM(z\bbI-b_n(z)\bbT)^{-1}) \\
		\nonumber &\quad\quad\quad\quad\quad\quad+(\beta_j-b_n)n^{-1}\operatorname{tr}\bbT\bbD_j^{-1}\bbM(z\bbI-b_n(z)\bbT)^{-1}]\rvert \\
		\nonumber& \leq n^{-1}\sum_{j=1}^{n}\sqrt{\mathbb{E}\lvert \bbr_j^*\bbD_j^{-1}\bbM(z\bbI-b_n(z)\bbT)^{-1}\bbr_j-n^{-1}\operatorname{tr}\bbT\bbD_j^{-1}\bbM(z\bbI-b_n(z)\bbT)^{-1}\rvert^2} \sqrt{\mathbb{E}\lvert \beta_j-b_n\rvert^2}  \\
		\nonumber&\quad\quad+n^{-1}\sum_{j=1}^{n} \lvert \mathbb{E}(\beta_j-b_n)n^{-1}\operatorname{tr}\bbT\bbD_j^{-1}\bbM(z\bbI-b_n(z)\bbT)^{-1}\rvert \\
		\nonumber&\leq Kn^{-1}(\mathbb{E}\lVert \bbM\rVert^2)^{1/2}+Kn^{-1}\mathbb{E}\lVert \bbM\rVert \leq Kn^{-1}(\mathbb{E}\lVert \bbM\rVert^2)^{1/2},
	\end{align}
	and
	\begin{align}\label{tr C(z)M}
		\lvert n^{-1} \mathbb{E} \operatorname{tr} \bbC(z) \bbM\rvert  \leq K \mathbb{E}\lvert\beta_1\rvert \bbr_1^* \bbr_1\lVert \bbD_1^{-1}\rVert^2\lVert \bbM\rVert  \leq K n^{-1}(\mathbb{E}\lVert \bbM\rVert^2)^{1 / 2}.
	\end{align}
	For the following $\bbM, n \times n$ matrix, is nonrandom, bounded in norm. Write
	\begin{align}
		\operatorname{tr} \bbA(z) \bbT \bbD^{-1} \bbM=A_1(z)+A_2(z)+A_3(z),
	\end{align}
	where
	\begin{align*}
		& A_1(z)=\operatorname{tr} \sum_{j=1}^n(z \bbI-b_n \bbT)^{-1} \bbr_j \bbr_j^* \bbD_j^{-1} \bbT(\bbD^{-1}-\bbD_j^{-1}) \bbM \\
		& A_2(z)=\operatorname{tr} \sum_{j=1}^n(z \bbI-b_n \bbT)^{-1}(\bbr_j \bbr_j^* \bbD_j^{-1} \bbT \bbD_j^{-1}-n^{-1} \bbT \bbD_j^{-1} \bbT \bbD_j^{-1}) \bbM
	\end{align*}
	and
	\begin{align*}
		A_3(z)=\operatorname{tr} \sum_{j=1}^n(z \bbI-b_n \bbT)^{-1} n^{-1} \bbT \bbD_j^{-1} \bbT(\bbD_j^{-1}-\bbD^{-1}) \bbM.
	\end{align*}
	
	We have $\mathbb{E}A_2(z)=0$ and similarly to \eqref{tr C(z)M} we have
	\begin{align}
		\lvert \mathbb{E}n^{-1}A_3(z)\rvert\leq Kn^{-1}.
	\end{align}
	Using Lemma \ref{a(v)_quadratic_minus_trace}, \eqref{D^{-1}-D_j^{-1}} and \eqref{E|b_n-beta_1|} we get
	\begin{align*}
		\mathbb{E} n^{-1} A_1(z) & =-\mathbb{E} \beta_1 \bbr_1^* \bbD_1^{-1} \bbT \bbD_1^{-1} \bbr_1 \bbr_1^* \bbD_1^{-1} \bbM(z \bbI-b_n \bbT)^{-1} \bbr_1 \\
		& =-b_n \mathbb{E}(n^{-1} \operatorname{tr} \bbD_1^{-1} \bbT \bbD_1^{-1} \bbT)(n^{-1} \operatorname{tr} \bbD_1^{-1} \bbM(z \bbI-b_n \bbT)^{-1} \bbT)+O(n^{-1}) \\
		& =-b_n \mathbb{E}(n^{-1} \operatorname{tr} \bbD^{-1} \bbT \bbD^{-1} \bbT)(n^{-1} \operatorname{tr} \bbD^{-1} \bbM(z \bbI-b_n \bbT)^{-1} \bbT)+O(n^{-1}) . 
	\end{align*}
	Using Lemmas \ref{bound moment of b_j and beta_j} and \ref{E|trTDj-1-EtrTDj-1|} we find
	\begin{align*}
		& \lvert \operatorname{Cov}  (n^{-1} \operatorname{tr} \bbD^{-1} \bbT \bbD^{-1} \bbT, n^{-1} \operatorname{tr} \bbD^{-1} \bbM (z \bbI-b_n \bbT )^{-1} \bbT ) \rvert \\
		\leq &  (\mathbb{E}\lvert n^{-1} \operatorname{tr} \bbD^{-1} \bbT \bbD^{-1} \bbT\rvert^2 )^{1 / 2} n^{-1} (\mathbb{E}\lvert\operatorname{tr} \bbD^{-1} \bbM(z \bbI-b_n \bbT)^{-1} \bbT-\mathbb{E}\operatorname{tr} \bbD^{-1} \bbM(z \bbI-b_n \bbT)^{-1} \bbT\rvert^2)^{1 / 2} \\
		\leq & K n^{-1}.
	\end{align*}
	Therefore 
	\begin{align}\label{n^-1A_1(z)}
		&\mathbb{E} n^{-1}  A_1(z) \\
		\nonumber= & -b_n(\mathbb{E} n^{-1} \operatorname{tr} \bbD^{-1} \bbT \bbD^{-1} \bbT) (\mathbb{E} n^{-1} \operatorname{tr} \bbD^{-1} \bbM(z \bbI-b_n \bbT)^{-1} \bbT)+O(n^{-1}) .
	\end{align}
	Since $\mathbb{E}\beta_1=-z\mathbb{E}\underline{s}_n$, $\mathbb{E}\beta_1=b_n+O(n^{-1})$ and Lemma \ref{|Es_n-s_n^0|}, we have $b_n=-z\underline{s}_n^0(z)+O(n^{-1})$. From \eqref{D^-1(z)}, \eqref{tr B(z)M} and \eqref{tr C(z)M} we get
	\begin{align}\label{En^-1trD^-1T(zI-b_nT)^-1T}
		& \mathbb{E} n^{-1}  \operatorname{tr} \bbD^{-1} \bbT(z \bbI-b_n \bbT)^{-1} \bbT \\
		\nonumber=& n^{-1} \operatorname{tr}[-(z \bbI-b_n \bbT)^{-1}+\mathbb{E} \bbB(z)+\mathbb{E} \bbC(z)] \bbT(z \bbI-b_n \bbT)^{-1} \bbT \\
		\nonumber=&  -\frac{y_n}{z^2} \int \frac{t^2 d H_p(t)}{(1+t \underline{s}_n^0)^2}+O(n^{-1}) .
	\end{align}
	Similarly,
	\begin{align}\label{En^-1trD^-1(Es_nT+I)^-1T(zI-b_nT)^-1T}
		& \mathbb{E} n^{-1}  \operatorname{tr} \bbD^{-1}(\underline{s}_n^0 \bbT+\bbI)^{-1} \bbT(z \bbI-b_n \bbT)^{-1} \bbT =-\frac{y_n}{z^2} \int \frac{t^2 d H_p(t)}{\left(1+t \underline{s}_n^0\right)^3}+O(n^{-1}) .
	\end{align}
	Using \eqref{D^-1(z)} and \eqref{tr B(z)M}--\eqref{En^-1trD^-1T(zI-b_nT)^-1T}, we get
	\begin{align*}
		&\mathbb{E} n^{-1} \operatorname{tr}  \bbD^{-1} \bbT \bbD^{-1} \bbT \\
		=&  -\mathbb{E} n^{-1} \operatorname{tr} \bbD^{-1} \bbT(z \bbI-b_n \bbT)^{-1} \bbT -b_n^2(\mathbb{E} n^{-1} \operatorname{tr} \bbD^{-1} \bbT \bbD^{-1} \bbT)(\mathbb{E} n^{-1} \operatorname{tr} \bbD^{-1} \bbT(z \bbI-b_n \bbT)^{-1} \bbT)+O(n^{-1}) \\
		=&  \frac{y_n}{z^2} \int \frac{t^2 d H_p(t)}{(1+t \underline{s}_n^0)^2}\left(1+z^2 (\underline{s}_n^0)^2 \mathbb{E} n^{-1} \operatorname{tr} \bbD^{-1} \bbT \bbD^{-1} \bbT\right)+O(n^{-1}).
	\end{align*}
	Therefore
	\begin{align}\label{En^-1trD^-1TD^-1T}
		& \mathbb{E} n^{-1} \operatorname{tr} \bbD^{-1} \bbT \bbD^{-1} \bbT \\
		\nonumber& =\left[\frac{y_n}{z^2} \int \frac{t^2 d H_p(t)}{(1+t \underline{s}_n^0)^2}\right]\left[1-y_n \int \frac{(\underline{s}_n^0)^2 t^2 d H_p(t)}{(1+t \underline{s}_n^0)^2}\right]^{-1}+O(n^{-1}).
	\end{align}
	Finally we have from \eqref{D^-1(z)}--\eqref{n^-1A_1(z)}, \eqref{En^-1trD^-1(Es_nT+I)^-1T(zI-b_nT)^-1T} and \eqref{En^-1trD^-1TD^-1T}
	\begin{align*}
		& n^{-1} \mathbb{E} \operatorname{tr} \bbD^{-1}(\underline{s}_n^0 \bbT+\bbI)^{-1} \bbT \bbD^{-1} \bbT \\
		=&  -\mathbb{E} n^{-1}\operatorname{tr} \bbD^{-1}(\underline{s}_n^0 \bbT+\bbI)^{-1} \bbT(z \bbI-b_n \bbT)^{-1} \bbT  -b_n^2(\mathbb{E} n^{-1} \operatorname{tr} \bbD^{-1} \bbT \bbD^{-1} \bbT) \\
		&\quad\quad \times(\mathbb{E} n^{-1} \operatorname{tr} \bbD^{-1}(\underline{s}_n^0 \bbT+\bbI)^{-1} \bbT(z \bbI-b_n \bbT)^{-1} \bbT)+O(n^{-1}) \\
		&= \frac{y_n}{z^2} \int \frac{t^2 d H_p(t)}{(1+t \underline{s}_n^0)^3} \\
		& \quad\times\left(1+z^2 (\underline{s}_n^0)^2\left[\frac{y_n}{z^2} \int \frac{t^2 d H_p(t)}{(1+t \underline{s}_n^0(z))^2}\right]\left[1-y_n \int \frac{(\underline{s}_n^0)^2 t^2 d H_p(t)}{(1+t \underline{s}_n^0)^2}\right]^{-1}\right)+O(n^{-1}) \\
		&=  \left[\frac{y_n}{z^2} \int \frac{t^2 d H_p(t)}{(1+t \underline{s}_n^0)^3}\right]\left[1-y_n \int \frac{(\underline{s}_n^0)^2 t^2 d H_p(t)}{(1+t \underline{s}_n^0)^2}\right]^{-1}+O(n^{-1}) . 
	\end{align*}
	Therefore, from \eqref{M_n^2} we conclude that under the RG case
	\begin{align}\label{sup M_n^2 in RG case}
		\sup_{z\in\mathcal{C}}\left| M_n^2(z)-\frac{y_n\int \underline{s}_n^0(z)^3t^2(1+t\underline{s}_n^0(z))^{-3}dH_p(t)}{(1-y_n\int \underline{s}_n^0(z)^2t^2(1+t\underline{s}_n^0(z))^{-2}dH_p(t))^2}\right| =O(n^{-1}).
	\end{align}
	From \eqref{sup M_n^2 in RG case} we can get \eqref{nonrandom part under RG case}. From \eqref{sup M_n^2 in CG case} we can get \eqref{nonrandom part under CG case}. Thus, the proof of Lemma \ref{convergence of mean} is complete.

	


	\begin{appendix}
	\section{Proof of Proposition \ref{simplification of the random part}}\label{Proof of Proposition Simplification}
	
	From the definition of $Q_j$ provided below \eqref{decomposition of random part}, it follows that	
	\begin{align}
		\nonumber &\sum_{j=1}^{n}\oint_{\mathcal{C}}f^{\prime}(z) (\mathbb{E}_j-\mathbb{E}_{j-1})Q_j(z)dz \\
		=&  \sum_{j=1}^{n}\oint_{\mathcal{C}}f^{\prime}(z) (\mathbb{E}_j-\mathbb{E}_{j-1})\varepsilon_j(z)(\widetilde{\beta}_j(z)-b_j(z))dz \label{the first term of simplification of Qj} \\
		 & +\sum_{j=1}^{n}\oint_{\mathcal{C}}f^{\prime}(z) (\mathbb{E}_j-\mathbb{E}_{j-1})R_j(z)I\left\lbrace \lvert \varepsilon_j(z)\widetilde{\beta}_j(z)\rvert<\frac{1}{2}\right\rbrace dz  \label{the second term of simplification of Qj} \\
		 & +	\sum_{j=1}^{n}\oint_{\mathcal{C}}f^{\prime}(z) (\mathbb{E}_j-\mathbb{E}_{j-1})R_j(z)I\left\lbrace \lvert \varepsilon_j(z)\widetilde{\beta}_j(z)\rvert\geq\frac{1}{2}\right\rbrace dz  \label{the third term of simplification of Qj}.
	\end{align}
	We proceed to simplify items \eqref{the first term of simplification of Qj}, \eqref{the second term of simplification of Qj} and \eqref{the third term of simplification of Qj} individually. In view of Remark 	\ref{about K-S distance equivalent}, it suffices to consider the moments of these terms.
	
	\textbf{Estimation of \eqref{the third term of simplification of Qj}}. Denote $\Lambda_j(z)=I\left\lbrace \lvert \varepsilon_j(z)\widetilde{\beta}_j(z)\rvert\geq\frac{1}{2}\right\rbrace$. We write
		\begin{align*}
			& \mathbb{E}\left|\sum_{j=1}^{n}\oint_{\mathcal{C}}f^{\prime}(z) (\mathbb{E}_j-\mathbb{E}_{j-1})R_j(z)\Lambda_j(z) dz\right|   \\
			\leq& \sum_{j=1}^{n}\mathbb{E}\left[ \left| \int_{\mathcal{C}_u\cup\overline{\mathcal{C}_u}}f^{\prime}(z) (\mathbb{E}_j-\mathbb{E}_{j-1})R_j(z)\Lambda_j(z) dz  \right| \right.  +  \left| \int_{\mathcal{C}_l\cup\overline{\mathcal{C}_l}}f^{\prime}(z) (\mathbb{E}_j-\mathbb{E}_{j-1})R_j(z)\Lambda_j(z) dz \right| \\
			&\quad\quad\quad +\left.  \left| \int_{\mathcal{C}_r\cup\overline{\mathcal{C}_r}}f^{\prime}(z) (\mathbb{E}_j-\mathbb{E}_{j-1})R_j(z)\Lambda_j(z) dz  \right|\right]  \\
			\leq& K\sum_{j=1}^{n}\left[\int_{x_l}^{x_r} \mathbb{E}\left[\lvert R_j(u+iv_0)\rvert \Lambda_j(u+iv_0)\right] du \right.+ \int_{x_l}^{x_r} \mathbb{E}\left[\lvert R_j(u-iv_0)\rvert \Lambda_j(u-iv_0)\right] du \\
			& \quad\quad+ \int_{-v_0}^{v_0} \mathbb{E}\left[\lvert R_j(x_r+iv)\rvert \Lambda_j(x_r+iv)\right] dv \left. + \int_{-v_0}^{v_0} \mathbb{E}\left[\lvert R_j(x_l+iv)\rvert \Lambda_j(x_l+iv)\right] dv \right].
		\end{align*}
		The second inequality mentioned above follows from the estimation formulas for complex integrals: for any $z=u+iv$ and integrable function $g(z)$ on $\mathcal{C}$, 
		\begin{align}\label{estimation formulas for complex integrals}
			\left|\oint_{\mathcal{C}}g(z)dz\right|\leq\oint_{\mathcal{C}} \lvert g(z)\rvert \lvert dz\rvert ,\quad  \lvert dz\rvert = \sqrt{(dx)^2+(dy)^2}.
		\end{align}
		It is sufficient to estimate $\mathbb{E}\left[\lvert R_j(z)\rvert \Lambda(z)\right]$, for $z\in\mathcal{C}$.
		First, we observe the following trivial inequality, 
		\begin{align*}
			&  \lvert \ln(1+\varepsilon_j(z)\widetilde{\beta}_j(z))\rvert = [ (\ln\lvert 1+\varepsilon_j(z)\widetilde{\beta}_j(z)\rvert)^2+(arg(1+\varepsilon_j(z)\widetilde{\beta}_j(z)))^2]^{1/2} \\
			\leq& [ (\lvert 1+\varepsilon_j(z)\widetilde{\beta}_j(z)\rvert-1)^2+\pi^2]^{1/2} \leq [ \lvert 1+\varepsilon_j(z)\widetilde{\beta}_j(z)\rvert^2+1+\pi^2]^{1/2} \\
			\leq& [3+2 \lvert \varepsilon_j(z)\widetilde{\beta}_j(z) \rvert^2+\pi^2]^{1/2}.
		\end{align*}
		Here $arg(z)$ denotes the argument of a complex number $z$. Next, due to Lemma \ref{Q_QMTr} and Markov's inequality,
		\begin{align*}
			& \mathbb{E}\left[\lvert R_j(z)\rvert \Lambda(z)\right] \leq \mathbb{E}\lvert \ln(1+\varepsilon_j(z)\widetilde{\beta}_j(z))\rvert \Lambda(z) + \mathbb{E}\lvert \varepsilon_j(z)\widetilde{\beta}_j(z)\rvert \Lambda(z) \\
			\leq& \sqrt{2\mathbb{E}\lvert \varepsilon_j(z)\widetilde{\beta}_j(z)\rvert^2+3+\pi^2}\sqrt{\mathbb{P}\left(\lvert \varepsilon_j(z)\widetilde{\beta}_j(z)\rvert\geq\frac{1}{2}\right)} \\ 
			&\quad +\sqrt{\mathbb{E}\lvert \varepsilon_j(z)\widetilde{\beta}_j(z)\rvert^2}\sqrt{\mathbb{P}\left(\lvert \varepsilon_j(z)\widetilde{\beta}_j(z)\rvert\geq\frac{1}{2}\right)} \\
			\leq& \sqrt{2\mathbb{E}\lvert \varepsilon_j(z)\widetilde{\beta}_j(z)\rvert^2+3+\pi^2}\sqrt{2^{s}\mathbb{E}\lvert \varepsilon_j(z)\widetilde{\beta}_j(z)\rvert^s} \\
			&\quad +\sqrt{\mathbb{E}\lvert \varepsilon_j(z)\widetilde{\beta}_j(z)\rvert^2}\sqrt{2^{s}\mathbb{E}\lvert \varepsilon_j(z)\widetilde{\beta}_j(z)\rvert^s}  \\
			\leq& K_s n^{-s/4}+K_s n^{-(s+2)/4}\leq K_s n^{-s/4}.
		\end{align*}
		Choosing $s=8$ in the preceding inequality leads to 
		\begin{align}\label{Taylor_expansion_1}
			\mathbb{E}\left| \sum_{j=1}^{n}\oint_{\mathcal{C}}f^{\prime}(z) (\mathbb{E}_j-\mathbb{E}_{j-1})R_j(z)\Lambda(z) dz \right|\leq Kn^{-1}.
		\end{align}
		  
	\textbf{Estimation of \eqref{the second term of simplification of Qj}}. Let $\Theta_j(z)=I\left\lbrace \lvert \varepsilon_j(z)\widetilde{\beta}_j(z)\rvert<\frac{1}{2}\right\rbrace$. For any fixed $m\geq 2$, according to Lemma \ref{Burkholder_2},
		\begin{align}
			\nonumber &  \mathbb{E}\left| \sum_{j=1}^{n}\oint_{\mathcal{C}}f^{\prime}(z) (\mathbb{E}_j-\mathbb{E}_{j-1})R_j(z)\Theta_j(z) dz \right|^m \\
			  \leq & K_m \sum_{j=1}^{n}\mathbb{E}\left| \oint_{\mathcal{C}}f^{\prime}(z)(\mathbb{E}_j-\mathbb{E}_{j-1})R_j(z)\Theta_j(z) dz \right|^m  \label{the first term of Taylor expansion 2}\\
			 & +K_m \left(\sum_{j=1}^{n}\mathbb{E}\left| \oint_{\mathcal{C}}f^{\prime}(z)(\mathbb{E}_j-\mathbb{E}_{j-1})R_j(z)\Theta_j(z) dz\right|^2\right)^{m/2}.\label{the second term of Taylor expansion 2}
		\end{align}
	With regards to \eqref{the first term of Taylor expansion 2},  we have
		\begin{align*}
			& \sum_{j=1}^{n}\mathbb{E}\left| \oint_{\mathcal{C}}f^{\prime}(z)(\mathbb{E}_j-\mathbb{E}_{j-1})R_j(z)\Theta_j(z) dz \right|^m \\
			\leq & \sum_{j=1}^{n}\mathbb{E}\left[\int_{x_l}^{x_r} \lvert f^{\prime}(u+iv_0)(\mathbb{E}_j-\mathbb{E}_{j-1}) R_j(u+iv_0)\rvert \Theta_j(u+iv_0) du \right.\\ 
			& \quad\quad+ \int_{x_l}^{x_r} \lvert f^{\prime}(u-iv_0)(\mathbb{E}_j-\mathbb{E}_{j-1}) R_j(u-iv_0)\rvert \Theta_j(u-iv_0) du \\
			& \quad\quad+ \int_{-v_0}^{v_0} \lvert f^{\prime}(x_r+iv)(\mathbb{E}_j-\mathbb{E}_{j-1})R_j(x_r+iv)\rvert \Theta_j(x_r+iv) dv \\
			& \quad\quad\left. + \int_{-v_0}^{v_0} \lvert f^{\prime}(x_l+iv)(\mathbb{E}_j-\mathbb{E}_{j-1})R_j(x_l+iv)\rvert \Theta_j(x_l+iv) dv \right]^m \\
			\leq& K_m\sum_{j=1}^{n}\left[ (x_r-x_l)^{m-1}\int_{x_l}^{x_r} \mathbb{E} \left[\lvert  R_j(u+iv_0)\rvert^m \Theta_j(u+iv_0)\right] du \right. \\
			&\quad\quad\quad + (x_r-x_l)^{m-1}\int_{x_l}^{x_r} \mathbb{E}\left[\lvert  R_j(u-iv_0)\rvert^m \Theta_j(u-iv_0)\right] du \\
			&\quad\quad\quad +(2v_0)^{m-1}\int_{-v_0}^{v_0} \mathbb{E}\left[\lvert R_j(x_r+iv)\rvert^m \Theta_j(x_r+iv)\right] dv \\
			&\quad\quad\quad \left. +(2v_0)^{m-1}\int_{-v_0}^{v_0} \mathbb{E}\left[\lvert R_j(x_l+iv)\rvert^m \Theta_j(x_l+iv)\right] dv\right].  
		\end{align*} 
	The first inequality above is due to \eqref{estimation formulas for complex integrals}, the second one is due to Jensen's inequality for real integrals. Next, we turn our attention to $\mathbb{E}\left[\lvert R_j(z)\rvert^m \Theta_j(z)\right]$, for $z\in\mathcal{C}$. According to 
		the Taylor expansion, we can get the following inequality, for $\lvert z\rvert<1/2$,
		\begin{align*}
			\lvert \ln(1+z)- z\rvert\leq K\lvert z\rvert^2.
		\end{align*}
		Thus, due to Lemma \ref{Q_QMTr},
		\begin{align}\label{R_j(z)^mI<1}
			&  \mathbb{E}\lvert R_j(z)\rvert^m \Theta(z)  \leq K_m \mathbb{E}\lvert \varepsilon_j(z)\widetilde{\beta}_j(z)\rvert^{2m} 
			\leq K_m n^{-m}.
		\end{align}
		In summary, we can get the order of \eqref{the first term of Taylor expansion 2}, which is
		\begin{align}\label{the_first_term_of_lemma4.2}
			&  \mathbb{E}\sum_{j=1}^{n}\left| \oint_{\mathcal{C}}f^{\prime}(z)(\mathbb{E}_j-\mathbb{E}_{j-1})R_j(z)\Theta(z) dz \right|^m \leq K_m n^{-m+1}.
		\end{align}
		With regards to \eqref{the second term of Taylor expansion 2}, we have
		\begin{align*}
			& \left(\sum_{j=1}^{n}\mathbb{E}\left| \oint_{\mathcal{C}}f^{\prime}(z)(\mathbb{E}_j-\mathbb{E}_{j-1})R_j(z)\Theta_j(z) dz\right|^2 \right)^{m/2} \\
			\leq & K_m \left[ \sum_{j=1}^{n}\left[(x_l-x_r)\int_{x_l}^{x_r} \mathbb{E}\left[ \lvert R_j(u+iv_0)\rvert^2 \Theta_j(u+iv_0) \right] du \right. \right.  \\
			&\quad\quad +(x_l-x_r)\int_{x_l}^{x_r} \mathbb{E}\left[\lvert R_j(u-iv_0)\rvert^2 \Theta_j(u-iv_0)\right] du \\
			&\quad\quad+ 2v_0\int_{-v_0}^{v_0} \mathbb{E}\left[\lvert R_j(x_r+iv)\rvert^2 \Theta_j(x_r+iv)\right] dv \\
			&\quad\quad  \left. \left. + 2v_0 \int_{-v_0}^{v_0} \mathbb{E}\left[\lvert R_j(x_l+iv)\rvert^2 \Theta_j(x_l+iv)\right] dv  \right]\right]^{m/2} 
		\end{align*}
		Due to \eqref{R_j(z)^mI<1}, for $z\in\mathcal{C}$, we have
		$$\mathbb{E}\left[\lvert R_j(z)\rvert^2 \Theta_j(z)\right]  \leq K \mathbb{E}\lvert \varepsilon_j(z)\widetilde{\beta}_j(z)\rvert^{4}\leq Kn^{-2}$$
		Thus, the order of \eqref{the second term of Taylor expansion 2} is 
		\begin{align}\label{the_second_D_term_of_lemma4.2}
			&  \left[\sum_{j=1}^{n}\mathbb{E}\left| \oint_{\mathcal{C}}f^{\prime}(z)(\mathbb{E}_j-\mathbb{E}_{j-1})R_j(z)\Theta_j(z) dz\right|^2\right]^{m/2} \leq K_mn^{-m/2}.
		\end{align} 
		Based on (\ref{the_first_term_of_lemma4.2}) and (\ref{the_second_D_term_of_lemma4.2}), we get the order of \eqref{the second term of simplification of Qj}, for  any fixed $m\geq 2$,
		\begin{align}\label{Taylor_expansion_2}
			\mathbb{E}\left|\sum_{j=1}^{n}\oint_{\mathcal{C}}f^{\prime}(z) (\mathbb{E}_j-\mathbb{E}_{j-1})R_j(z) \Theta_j(z) dz\right|^m\leq K_m n^{-m/2}.
		\end{align}

	\textbf{Estimation of \eqref{the first term of simplification of Qj}}. Similarly to the proof of \eqref{Taylor_expansion_2}, for any $m\geq 2$,
		\begin{align}
			\nonumber & \mathbb{E}\left| \sum_{j=1}^{n}\oint_{\mathcal{C}}f^{\prime}(z)(\mathbb{E}_j-\mathbb{E}_{j-1})\varepsilon_j(z)(\widetilde{\beta}_j(z)-b_j(z))dz\right|^m \\
			\leq & K_m \sum_{j=1}^{n}\mathbb{E}\left| \oint_{\mathcal{C}}f^{\prime}(z)(\mathbb{E}_j-\mathbb{E}_{j-1})\varepsilon_j(z)(\widetilde{\beta}_j(z)-b_j(z))dz\right|^m \label{the first term of betaj-bj} \\ 
			& +K_m\left(\sum_{j=1}^{n}\mathbb{E}\left| \oint_{\mathcal{C}}f^{\prime}(z)(\mathbb{E}_j-\mathbb{E}_{j-1})\varepsilon_j(z)(\widetilde{\beta}_j(z)-b_j(z))dz\right|^2\right)^{m/2}.\label{the second term of betaj-bj}
		\end{align}
	We only need to estimate $\mathbb{E}\lvert \varepsilon_j(z)(\widetilde{\beta}_j(z)-b_j(z))\rvert^m$, for $z\in\mathcal{C}$. According to Holder's inequality, Lemmas \ref{QMTr_1} and \ref{E|beta_j-b_j|},
		\begin{align}\label{E|varepsilon(tilde_beta_j-b_j)|m}
			&  \mathbb{E}\lvert \varepsilon_j(z)(\widetilde{\beta}_j(z)-b_j(z))\rvert^m\leq \sqrt{\mathbb{E}\lvert\varepsilon_j(z)\rvert^{2m}}\sqrt{\mathbb{E}\lvert \widetilde{\beta}_j(z)-b_j(z)\rvert^{2m}} \\
			\nonumber	\leq& K_m\sqrt{n^{-m}}\sqrt{n^{-2m}}=K_m n^{-3m/2}.
		\end{align}
	Based on \eqref{E|varepsilon(tilde_beta_j-b_j)|m}, the estimation of \eqref{the first term of betaj-bj} and \eqref{the second term of betaj-bj} are derived as:
		\begin{align*}
			\sum_{j=1}^{n}\mathbb{E}\left| \oint_{\mathcal{C}}f^{\prime}(z)(\mathbb{E}_j-\mathbb{E}_{j-1})\varepsilon_j(z)(\widetilde{\beta}_j(z)-b_j(z))dz\right|^m \leq K_mn^{-3m/2+1},
		\end{align*}
		and
		\begin{align*}
			&   \left(\sum_{j=1}^{n}\mathbb{E}\left| \oint_{\mathcal{C}}f^{\prime}(z)(\mathbb{E}_j-\mathbb{E}_{j-1})\varepsilon_j(z)(\widetilde{\beta}_j(z)-b_j(z))dz\right|^2\right)^{m/2} \leq K_mn^{-m}.
		\end{align*}
	Thus, we can get the order of \eqref{the first term of simplification of Qj}, for any $m\geq2$,
		\begin{align}\label{the order of the first term of simplification of Qj}
			\mathbb{E}\left| \sum_{j=1}^{n}\oint_{\mathcal{C}}f^{\prime}(z)(\mathbb{E}_j-\mathbb{E}_{j-1})\varepsilon_j(z)(\widetilde{\beta}_j(z)-b_j(z))dz\right|^m\leq K_mn^{-m}.
		\end{align}

	From \eqref{Taylor_expansion_1}, \eqref{Taylor_expansion_2} and \eqref{the order of the first term of simplification of Qj}, in conjunction with Remark \ref{about K-S distance equivalent}, for any fixed $\kappa>0$,
	\begin{align}\label{simplify the Taylor expansion 1}
		&   -\frac{p}{\sqrt{\sigma_n(f)}}\oint_{\mathcal{C}}f(z)[s_{F^{\bbB_n}}(z)-\mathbb{E}s_{F^{\bbB_n}}(z)]dz \\
		\nonumber&\stackrel{n^{-\frac{1}{2}+\kappa}}{\sim} -\frac{1}{\sqrt{\sigma_n(f)}}\sum_{j=1}^{n}\left[\oint_{\mathcal{C}}f^{\prime}(z) (\mathbb{E}_j-\mathbb{E}_{j-1})\varepsilon_j(z)b_j(z)dz\right].
	\end{align}
	
	The subsequent lemma, crucial for replacing $\sqrt{\sigma_n(f)}$ with $\sqrt{\sigma_n^0(f)}$ in accordance with Corollary \ref{Chen_2}, is stated here for completeness. As its proof is independent of the current section, we show it in Section \ref{proof of convergence of variance}. 
	
	\begin{lemma}\label{convergence of variance}
		 $(\romannumeral1)$ Under Assumptions \ref{the moment assumption}--\ref{assumption about RG case},
		\begin{align*}
			\left|\frac{\sqrt{\sigma_n^0(f)}}{\sqrt{\sigma_n(f)}}-1\right|\leq Kn^{-1}.
		\end{align*}
		 $(\romannumeral2)$  Under Assumptions \ref{the moment assumption}--\ref{assumption about test function} and \ref{assumption about CG case}, 
		\begin{align*}
			\left|\frac{\sqrt{\sigma_n^0(f)}}{\sqrt{(1/2)\sigma_n(f)}}-1\right|\leq Kn^{-1}.
		\end{align*}
	\end{lemma}
	
	From \eqref{simplify the Taylor expansion 1} and Lemma \ref{convergence of variance}, we can derive the stated result of Proposition \ref{simplification of the random part}.

	\section{Proof of Proposition \ref{Propostition of using stein method estimate rate}}\label{the lemmas in Proposition of using stein method estimate rate}
		
	This appendix details the proofs of Lemmas \ref{lemma of sumYj2E(g'-g'(W^j))}, \ref{lemma of sumYj2E(g'-g'(W))}, \ref{lemma of sumEYjsum(Yk-Ykj)gh'} and \ref{lemma of sumEYjsumYk-Ykjgh'Wnj} which are instrumental in establishing Proposition \ref{Propostition of using stein method estimate rate}.
	
	\begin{proof}[Proof of Lemma \ref{lemma of sumYj2E(g'-g'(W^j))}]
		According to \eqref{decomposition of rho}, \eqref{Yjs}, \eqref{(sum k>j(Yk-Ykj))t} and Lemma \ref{properties_of_the_smoothed_stein_solution},
		\begin{align*}
			\nonumber &\left|\sum_{j=1}^{n}\mathbb{E}Y_j^2\mathbb{E}\rho \right| \\
			\nonumber \leq& Kn^{-1}\sum_{j=1}^{n} \mathbb{E}\lvert W_n-W_n^{(j)}\rvert+Kn^{-1}\sum_{j=1}^{n}\mathbb{E}\left[\lvert W_n-W_n^{(j)}\rvert\lvert W_n^{(j)}\rvert\right] +Kn^{-1}\sum_{j=1}^{n}\mathbb{E}\left[\lvert W_n-W_n^{(j)}\rvert^2\right] \\
			\nonumber &+Kn^{-1}\sum_{j=1}^{n}\left|\mathbb{E}\left[\int_{0}^{1}(h_{w_0,\theta_n}(W_n^{(j)}+t(W_n-W_n^{(j)}))-h_{w_0,\theta_n}(W_n^{(j)}))dt \right]\right| \\
			\leq &Kn^{-1/2}+Kn^{-1} \\
			\nonumber &+Kn^{-1}\sum_{j=1}^{n} \left|\mathbb{E}\left[\int_{0}^{1}(h_{w_0,\theta_n}(W_n^{(j)}+t(W_n-W_n^{(j)}))-h_{w_0,\theta_n}(W_n^{(j)}))dt \right]\right|.
		\end{align*}
		Next, referring to \eqref{sumEYj2h'}, due to \eqref{Yjs}, \eqref{(sum k>j(Yk-Ykj))t},
		\begin{align*}
			& Kn^{-1}\sum_{j=1}^{n} \left|\mathbb{E}\left[\int_{0}^{1}(h_{w_0,\theta_n}(W_n^{(j)}+t(W_n-W_n^{(j)}))-h_{w_0,\theta_n}(W_n^{(j)}))dt \right]\right| \\
			& \leq K_mn^{-\kappa m/2}+\frac{K\theta_n}{M}+\frac{K}{M}\mathbb{K}(W_n,Z).
		\end{align*}
		By selecting a suitable $m$ such that $ K_m n^{-\kappa m}\leq Kn^{-1/2}$, we then obtain the conclusion of Lemma \ref{lemma of sumYj2E(g'-g'(W^j))}.
	\end{proof}

	\begin{proof}[Proof of Lemma \ref{lemma of sumYj2E(g'-g'(W))}]
	 Referring to \eqref{decomposition of rho}, 
	 	\begin{align}\label{decomposition of tilde-rho}
			 \widetilde{\rho} =& \int_{0}^{1}(g_h^{\prime}(W_n+(1-t)(W_n-W_n^{(j)}))-g_h^{\prime}(W_n))dt \\
			\nonumber =&\int_{0}^{1}((W_n^{(j)}+t(W_n-W_n^{(j)}))g_h(W_n^{(j)}+t(W_n-W_n^{(j)}))-W_ng_h(W_n))dt \\
			\nonumber &+\int_{0}^{1}(h_{w_0,\theta_n}(W_n^{(j)}+t(W_n-W_n^{(j)}))-h_{w_0,\theta_n}(W_n))dt.
		\end{align}
	  
	 Due to \eqref{decomposition of tilde-rho}, \eqref{Yjs}, \eqref{(sum k>j(Yk-Ykj))t} and Lemma \ref{properties_of_the_smoothed_stein_solution},
		\begin{align*}
			& \left|\sum_{j=1}^{n}\mathbb{E}Y_j^2 \mathbb{E}\widetilde{\rho}\right| \\
			\nonumber \leq&\sum_{j=1}^{n}\mathbb{E}\lvert Y_j\rvert^2\mathbb{E}\left[\int_{0}^{1}\lvert t(W_n-W_n^{(j)})\rvert\lvert g_h(W_n+ts(W_n-W_n^{(j)}))\rvert dt\right] \\
			\nonumber &+\sum_{j=1}^{n}\mathbb{E}\lvert Y_j\rvert^2\mathbb{E}\left[\int_{0}^{1}\lvert t(W_n-W_n^{(j)})\rvert\lvert W_n+ts(W_n-W_n^{(j)})\rvert  \lvert g_h^{\prime}(W_n+ts(W_n-W_n^{(j)}))\rvert dt\right] \\ 
			\nonumber &+\sum_{j=1}^{n}\mathbb{E}\lvert Y_j\rvert^2 \left|\mathbb{E}\left[\int_{0}^{1}(h_{w_0,\theta_n}(W_n+t(W_n-W_n^{(j)}))-h_{w_0,\theta_n}(W_n))dt \right]\right| \\
			\nonumber \leq& Kn^{-1}\sum_{j=1}^{n} \mathbb{E}\lvert W_n-W_n^{(j)}\rvert+Kn^{-1}\sum_{j=1}^{n}\mathbb{E}\left[\lvert W_n-W_n^{(j)}\rvert\lvert W_n \rvert\right] +Kn^{-1}\sum_{j=1}^{n}\mathbb{E}\left[\lvert W_n-W_n^{(j)}\rvert^2\right] \\
			\nonumber &+Kn^{-1}\sum_{j=1}^{n}\left|\mathbb{E}\left[\int_{0}^{1}(h_{w_0,\theta_n}(W_n+t(W_n-W_n^{(j)}))-h_{w_0,\theta_n}(W_n))dt \right]\right| \\
			\leq &Kn^{-1/2}+Kn^{-1} \\
			\nonumber &+Kn^{-1}\sum_{j=1}^{n} \left|\mathbb{E}\left[\int_{0}^{1}(h_{w_0,\theta_n}(W_n+t(W_n-W_n^{(j)}))-h_{w_0,\theta_n}(W_n))dt \right]\right|.
		\end{align*}
		Similarly, we only need to estimate the following term, due to \eqref{Yjs}, \eqref{(sum k>j(Yk-Ykj))t},
		\begin{align*}
			&Kn^{-1}\sum_{j=1}^{n} \left|\mathbb{E}\left[\int_{0}^{1}(h_{w_0,\theta_n}(W_n+t(W_n-W_n^{(j)}))-h_{w_0,\theta_n}(W_n))dt \right]\right| \\
			\leq&Kn^{-1}\sum_{j=1}^{n} \left|\mathbb{E}\left[\int_{0}^{1}(h_{w_0,\theta_n}(W_n+t(W_n-W_n^{(j)}))-h_{w_0,\theta_n}(W_n))dt I\left\lbrace \lvert W_n-W_n^{(j)}\rvert\geq\frac{\theta_n}{M}\right\rbrace \right]\right| \\
			&+Kn^{-1}\sum_{j=1}^{n} \left|\mathbb{E}\left[\int_{0}^{1}(h_{w_0,\theta_n}(W_n+t(W_n-W_n^{(j)}))-h_{w_0,\theta_n}(W_n))dt \right.\right. \\
			&\left.\left.\quad\quad\quad\quad\quad\quad\quad I\left\lbrace \lvert W_n-W_n^{(j)}\rvert<\frac{\theta_n}{M},W_n\in\left[w_0-\theta_n,w_0+2\theta_n\right]\right\rbrace \right]\right| \\
			&+Kn^{-1}\sum_{j=1}^{n} \left|\mathbb{E}\left[\int_{0}^{1}(h_{w_0,\theta_n}(W_n+t(W_n-W_n^{(j)}))-h_{w_0,\theta_n}(W_n))dt \right.\right. \\
			&\left.\left.\quad\quad\quad\quad\quad\quad\quad I\left\lbrace \lvert W_n-W_n^{(j)}\rvert<\frac{\theta_n}{M},W_n\notin\left[w_0-\theta_n,w_0+2\theta_n\right]\right\rbrace \right]\right| \\
			\leq & Kn^{-1}\sum_{j=1}^{n}\mathbb{P}\left(\lvert W_n-W_n^{(j)}\rvert\geq\frac{\theta_n}{M}\right) \\
			&+Kn^{-1}\sum_{j=1}^{n}\frac{1}{2\theta_n}\mathbb{E}\left[\lvert W_n-W_n^{(j)}\rvert I\left\lbrace \lvert W_n-W_n^{(j)}\rvert<\frac{\theta_n}{M},W_n\in\left[w_0-\theta_n,w_0+2\theta_n\right]\right\rbrace \right] \\
			\leq& K_m n^{-\kappa m}+\frac{K}{M}\mathbb{P}\left(W_n\in\left[w_0-\theta_n,w_0+2\theta_n\right]\right) \\
			\leq&K_mn^{-\kappa m/2}+\frac{K\theta_n}{M}+\frac{K}{M}\mathbb{K}(W_n,Z).
		\end{align*}
		By selecting $m$ such that $ K_m n^{-\kappa m}\leq Kn^{-1/2}$, we then derive the result of Lemma \ref{lemma of sumYj2E(g'-g'(W))}.	
	\end{proof}

	\begin{proof}[Proof of Lemma \ref{lemma of sumEYjsum(Yk-Ykj)gh'}]
	Initially, due to \eqref{decomposition of rho} and Lemma \ref{properties_of_the_smoothed_stein_solution},
		\begin{align}\label{sumEYjsum(Yk-Ykj)gh'}
			\nonumber & \left| \sum_{j=1}^{n}\mathbb{E}\left[ Y_j\left(\sum_{k> j}(Y_k-Y_{kj})\right)\rho \right] \right| \\
			\nonumber \leq & K\sum_{j=1}^{n}\mathbb{E}\left[\left| Y_j\left(\sum_{k> j}(Y_k-Y_{kj})\right)\right|\lvert W_n-W_n^{(j)}\rvert\right] +K\sum_{j=1}^{n}\mathbb{E}\left[\lvert Y_j\left(\sum_{k> j}(Y_k-Y_{kj})\right)\rvert\lvert W_n-W_n^{(j)}\rvert\lvert W_n^{(j)}\rvert\right] \\
			\nonumber &+K\sum_{j=1}^{n}\mathbb{E} \left[ \left| Y_j \left(\sum_{k> j}(Y_k-Y_{kj})\right) \right| \lvert W_n-W_n^{(j)}\rvert^2\right] \\
			\nonumber &+\left|\sum_{j=1}^{n}\mathbb{E}\left[ Y_j \left(\sum_{k> j}(Y_k-Y_{kj}) \right) \left(\int_{0}^{1}(h_{w_0,\theta_n}(W_n^{(j)}+t(W_n-W_n^{(j)}))-h_{w_0,\theta_n}(W_n^{(j)}))dt \right) \right] \right| \\
			\nonumber \leq & Kn^{-1/2}+Kn^{-1} \\
			 &+\left|\sum_{j=1}^{n}\mathbb{E}\left[Y_j \left( \sum_{k> j}(Y_k-Y_{kj})\right)  \left( \int_{0}^{1}(h_{w_0,\theta_n}(W_n^{(j)}+t(W_n-W_n^{(j)}))-h_{w_0,\theta_n}(W_n^{(j)}))dt \right) \right] \right|.
		\end{align}
		Next, our focus shifts to studying the order of \eqref{sumEYjsum(Yk-Ykj)gh'}. Due to \eqref{Yjs}, \eqref{(sum k>j(Yk-Ykj))t}, 
		\begin{align*}
			\nonumber \eqref{sumEYjsum(Yk-Ykj)gh'} &\leq  \left|\sum_{j=1}^{n}\mathbb{E}\left[Y_j\left(\sum_{k> j}(Y_k-Y_{kj})\right)I\left\lbrace \lvert W_n-W_n^{(j)}\rvert\geq \frac{\theta_n}{M}\right\rbrace \right. \right. \\
			\nonumber &\quad\quad\quad\quad \left. \left. \left(\int_{0}^{1}(h_{w_0,\theta_n}(W_n^{(j)}+t(W_n-W_n^{(j)}))-h_{w_0,\theta_n}(W_n^{(j)}))dt\right)   \right]\right| \\
			\nonumber&+\left|\sum_{j=1}^{n}\mathbb{E}\left[Y_j\left(\sum_{k> j}(Y_k-Y_{kj})\right) I\left\lbrace \lvert W_n-W_n^{(j)}\rvert< \frac{\theta_n}{M},~W_n^{(j)}\in\left[w_0-\theta_n,w_0+2\theta_n\right]\right\rbrace \right. \right. \\
			\nonumber &\quad\quad\quad\quad \left. \left. \left(\int_{0}^{1}(h_{w_0,\theta_n}(W_n^{(j)}+t(W_n-W_n^{(j)}))-h_{w_0,\theta_n}(W_n^{(j)}))dt\right)\right]\right|\\
			\nonumber&+\left|\sum_{j=1}^{n}\mathbb{E}\left[Y_j\left(\sum_{k> j}(Y_k-Y_{kj})\right)I\left\lbrace \lvert W_n-W_n^{(j)}\rvert< \frac{\theta_n}{M},~W_n^{(j)}\notin\left[w_0-\theta_n,w_0+2\theta_n\right]\right\rbrace \right. \right. \\
			\nonumber & \quad\quad\quad\quad \left. \left. \left(\int_{0}^{1}(h_{w_0,\theta_n}(W_n^{(j)}+t(W_n-W_n^{(j)}))-h_{w_0,\theta_n}(W_n^{(j)}))dt\right) \right]\right| \\
			\nonumber \leq&\sum_{j=1}^{n}\left(\mathbb{E}\left| Y_j\left(\sum_{k> j}(Y_k-Y_{kj})\right)\right|^2\right)^{1/2}\left(\mathbb{P}\left(\lvert W_n-W_n^{(j)}\rvert\geq \frac{\theta_n}{M}\right)\right)^{1/2} \\
			\nonumber &+\sum_{j=1}^{n}\frac{1}{2\theta_n}\mathbb{E}\left[\left| Y_j\left(\sum_{k> j}(Y_k-Y_{kj})\right)\right| \lvert W_n-W_n^{(j)}\rvert  \right.\\
			\nonumber & \left. \quad\quad\quad\quad\quad\quad\quad\quad\quad I\left\lbrace \lvert W_n-W_n^{(j)}\rvert< \frac{\theta_n}{M},~W_n^{(j)}\in\left[w_0-\theta_n,w_0+2\theta_n\right]\right\rbrace\right] \\
			\nonumber \leq& \sum_{j=1}^{n} Kn^{-1}\left(\frac{M^m \mathbb{E}\lvert W_n-W_n^{(j)}\rvert^m}{\theta_n^m}\right)^{1/2} \\
			\nonumber &+\sum_{j=1}^{n}\frac{1}{2\theta_n}\mathbb{E}\left[\lvert Y_j\rvert^2\left|\left(\sum_{k> j}(Y_k-Y_{kj}) \right)\right| I\left\lbrace W_n\in\left[w_0-2\theta_n,w_0+3\theta_n\right]\right\rbrace\right] \\
			\nonumber &+\sum_{j=1}^{n}\frac{1}{2\theta_n}\mathbb{E}\left[\lvert Y_j\rvert\left|\left(\sum_{k> j}(Y_k-Y_{kj}) \right)\right|^2 I\left\lbrace W_n\in\left[w_0-2\theta_n,w_0+3\theta_n\right]\right\rbrace\right] \\
			\nonumber \leq&K M^{m/2}n^{-\kappa m/2}\\
			\nonumber &+\sum_{j=1}^{n}\frac{1}{2\theta_n}\mathbb{E}\left[\lvert Y_j\rvert^2\left|\left(\sum_{k> j}(Y_k-Y_{kj}) \right)\right| \right.\\
			\nonumber &\left.\quad\quad\quad I\left\lbrace W_n\in\left[w_0-2\theta_n,w_0+3\theta_n\right],~\lvert Y_j\rvert^2\left|\left(\sum_{k> j}(Y_k-Y_{kj}) \right)\right|<\frac{n^{-3/2+\kappa}}{M}\right\rbrace \right]\\
			\nonumber &+\sum_{j=1}^{n}\frac{1}{2\theta_n}\mathbb{E}\left[\lvert Y_j\rvert^2\left|\left(\sum_{k> j}(Y_k-Y_{kj}) \right)\right| \right.\\
			\nonumber &\left.\quad\quad\quad I\left\lbrace W_n\in\left[w_0-2\theta_n,w_0+3\theta_n\right],~\lvert Y_j\rvert^2\left|\left(\sum_{k> j}(Y_k-Y_{kj}) \right)\right|\geq\frac{n^{-3/2+\kappa}}{M}\right\rbrace\right] \\
			\nonumber &	+\sum_{j=1}^{n}\frac{1}{2\theta_n}\mathbb{E}\left[\lvert Y_j\rvert\left|\left(\sum_{k> j}(Y_k-Y_{kj}) \right)\right|^2 \right.\\
			\nonumber &\left.\quad\quad\quad I\left\lbrace W_n\in\left[w_0-2\theta_n,w_0+3\theta_n\right],~\lvert Y_j\rvert\left|\left(\sum_{k> j}(Y_k-Y_{kj}) \right)\right|^2<\frac{n^{-3/2+\kappa}}{M}\right\rbrace\right] \\
			\nonumber &	+\sum_{j=1}^{n}\frac{1}{2\theta_n}\mathbb{E}\left[\lvert Y_j\rvert\left|\left(\sum_{k> j}(Y_k-Y_{kj}) \right)\right|^2 \right.\\
			\nonumber &\left.\quad\quad\quad I\left\lbrace W_n\in\left[w_0-2\theta_n,w_0+3\theta_n\right],~\lvert Y_j\rvert\left|\left(\sum_{k> j}(Y_k-Y_{kj}) \right)\right|^2\geq\frac{n^{-3/2+\kappa}}{M}\right\rbrace\right] \\
			\nonumber \leq & KM^{m/2}n^{-\kappa m/2} \\
			\nonumber &+\frac{1}{2\theta_n}\sum_{j=1}^{n}\left(\mathbb{E} \left[\lvert Y_j\rvert^4 \left|\left(\sum_{k> j}(Y_k-Y_{kj}) \right)\right|^2 \right]\right)^{1/2} \left(\mathbb{P}\left(\lvert Y_j\rvert^2 \left|\left(\sum_{k> j}(Y_k-Y_{kj}) \right)\right|\geq \frac{n^{-3/2+\kappa}}{M}\right)\right)^{1/2} \\
			\nonumber &+\frac{1}{2\theta_n}\sum_{j=1}^{n}\left(\mathbb{E} \left[\lvert Y_j\rvert^2 \left|\left(\sum_{k> j}(Y_k-Y_{kj}) \right)\right|^4 \right]\right)^{1/2} \left(\mathbb{P}\left(\lvert Y_j\rvert \left|\left(\sum_{k> j}(Y_k-Y_{kj}) \right)\right|^2 \geq \frac{n^{-3/2+\kappa}}{M}\right)\right)^{1/2} \\
			\nonumber &+\frac{1}{M}\mathbb{P}\left(W_n\in\left[w_0-2\theta_n,w_0+3\theta_n\right]\right) \\
			\nonumber \leq & KM^{m/2}n^{-\kappa m/2}+\frac{1}{M}\mathbb{P}\left(W_n\in\left[w_0-2\theta_n,w_0+3\theta_n\right]\right) \\
			\nonumber &+\frac{K}{\theta_n}\sum_{j=1}^{n}n^{-3/2}\left(\frac{M^m \mathbb{E}\left[\lvert Y_j\rvert^{2m}\left|\left(\sum_{k> j}(Y_k-Y_{kj}) \right)\right|^{m}\right]}{n^{-3m/2+\kappa m}}\right)^{1/2} \\
			\nonumber &+\frac{K}{\theta_n}\sum_{j=1}^{n}n^{-3/2}\left(\frac{M^m \mathbb{E}\left[\lvert Y_j\rvert^{m}\left|\left(\sum_{k> j}(Y_k-Y_{kj}) \right)\right|^{2m}\right]}{n^{-3m/2+\kappa m}}\right)^{1/2} \\
			\nonumber \leq &KM^{m/2}n^{-\kappa m/2}+\frac{1}{M}\mathbb{P}\left(W_n\in\left[w_0-2\theta_n,w_0+3\theta_n\right]\right) \\
			\nonumber \leq &KM^{m/2}n^{-\kappa m/2}+\frac{1}{M}\mathbb{K}(W_n,Z)+\frac{K\theta_n}{M}.
		\end{align*}
		Selecting $m$ such that $KM^{m/2}n^{-\kappa m/2}\leq Kn^{-1/2}$ completes the proof of Lemma \ref{lemma of sumEYjsum(Yk-Ykj)gh'}.
	\end{proof}

	\begin{proof}[Proof of Lemma \ref{lemma of sumEYjsumYk-Ykjgh'Wnj}]
	In the first step of the proof, due to Lemma \ref{properties_of_the_smoothed_stein_solution},
		\begin{align*}
			&\left|\mathbb{E}\sum_{j=1}^{n}\left[ Y_j\left(\sum_{k>j}(Y_k-Y_{kj})\right)g_h^{\prime}(W_n^{(j)})\right]\right| =\left| \sum_{j=1}^{n}\mathbb{E}\left[g_h^{\prime}(W_n^{(j)})\mathbb{E}_n^{(j)}\left[Y_j\left(\sum_{k>j}(Y_k-Y_{kj})\right)\right]\right]\right|  \\
			\leq& \sum_{j=1}^{n}\mathbb{E}\left| \mathbb{E}_{n}^{(j)}\left[Y_j\left(\sum_{k>j}(Y_k-Y_{kj})\right)\right]\right|.
		\end{align*}
		We only need to estimate, for $z\in\mathcal{C}$, due to \eqref{D_j^{-1}-D_kj^{-1}},
		\begin{align}
			\nonumber & \mathbb{E}\left| \mathbb{E}_{n}^{(j)}\left[(\mathbb{E}_j-\mathbb{E}_{j-1})(\bbr_j^*\bbD_j^{-1}\bbr_j-n^{-1}\operatorname{tr}\bbT\bbD_j^{-1}) \right. \right.\\
			\nonumber &\quad\quad \left. \left. \left(\sum_{k>j}(\mathbb{E}_k-\mathbb{E}_{k-1})(\bbr_k^*(\bbD_k^{-1}-\bbD_{kj}^{-1})\bbr_k-n^{-1}\operatorname{tr}\bbT(\bbD_k^{-1}-\bbD_{kj}^{-1}))\right)\right]\right| \\
			\nonumber=&\mathbb{E}\left| \mathbb{E}_{n}^{(j)}\left[(\bbr_j^*(\mathbb{E}_j\bbD_j^{-1})\bbr_j-n^{-1}\operatorname{tr}\bbT(\mathbb{E}_j\bbD_j^{-1}))\right.\right.\\
			\nonumber & \quad\quad\left.\left. \left(\sum_{k>j}(\mathbb{E}_k-\mathbb{E}_{k-1})(\bbr_k^*\beta_{jk}\bbD_{jk}^{-1}\bbr_j\bbr_j^*\bbD_{jk}^{-1}\bbr_k-n^{-1}\operatorname{tr}\bbT\beta_{jk}\bbD_{jk}^{-1}\bbr_j\bbr_j^*\bbD_{jk}^{-1})\right)\right]\right|  \\
			\leq & \mathbb{E}\left| \mathbb{E}_{n}^{(j)}\left[(\bbr_j^*(\mathbb{E}_j\bbD_j^{-1})\bbr_j-n^{-1}\operatorname{tr}\bbT(\mathbb{E}_j\bbD_j^{-1}))\right.\right.\label{Esum Yj(sumYk-Ykj)g'W_n(j)1} \\
			\nonumber &\quad\quad\quad\quad \left.\left. \left(\bbr_j^*(\sum_{k>j}(\mathbb{E}_k-\mathbb{E}_{k-1})\widetilde{\beta}_{jk}(\bbD_{jk}^{-1}\bbr_k\bbr_k^*\bbD_{jk}^{-1}-n^{-1}\bbD_{jk}^{-1}\bbT\bbD_{jk}^{-1}))\bbr_j\right)\right]\right|  \\
			 &+\mathbb{E}\left| \mathbb{E}_{n}^{(j)}\left[(\bbr_j^*(\mathbb{E}_j\bbD_j^{-1})\bbr_j-n^{-1}\operatorname{tr}\bbT(\mathbb{E}_j\bbD_j^{-1}))\right.\right. \label{Esum Yj(sumYk-Ykj)g'W_n(j)2} \\
			\nonumber &\quad\quad\quad\quad \left(\sum_{k>j}(\mathbb{E}_k-\mathbb{E}_{k-1})\beta_{jk}\widetilde{\beta}_{jk}(\bbr_j^*\bbD_{jk}^{-1}\bbr_j-n^{-1}\operatorname{tr}\bbT\bbD_{jk}^{-1}) \right.\\ 
			\nonumber &\quad\quad\quad\quad\quad\quad\quad\quad\quad\quad\quad\quad \left.\left. \left.  (\bbr_k^*\bbD_{jk}^{-1}\bbr_j\bbr_j^*\bbD_{jk}^{-1}\bbr_k-n^{-1}\operatorname{tr}\bbT\bbD_{jk}^{-1}\bbr_j\bbr_j^*\bbD_{jk}^{-1})\right)\right]\right|. 
		\end{align}
		Regarding \eqref{Esum Yj(sumYk-Ykj)g'W_n(j)1}, due to Lemmas \ref{Burkholder_1}, \ref{bound moment of b_j and beta_j}, \ref{QMTr}, \ref{product of two quadratic minus trace of two matrices}  and Theorem A.19 in \cite{BaiS10S},  under the  CG case, namely $\mathbb{E} x_{ij}^2=o(n^{-2}\eta_n^2)$ and $\mathbb{E}\lvert x_{ij}\rvert^4=2+o(n^{-3/2}\eta_n^4)$,  
		\begin{align}\label{Esum Yj(sumYk-Ykj)g'W_n(j)_1}
			&\mathbb{E}\left| \mathbb{E}_{n}^{(j)}\left[(\bbr_j^*(\mathbb{E}_j\bbD_j^{-1})\bbr_j-n^{-1}\operatorname{tr}\bbT(\mathbb{E}_j\bbD_j^{-1}))\right.\right.\\
			\nonumber&\quad\quad\quad\quad \left.\left. \left(\bbr_j^*(\sum_{k>j}(\mathbb{E}_k-\mathbb{E}_{k-1})\widetilde{\beta}_{jk}(\bbD_{jk}^{-1}\bbr_k\bbr_k^*\bbD_{jk}^{-1}-n^{-1}\bbD_{jk}^{-1}\bbT\bbD_{jk}^{-1}))\bbr_j\right)\right]\right| \\
			\nonumber \leq &o(n^{-7/2}\eta_n^4)\mathbb{E}\left| \operatorname{tr}(\bbT\mathbb{E}_j\bbD_{j}^{-1})\circ \left(\sum_{k>j}(\mathbb{E}_k-\mathbb{E}_{k-1})\widetilde{\beta}_{jk}(\bbD_{jk}^{-1}\bbr_k\bbr_k^*\bbD_{jk}^{-1}-n^{-1}\bbD_{jk}^{-1}\bbT\bbD_{jk}^{-1})\right)\right| \\
			&\nonumber + o(n^{-6}\eta_n^4)\mathbb{E}\left| \operatorname{tr}(\bbT\mathbb{E}_j\bar{\bbD}_{j}^{-1}) \left(\sum_{k>j}(\mathbb{E}_k-\mathbb{E}_{k-1})\widetilde{\beta}_{jk}(\bbD_{jk}^{-1}\bbr_k\bbr_k^*\bbD_{jk}^{-1}-n^{-1}\bbD_{jk}^{-1}\bbT\bbD_{jk}^{-1})\right)\right|\\
			&\nonumber+n^{-2}\mathbb{E}\left|\operatorname{tr}\bbT(\mathbb{E}_j\bbD_j^{-1})\sum_{k>j}(\mathbb{E}_k-\mathbb{E}_{k-1})\widetilde{\beta}_{jk}(\bbD_{jk}^{-1}\bbr_k\bbr_k^*\bbD_{jk}^{-1}-n^{-1}\bbD_{jk}^{-1}\bbT\bbD_{jk}^{-1}) \right| \\
			\nonumber \leq & n^{-2}\mathbb{E}\left| \lVert\bbT(\mathbb{E}_j\bbD_j^{-1})\rVert\sum_{k>j}(\mathbb{E}_k-\mathbb{E}_{k-1})\widetilde{\beta}_{jk}(\bbr_k^*\bbD_{jk}^{-2}\bbr_k-n^{-1}\operatorname{tr}\bbT\bbD_{jk}^{-2})\right| \\
			\nonumber &+ o(n^{-6}\eta_n^4) \mathbb{E}\left| \lVert\bbT(\mathbb{E}_j\bar{\bbD}_j^{-1})\rVert\sum_{k>j}(\mathbb{E}_k-\mathbb{E}_{k-1})\widetilde{\beta}_{jk}(\bbr_k^*\bbD_{jk}^{-2}\bbr_k-n^{-1}\operatorname{tr}\bbT\bbD_{jk}^{-2})\right| \\
			\nonumber \leq & Kn^{-2}\sqrt{\sum_{k>j}\mathbb{E}\lvert(\mathbb{E}_k-\mathbb{E}_{k-1})\widetilde{\beta}_{jk}(\bbr_k^*\bbD_{jk}^{-2}\bbr_k-n^{-1}\operatorname{tr}\bbT\bbD_{jk}^{-2})\rvert^2} \\
			\nonumber &+Ko(n^{-6}\eta_n^4)\sqrt{\sum_{k>j}\mathbb{E}\lvert(\mathbb{E}_k-\mathbb{E}_{k-1})\widetilde{\beta}_{jk}(\bbr_k^*\bbD_{jk}^{-2}\bbr_k-n^{-1}\operatorname{tr}\bbT\bbD_{jk}^{-2})\rvert^2} \\
			\nonumber\leq& Kn^{-2}.
		\end{align}
		Under the RG case $\mathbb{E}\lvert x_{1j}\rvert^4=3+o(n^{-3/2}\eta_n^{4})$, \eqref{Esum Yj(sumYk-Ykj)g'W_n(j)_1} can get the same order.
		
		As for \eqref{Esum Yj(sumYk-Ykj)g'W_n(j)2}, due to Lemmas \ref{Burkholder_1}, \ref{bound moment of b_j and beta_j}, \ref{Q_QMTr} and \ref{quadratic form minus trace of qudratic form}, 
		\begin{align}\label{Esum Yj(sumYk-Ykj)g'W_n(j)_2}
			&\mathbb{E}\left| \mathbb{E}_{n}^{(j)}\left[(\bbr_j^*(\mathbb{E}_j\bbD_j^{-1})\bbr_j-n^{-1}\operatorname{tr}\bbT(\mathbb{E}_j\bbD_j^{-1}))\right.\right.\\
			\nonumber & \quad\quad\quad  \left(\sum_{k>j}(\mathbb{E}_k-\mathbb{E}_{k-1})\beta_{jk}\widetilde{\beta}_{jk}(\bbr_j^*\bbD_{jk}^{-1}\bbr_j-n^{-1}\operatorname{tr}\bbT\bbD_{jk}^{-1}) \right.\\
			\nonumber &\quad\quad\quad\quad\quad\quad\quad\quad\quad\quad \left.\left. \left. (\bbr_k^*\bbD_{jk}^{-1}\bbr_j\bbr_j^*\bbD_{jk}^{-1}\bbr_k-n^{-1}\operatorname{tr}\bbT\bbD_{jk}^{-1}\bbr_j\bbr_j^*\bbD_{jk}^{-1})\right)\right]\right| \\
			\nonumber\leq & K \left(\mathbb{E}\left| \sum_{k>j}(\mathbb{E}_k-\mathbb{E}_{k-1})\beta_{jk}(\bbr_j^*\bbD_{jk}^{-1}\bbr_j-n^{-1}\operatorname{tr}\bbT\bbD_{jk}^{-1}) \right.\right. \\
			\nonumber &\quad\quad\quad\quad\quad\quad\quad\quad\quad \left. \left. \widetilde{\beta}_{jk}(\bbr_k^*\bbD_{jk}^{-1}\bbr_j\bbr_j^*\bbD_{jk}^{-1}\bbr_k-n^{-1}\operatorname{tr}\bbT\bbD_{jk}^{-1}\bbr_j\bbr_j^*\bbD_{jk}^{-1})\right|^2\right)^{1/2} \\
			\nonumber &\quad\quad \times \sqrt{\mathbb{E}\lvert \bbr_j^*(\mathbb{E}_j\bbD_j^{-1})\bbr_j-n^{-1}\operatorname{tr}\bbT(\mathbb{E}_j\bbD_j^{-1})\rvert^2} \\
			\nonumber \leq & Kn^{-1/2}\left(\sum_{k>j}\mathbb{E}\lvert\beta_{jk}(\bbr_j^*\bbD_{jk}^{-1}\bbr_j-n^{-1}\operatorname{tr}\bbT\bbD_{jk}^{-1}) \right. \\
			\nonumber &\quad\quad\quad\quad\quad\quad\quad \left. \widetilde{\beta}_{jk}(\bbr_k^*\bbD_{jk}^{-1}\bbr_j\bbr_j^*\bbD_{jk}^{-1}\bbr_k-n^{-1}\operatorname{tr}\bbT\bbD_{jk}^{-1}\bbr_j\bbr_j^*\bbD_{jk}^{-1})\rvert^2 \right)^{1/2} \\
			\nonumber \leq &Kn^{-1/2} \left(\sum_{k>j}(\mathbb{E}\lvert\beta_{jk}(\bbr_j^*\bbD_{jk}^{-1}\bbr_j-n^{-1}\operatorname{tr}\bbT\bbD_{jk}^{-1}) \rvert^4)^{1/2} \right. \\
			\nonumber &\quad\quad\quad\quad\quad\quad \left.(\mathbb{E}\lvert\widetilde{\beta}_{jk}(\bbr_k^*\bbD_{jk}^{-1}\bbr_j\bbr_j^*\bbD_{jk}^{-1}\bbr_k-n^{-1}\operatorname{tr}\bbT\bbD_{jk}^{-1}\bbr_j\bbr_j^*\bbD_{jk}^{-1})\rvert^4)^{1/2}\right)^{1/2} \\
			\nonumber\leq& Kn^{-3/2}.
		\end{align}
		From \eqref{Esum Yj(sumYk-Ykj)g'W_n(j)_1} and \eqref{Esum Yj(sumYk-Ykj)g'W_n(j)_2}, we have
		\begin{align*}
			\nonumber&\mathbb{E}\left| \mathbb{E}_{n}^{(j)}\left[(\mathbb{E}_j-\mathbb{E}_{j-1})\varepsilon_j \left(\sum_{k>j}(\mathbb{E}_k-\mathbb{E}_{k-1})(\bbr_k^*(\bbD_k^{-1}-\bbD_{kj}^{-1})\bbr_k-n^{-1}\operatorname{tr}\bbT(\bbD_k^{-1}-\bbD_{kj}^{-1}))\right)\right]\right| \leq K n^{-3/2},
		\end{align*}
		which leads to the conclusion of this lemma.
	\end{proof}

	\section{Technical lemmas}\label{Technical_lemmas}
	This appendix presents technical lemmas utilized throughout the paper.
	
	
	
	\begin{lemma}[\cite{BURKHOLDER73D}]\label{Burkholder_1}
		Let $\left\lbrace X_k\right\rbrace $ be a complex martingale difference sequence with respect to the increasing $\sigma$-field $\left\lbrace \mathcal{F}_k\right\rbrace $. Then, for $p>1$,
		$$\mathbb{E}\left| \sum X_k\right|^p\leq K_p\mathbb{E}\left(\sum \lvert X_k\rvert^2\right)^{p/2}.$$
	\end{lemma}
	\begin{lemma}\label{Burkholder_2}
		Let $\left\lbrace X_k\right\rbrace$ be a complex martingale difference sequence with respect to the increasing $\sigma$-field $\left\lbrace \mathcal{F}_k\right\rbrace $. Then, for $p\geq 2$,
		$$\mathbb{E}\left|\sum X_k\right|^p\leq K_p\left(\left(\sum \mathbb{E}\lvert X_k\rvert^2\right)^{p/2}+\mathbb{E}\sum \lvert X_k\rvert^p\right).$$
	\end{lemma}

	
		
	\begin{lemma}[(3.4) in \cite{BaiS98N}]\label{bounds about b_j and beta_j}
		For all $z\in\mathcal{C}_u$ or $\overline{\mathcal{C}_u}$,
		\begin{align*}
			& \lvert\beta_{j}(z)\rvert=\bigg\lvert\frac{1}{1+\bbr_{j}^{*} \bbD_{j}^{-1}(z) \bbr_{j}}\bigg\rvert\leq\frac{\lvert z\rvert}{\Im z}, \\
			& \lvert\widetilde{\beta}_{j}(z)\rvert=\bigg\lvert\frac{1}{1+n^{-1} \operatorname{tr} \bbT \bbD_{j}^{-1}(z)}\bigg\rvert\leq\frac{\lvert z\rvert}{\Im z}, \\
			&\lvert b_{j}(z)\rvert=\bigg\lvert\frac{1}{1+n^{-1} \mathbb{E} \operatorname{tr} \bbT \bbD_{j}^{-1}(z)}\bigg\rvert\leq\frac{\lvert z\rvert}{\Im z} .
		\end{align*}
	\end{lemma}
	The following uniform bounds, which hold when $\Xi_n^c$ (defined in \eqref{Xi_n}) occurs, are shown in Section 3 of \cite{BaiS04C}.
	\begin{lemma}\label{bound moment of b_j and beta_j}
		For all $z\in\mathcal{C}$, any $p>0$, $i, j=1,\dots,n$, we have
		\begin{align*}
			& \mathbb{E}\lvert \beta_j(z)\rvert^p \leq \lvert 1+Kx_r\rvert^p, \\
			& max(\lVert\bbD(z)\rVert^p,\lVert \bbD_j^{-1}(z)\rVert^p,\lVert \bbD_{ij}^{-1}(z)\rVert^p)\leq K , \\
			& \lvert b_j(z)\rvert\leq \frac{K}{1-Kn^{-1}}.
		\end{align*}
	\end{lemma}
	
	\begin{lemma}\label{bound spectral norm of zI-bj(z)T}
		For all $z\in\mathcal{C}$,
		\begin{align*}
			\lVert (z\bbI-\frac{n-1}{n}b_j(z)\bbT)^{-1}\rVert\leq K.
		\end{align*}
	\end{lemma}
	\begin{proof}
		For $z\in\mathcal{C}_u$ or $\overline{\mathcal{C}_u}$ , it is readily verified that for any real $t$, 
		\begin{align*}
			&  \left|1-\frac{t}{z(1+n^{-1}\mathbb{E} \operatorname{tr}\bbT\bbD_j^{-1}(z))}\right|^{-1}\leq \frac{\lvert z(1+n^{-1}\mathbb{E} \operatorname{tr}\bbT\bbD_j^{-1}(z))\rvert}{\Im z(1+n^{-1}\mathbb{E}\operatorname{tr} \bbT\bbD_j^{-1}(z))} \\
			\leq&\frac{\lvert z\rvert(1+p/(nv_0))}{v_0}.
		\end{align*}
		Thus, for $z\in\mathcal{C}_u$ or $\overline{\mathcal{C}_u}$, 
		\begin{align*}
			\lVert (z\bbI-\frac{n-1}{n}b_j(z)\bbT)^{-1}\rVert\leq \frac{1+p/(nv_0)}{v_0}.
		\end{align*}
		Let $x$ denote $\Re z$. Let us consider $x=x_l$ or $x_r$. Since $x$ is outside the support of $\underline{F}^{y,H}$, it follows from Theorem 4.1 in \cite{SilversteinC95A} that for any $t$ in the support of $H$ $\underline{s}(x)t+1\neq 0$. Choose any $t_0$ in the support of $H$. Since $\underline{s}(z)$ is continuous on $\mathcal{C}^{0}\equiv\left\lbrace x+iv:v\in \left[-v_0,v_0 \right] \right\rbrace $, there exists positive constant $\upsilon_1$ such that
		$$\inf_{z\in\mathcal{C}^0}\lvert\underline{s}(z)t_0+1\rvert\geq\upsilon_1.$$
		
		Due to (5.1) of \cite{BaiS04C} for any bounded subset $\mathcal{S}$ of $\mathcal{C}$, $\inf_{z\in \mathcal{S}}\lvert\underline{s}(z)\rvert>0$. There exists positive constant $\upsilon_2$ such that
		$$\inf_{z\in \mathcal{C}^0}\lvert \underline{s}(z)\rvert\geq \upsilon_2. $$
		Furthermore, notice that the definition of $x_l$ and $x_r$ in  Section \ref{Stieltjes transform}, for $z\in\mathcal{C}^0$, $\lvert z\rvert=\sqrt{x^2+v^2}\neq0$. Due to \eqref{b_j(z)+zs_n^0(z)} and $\underline{s}_n^0(z)\to\underline{s}(z)$, we have, for all large $n$, there exists positive constant $2\upsilon_3<\upsilon_1\upsilon_2$,
		\begin{align*}
			& \inf_{z\in\mathcal{C}^0}\lvert z(1+n^{-1}\mathbb{E}\operatorname{tr}\bbT\bbD_j^{-1}(z))-t_0\rvert
			\geq \inf_{z\in \mathcal{C}^0}\lvert-\frac{1}{\underline{s}(z)}-t_0\rvert-(\upsilon_1\upsilon_2-2\upsilon_3) \\
			=& \inf_{z\in \mathcal{C}^0}\lvert \underline{s}(z)\rvert^{-1}\lvert1+t_0\underline{s}(z)\rvert-(\upsilon_1\upsilon_2-2\upsilon_3) 
			\geq\upsilon_1\upsilon_2-(\upsilon_1\upsilon_2-2\upsilon_3)=2\upsilon_3.
		\end{align*}
		
		Using $H_p\stackrel{d}{\rightarrow} H$, for all large $n$, there exists an eigenvalue $\lambda^{\bbT}$ of $\bbT$ such that $\lvert\lambda^{\bbT}-t_0\rvert\leq\upsilon_3$. We have 
		\begin{align}\label{z(1+n-1EtrTDj-1z)}
			 \inf_{z\in \mathcal{C}^0}\lvert z(1+n^{-1}\mathbb{E}\operatorname{tr}\bbT\bbD_j^{-1}(z))-\lambda^{\bbT}\rvert \geq \inf_{z\in\mathcal{C}^0}\lvert z(1+n^{-1}\mathbb{E}\operatorname{tr}\bbT\bbD_j^{-1}(z))\rvert-\upsilon_3\geq \upsilon_3.
		\end{align}
		Therefore, due to \eqref{z(1+n-1EtrTDj-1z)} and Lemma \ref{bound moment of b_j and beta_j},  we have, for $z\in\mathcal{C}^0$,
		\begin{align*}
			  \left| 1-\frac{\lambda^{\bbT}}{z(1+n^{-1}\mathbb{E}\operatorname{tr}\bbT\bbD_j^{-1}(z))}\right|^{-1}= \frac{\lvert z(1+n^{-1}\mathbb{E}\operatorname{tr} \bbT\bbD_j^{-1}(z))\rvert}{\lvert z(1+n^{-1}\mathbb{E}\operatorname{tr}\bbT\bbD_j^{-1}(z))-\lambda^{\bbT}\rvert} 
			\leq K\lvert z\rvert.
		\end{align*} 
		Thus, we have, for $z\in\mathcal{C}^0$,
		$$\lVert (z\bbI-\frac{n-1}{n}b_j(z)\bbT)^{-1}\rVert\leq K.$$
		In summary, we can get the result of this lemma.
	\end{proof}
	
	
	
	\begin{lemma}[Lemma 2.7 in \cite{BaiS98N}]\label{QMTr}
		For $\bbX=(x_1,\dots,x_n)^{T}$ i.i.d. standardized (complex) entries, $\bbC$ $ n\times n$ matrix (complex),  we have, for any $p\geq2$,
		\begin{align*}
			\mathbb{E}\lvert \bbX^{*}\bbC\bbX-\operatorname{tr}\bbC\rvert^p\leq K_p((\mathbb{E}\lvert x_1\rvert^4\operatorname{tr}\bbC\bbC^{*})^{p/2}+\mathbb{E}\lvert x_1\rvert^{2p}\operatorname{tr}(\bbC\bbC^{*})^{p/2}).
		\end{align*}
	\end{lemma}
	\begin{lemma}\label{QMTr_1}
		For nonrandom Hermitian nonnegative definite $p\times p$ matrices $\bbA_l$, $l=1,\dots,k$,
		\begin{align}\label{QMTr_1result}
			\mathbb{E}\left| \prod_{l=1}^{k}(\bbr_1^{*}\bbA_l\bbr_1-n^{-1}\operatorname{tr}\bbT\bbA_l)\right|\leq K_k n^{-k/2}  \prod_{l=1}^{k} \lVert \bbA_l \rVert.
		\end{align}
	\end{lemma}
	\begin{proof}
		Recalling the truncation procedure from Section \ref{Truncation and normalization},  $\mathbb{E}\lvert X_{11}\rvert^{10}<\infty$, and Lemma \ref{QMTr}, we have, for all $p>1$,
		\begin{align}\label{QMTr_2}
			\mathbb{E} \lvert \bbr_1^{*}\bbA_l \bbr_1-n^{-1}\operatorname{tr}\bbT\bbA_l \rvert^p &\leq K_p n^{-p} [ (\mathbb{E}\lvert x_{11}\rvert^4 \operatorname{tr}\bbA_l\bbA_l^{*})^{p/2}+\mathbb{E}\lvert x_{11}\rvert^{2p}\operatorname{tr}(\bbA_l\bbA_l^{*})^{p/2}] \\
			\nonumber &\leq K_p n^{-p} \lVert \bbA_l\rVert^p \left[ n^{p/2} + (n^{1/4} \eta_n)^{(2p-10)\vee 0}n\right]  \\
			\nonumber &= K_p \lVert \bbA_l\rVert^p n^{-p/2}.
		\end{align}
		Then , (\ref{QMTr_1result}) can get from (\ref{QMTr_2}) and Holder's inequality.
	\end{proof}
	
	\begin{remark}
		Notice that, while $\bbA_l$ is stated as nonrandom in lemma \ref{QMTr_1}, the result extends to cases where $\bbA_l$ is random, provided that it is independent of $\bbr_1$. In such instances, the right-hand side of (\ref{QMTr_1result}) involves 
		$$K_k n^{-k/2} \mathbb{E} \prod_{l=1}^{k} \lVert \bbA_l \rVert .$$
	\end{remark}
	
	From the above lemma, we can further get the following lemmas.
	\begin{lemma}\label{Q_QMTr}
		For nonrandom Hermitian nonnegative definite $p\times p$ $\bbA_r$, $r=1,\dots,s$ and  $\bbB_l$, $l=1,\dots,k$, we  establish the following inequality:
		\begin{align*}
			&  \left| \mathbb{E} \left( \prod_{r=1}^{s}\bbr_1^{*}\bbA_r \bbr_1\prod_{l=1}^{k}(\bbr_1^{*}\bbB_l\bbr_1-n^{-1}\operatorname{tr}\bbT\bbB_l)\right) \right| \leq K n^{-k/2} \prod_{r=1}^{s}\lVert \bbA_r \rVert \prod_{l=1}^{k} \lVert \bbB_l \rVert, \quad s\geq0,k\geq0.
		\end{align*}
	\end{lemma}
	
	\begin{lemma}\label{quadratic form minus trace of qudratic form}
		For nonrandom Hermitian nonnegative definite $p\times p$ $\bbD$ with bounded spectral norm,
		\begin{align*}
			\mathbb{E}\lvert \bbr_1^{*}\bbD\bbr_2\bbr_2^{*}\bbD\bbr_1-n^{-1}\operatorname{tr}\bbT\bbD\bbr_2\bbr_2^{*}\bbD\rvert^{p}\leq Kn^{-p+((p/2-5/2)\vee 0)}\eta_n^{(2p-10)\vee 0}.
		\end{align*}
	\end{lemma}
	\begin{proof}
		Recalling the truncation steps in Section \ref{Truncation and normalization},  $\mathbb{E}\lvert x_{11}\rvert^{10}<\infty$, Lemma \ref{QMTr} and \ref{Q_QMTr}, we have, for all $p>1$,
		\begin{align*}
			& \mathbb{E}\lvert \bbr_1^{*}\bbD\bbr_2\bbr_2^{*}\bbD\bbr_1-n^{-1}\operatorname{tr}\bbT\bbD\bbr_2\bbr_2^{*}\bbD\rvert^{p} \\
			\leq& Kn^{-p}[(\mathbb{E}\lvert x_{11}\rvert^4\mathbb{E}\operatorname{tr}\bbD\bbr_2\bbr_2^{*}\bbD\bbD\bbr_2\bbr_2^{*}\bbD)+\mathbb{E}\lvert x_{11}\rvert^{2p}\mathbb{E}tr(\bbD\bbr_2\bbr_2^{*}\bbD)^p] \\
			\leq& n^{-p}[K+K(n^{1/4}\eta_n)^{(2p-10)\vee 0}\mathbb{E}(\bbr_2^{*}\bbD^{2}\bbr_2)^p] \\
			\leq& Kn^{-p+((p/2-5/2)\vee 0)}\eta_n^{(2p-10)\vee 0}.
		\end{align*}
		Finally, we get the result of this lemma.
	\end{proof}
	
	According to (3.2) of \cite{BaiS04C}, Lemma \ref{Q_QMTr} can be extended.
	\begin{lemma}\label{a(v)_quadratic_minus_trace}
		For $a(v)$ and $p\times p$ matrices $\bbB_l(v)$, $l=1,\dots,k$, which are independent of $\bbr_1$, both satisfy
		$$max(\lvert a(v)\rvert,\lVert \bbB_l(v)\rVert)\leq K\left(1+n^s I\left\lbrace \lVert \bbB_n\rVert\geq x_r or \lambda_{min}^{\widetilde{\bbB}}\leq x_l                                                                                \right\rbrace \right)$$
		for some positive $s$, with $\widetilde{\bbB}$ being $\bbB_n$ or $\bbB_n$ with one or two of $\bbr_j$'s removed. We get
		\begin{align*}
			\left|\mathbb{E}\left(a(v)\prod_{l=1}^{k}(\bbr_1^{*}\bbB_l(v)\bbr_1-n^{-1}\operatorname{tr}\bbT\bbB_l(v))\right)\right|\leq K n^{-k/2}.
		\end{align*}
	\end{lemma}
	The following two lemmas address the expectation of random quadratic forms.
	
	\begin{lemma}[(1.15) in \cite{BaiS04C}]\label{product of two quadratic minus trace of two matrices}
		For nonrandom $n\times n$ $\bbA=(a_{ij})$ and $\bbB=(b_{ij})$,
		\begin{align*}
			  \mathbb{E}(\bbX_{\cdot 1}^{*}\bbA\bbX_{\cdot 1}-\operatorname{tr}\bbA)(\bbX_{\cdot 1}^{*}\bbB\bbX_{\cdot 1}-\operatorname{tr}\bbB)  =(\mathbb{E}\lvert x_{11}\rvert^4-\lvert\mathbb{E} x_{11}^2\rvert^2-2)\operatorname{tr}(\bbA\circ \bbB)+\lvert\mathbb{E}x_{11}^2\rvert^2 \operatorname{tr} \bbA\bbB^{T}+\operatorname{tr}\bbA\bbB
		\end{align*}
		where $\circ$ is the Hadamard product of two matrices.
	\end{lemma}

	\begin{lemma}\label{the lemma of the three product of quadratic minus trace}
		For nonrandom  nonnegative $n\times n$ $\bbA=(a_{ij})$, $\bbB=(b_{ij})$ and $\bbC=(c_{ij})$, there are bounded spectral norm.   For $j=1,\dots,n$, we have
		\begin{align*}
			\lvert\mathbb{E}(\bbX_{\cdot j}^*\bbA\bbX_{\cdot j}-\operatorname{tr}\bbA)(\bbX_{\cdot j}^*\bbB\bbX_{\cdot j}-\operatorname{tr}\bbB)(\bbX_{\cdot j}^*\bbC\bbX_{\cdot j}-\operatorname{tr}\bbC)\rvert\leq Kn.
		\end{align*}
	\end{lemma}
	\begin{proof}
		Using the expression
		\begin{align*}
			\bbX_{\cdot j}^*\bbA\bbX_{\cdot j}-\operatorname{tr}\bbA=\sum_{p=1}^{n}a_{pp}(\lvert x_{pj}\rvert^2-1)+\sum_{p=1}^{n}\sum_{q\neq p}a_{pq}\bar{x}_{pj}x_{qj},
		\end{align*}
		we can write
		\begin{align}\label{the three product of quadratic minus trace}
			& \mathbb{E}(\bbX_{\cdot j}^*\bbA\bbX_{\cdot j}-\operatorname{tr}\bbA)(\bbX_{\cdot j}^*\bbB\bbX_{\cdot j}-\operatorname{tr}\bbB)(\bbX_{\cdot j}^*\bbC\bbX_{\cdot j}-\operatorname{tr}\bbC) \\
			\nonumber =&\mathbb{E}[(\sum_{p=1}^{n}a_{pp}(\lvert x_{pj}\rvert^2-1)+\sum_{p=1}^{n}\sum_{q\neq p}a_{pq}\bar{x}_{pj}x_{qj})(\sum_{r=1}^{n}b_{rr}(\lvert x_{rj}\rvert^2-1)+\sum_{r=1}^{n}\sum_{s\neq r}b_{rs}\bar{x}_{rj}x_{sj}) \\
			\nonumber &\quad\quad(\sum_{u=1}^{n}c_{uu}(\lvert x_{uj}\rvert^2-1)+\sum_{u=1}^{n}\sum_{v\neq u}c_{uv}\bar{x}_{uj}x_{vj})] \\
			\nonumber =&\mathbb{E}[(\sum_{p=1}^{n}a_{pp}(\lvert x_{pj}\rvert^2-1))(\sum_{r=1}^{n}b_{rr}(\lvert x_{rj}\rvert^2-1))(\sum_{u=1}^{n}c_{uu}(\lvert x_{uj}\rvert^2-1))] \\
			\nonumber &+\mathbb{E}[(\sum_{p=1}^{n}\sum_{q\neq p}a_{pq}\bar{x}_{pj}x_{qj})(\sum_{r=1}^{n}\sum_{s\neq r}b_{rs}\bar{x}_{rj}x_{sj})(\sum_{u=1}^{n}\sum_{v\neq u}c_{uv}\bar{x}_{uj}x_{vj})] \\
			\nonumber &+\mathbb{E}[(\sum_{p=1}^{n}a_{pp}(\lvert x_{pj}\rvert^2-1))(\sum_{r=1}^{n}b_{rr}(\lvert x_{rj}\rvert^2-1))(\sum_{u=1}^{n}\sum_{v\neq u}c_{uv}\bar{x}_{uj}x_{vj})] \\
			\nonumber&+\mathbb{E}[(\sum_{p=1}^{n}a_{pp}(\lvert x_{pj}\rvert^2-1))(\sum_{r=1}^{n}\sum_{s\neq r}b_{rs}\bar{x}_{rj}x_{sj})(\sum_{u=1}^{n}c_{uu}(\lvert x_{uj}\rvert^2-1))] \\
			\nonumber&+\mathbb{E}[(\sum_{p=1}^{n}a_{pp}(\lvert x_{pj}\rvert^2-1))(\sum_{r=1}^{n}\sum_{s\neq r}b_{rs}\bar{x}_{rj}x_{sj})(\sum_{u=1}^{n}\sum_{v\neq u}c_{uv}\bar{x}_{uj}x_{vj})] \\
			\nonumber&+\mathbb{E}[(\sum_{p=1}^{n}\sum_{q\neq p}a_{pq}\bar{x}_{pj}x_{qj})(\sum_{r=1}^{n}b_{rr}(\lvert x_{rj}\rvert^2-1))(\sum_{u=1}^{n}c_{uu}(\lvert x_{uj}\rvert^2-1))] \\
			\nonumber&+\mathbb{E}[(\sum_{p=1}^{n}\sum_{q\neq p}a_{pq}\bar{x}_{pj}x_{qj})(\sum_{r=1}^{n}b_{rr}(\lvert x_{rj}\rvert^2-1))(\sum_{u=1}^{n}\sum_{v\neq u}c_{uv}\bar{x}_{uj}x_{vj})] \\
			\nonumber&+\mathbb{E}[(\sum_{p=1}^{n}\sum_{q\neq p}a_{pq}\bar{x}_{pj}x_{qj})(\sum_{r=1}^{n}\sum_{s\neq r}b_{rs}\bar{x}_{rj}x_{sj})(\sum_{u=1}^{n}c_{uu}(\lvert x_{uj}\rvert^2-1))].
		\end{align}
		Notice that $\mathbb{E}\lvert x_{11}\rvert^2=1$ and  $\mathbb{E}x_{11}=0$, we could calculate the above terms separately.
		With regards to the first term, we can write
		\begin{align}\label{the first term of the three product}
			& \mathbb{E}[(\sum_{p=1}^{n}a_{pp}(\lvert x_{pj}\rvert^2-1))(\sum_{r=1}^{n}b_{rr}(\lvert x_{rj}\rvert^2-1))(\sum_{u=1}^{n}c_{uu}(\lvert x_{uj}\rvert^2-1))] \\
			\nonumber=&\sum_{p=1}^{n}a_{pp}b_{pp}c_{pp}\mathbb{E}(\lvert x_{pj}\rvert^2-1)^3+\sum_{p=1}^{n}\sum_{q\neq p}a_{pp}b_{pp}c_{qq}\mathbb{E}[(\lvert x_{pj}\rvert^2-1)^2(\lvert x_{qj}\rvert^2-1)] \\
			\nonumber&\quad+\sum_{p=1}^{n}\sum_{q\neq p}a_{pp}b_{qq}c_{pp}\mathbb{E}[(\lvert x_{pj}\rvert^2-1)^2(\lvert x_{qj}\rvert^2-1)] \\
			\nonumber&\quad+\sum_{p=1}^{n}\sum_{q\neq p}a_{qq}b_{pp}c_{pp}\mathbb{E}[(\lvert x_{pj}\rvert^2-1)^2(\lvert x_{qj}\rvert^2-1)] \\
			\nonumber=&\sum_{p=1}^{n}a_{pp}b_{pp}c_{pp}\mathbb{E}(\lvert x_{pj}\rvert^2-1)^3=\sum_{p=1}^{n}a_{pp}b_{pp}c_{pp}[\mathbb{E}\lvert x_{pj}\rvert^6-3\mathbb{E}\lvert x_{pj}\rvert^4+2].
		\end{align}
		With regards to the third term, 
		\begin{align}\label{the third term of the product of three quadratic}
			& \mathbb{E}[(\sum_{p=1}^{n}a_{pp}(\lvert x_{pj}\rvert^2-1))(\sum_{r=1}^{n}b_{rr}(\lvert x_{rj}\rvert^2-1))(\sum_{u=1}^{n}\sum_{v\neq u}c_{uv}\bar{x}_{uj}x_{vj})] \\
			\nonumber=&\sum_{p=1}^{n}\sum_{q\neq p}a_{pp}b_{qq}c_{pq}\mathbb{E}[(\lvert x_{pj}\rvert^2-1)\bar{x}_{pj}(\lvert x_{qj}\rvert^2-1)x_{qj}] \\
			\nonumber&\quad +\sum_{p=1}^{n}\sum_{q\neq p} a_{pp}b_{qq}c_{qp}\mathbb{E}[(\lvert x_{pj}\rvert^2-1)x_{pj}(\lvert x_{qj}\rvert^2-1)\bar{x}_{qj}] \\
			\nonumber=&\sum_{p=1}^{n}\sum_{q\neq p}a_{pp}b_{qq}c_{pq}\mathbb{E}[\lvert x_{pj}\rvert^2\bar{x}_{pj}]\mathbb{E}[\lvert x_{qj}\rvert^2x_{qj}]  \\
			\nonumber&\quad +\sum_{p=1}^{n}\sum_{q\neq p} a_{pp}b_{qq}c_{qp}\mathbb{E}[\lvert x_{pj}\rvert^2x_{pj}]\mathbb{E}[\lvert x_{qj}\rvert^2\bar{x}_{qj}].
		\end{align}
		Similarly as \eqref{the third term of the product of three quadratic}, we can get
		\begin{align}
			& \mathbb{E}[(\sum_{p=1}^{n}a_{pp}(\lvert x_{pj}\rvert^2-1))(\sum_{r=1}^{n}\sum_{s\neq r}b_{rs}\bar{x}_{rj}x_{sj})(\sum_{u=1}^{n}c_{uu}(\lvert x_{uj}\rvert^2-1))] \\
			\nonumber=& \sum_{p=1}^{n}\sum_{q\neq p}a_{pp}b_{pq}c_{qq}\mathbb{E}[\lvert x_{pj}\rvert^2\bar{x}_{pj}]\mathbb{E}[\lvert x_{qj}\rvert^2x_{qj}]  \\
			\nonumber&\quad +\sum_{p=1}^{n}\sum_{q\neq p} a_{pp}b_{qp}c_{qq}\mathbb{E}[\lvert x_{pj}\rvert^2x_{pj}]\mathbb{E}[\lvert x_{qj}\rvert^2\bar{x}_{qj}],
		\end{align}
		\begin{align}
			& \mathbb{E}[(\sum_{p=1}^{n}\sum_{q\neq p}a_{pq}\bar{x}_{pj}x_{qj})(\sum_{r=1}^{n}b_{rr}(\lvert x_{rj}\rvert^2-1))(\sum_{u=1}^{n}c_{uu}(\lvert x_{uj}\rvert^2-1))] \\
			\nonumber=& \sum_{p=1}^{n}\sum_{q\neq p}a_{pq}b_{pp}c_{qq}\mathbb{E}[\lvert x_{pj}\rvert^2\bar{x}_{pj}]\mathbb{E}[\lvert x_{qj}\rvert^2x_{qj}]  \\
			\nonumber&\quad +\sum_{p=1}^{n}\sum_{q\neq p} a_{qp}b_{pp}c_{qq}\mathbb{E}[\lvert x_{pj}\rvert^2x_{pj}]\mathbb{E}[\lvert x_{qj}\rvert^2\bar{x}_{qj}].
		\end{align}
		With regards to the fifth term, 
		\begin{align}\label{the fifth term of the product of three quadratic}
			& \mathbb{E}[(\sum_{p=1}^{n}a_{pp}(\lvert x_{pj}\rvert^2-1))(\sum_{r=1}^{n}\sum_{s\neq r}b_{rs}\bar{x}_{rj}x_{sj})(\sum_{u=1}^{n}\sum_{v\neq u}c_{uv}\bar{x}_{uj}x_{vj})]  \\
			\nonumber=&\sum_{p=1}^{n}\sum_{q\neq p} a_{pp}b_{pq}c_{pq}\mathbb{E}[(\lvert x_{pj}\rvert^2-1)\bar{x}_{pj}^2]\mathbb{E}x_{qj}^2  +\sum_{p=1}^{n}\sum_{q\neq p}a_{pp}b_{pq}c_{qp}\mathbb{E}[(\lvert x_{pj}\rvert^2-1)\lvert x_{pj}\rvert^2]\mathbb{E}\lvert x_{qj}\rvert^2  \\
			\nonumber &\quad +\sum_{p=1}^{n}\sum_{q\neq p}a_{pp}b_{qp}c_{pq}\mathbb{E}[(\lvert x_{pj}\rvert^2-1)\lvert x_{pj}\rvert^2]\mathbb{E}\lvert x_{qj}\rvert^2 +\sum_{p=1}^{n}\sum_{q\neq p}a_{pp}b_{qp}c_{qp}\mathbb{E}x_{pj}^2\mathbb{E}\bar{x}_{qj}^2 \\
			\nonumber=&\sum_{p=1}^{n}\sum_{q\neq p} a_{pp}b_{pq}c_{pq}\mathbb{E}[(\lvert x_{pj}\rvert^2-1)\bar{x}_{pj}^2]\mathbb{E}x_{qj}^2 +\sum_{p=1}^{n}\sum_{q\neq p}a_{pp}b_{pq}c_{qp}\mathbb{E}[(\lvert x_{pj}\rvert^2-1)\lvert x_{pj}\rvert^2] \\
			\nonumber &\quad +\sum_{p=1}^{n}\sum_{q\neq p}a_{pp}b_{qp}c_{pq}\mathbb{E}[(\lvert x_{pj}\rvert^2-1)\lvert x_{pj}\rvert^2]  +\sum_{p=1}^{n}\sum_{q\neq p}a_{pp}b_{qp}c_{qp}\mathbb{E}[(\lvert x_{pj}\rvert^2-1)x_{pj}^2]\mathbb{E}\bar{x}_{qj}^2.
		\end{align}
		Similarly as \eqref{the fifth term of the product of three quadratic}, we can get
		\begin{align}
			& \mathbb{E}[(\sum_{p=1}^{n}\sum_{q\neq p}a_{pq}\bar{x}_{pj}x_{qj})(\sum_{r=1}^{n}b_{rr}(\lvert x_{rj}\rvert^2-1))(\sum_{u=1}^{n}\sum_{v\neq u}c_{uv}\bar{x}_{uj}x_{vj})] \\
			\nonumber =&\sum_{p=1}^{n}\sum_{q\neq p} a_{pq}b_{pp}c_{pq}\mathbb{E}[(\lvert x_{pj}\rvert^2-1)\bar{x}_{pj}^2]\mathbb{E}x_{qj}^2  +\sum_{p=1}^{n}\sum_{q\neq p}a_{pq}b_{pp}c_{qp}\mathbb{E}[(\lvert x_{pj}\rvert^2-1)\lvert x_{pj}\rvert^2] \\
			\nonumber &\quad +\sum_{p=1}^{n}\sum_{q\neq p}a_{qp}b_{pp}c_{pq}\mathbb{E}[(\lvert x_{pj}\rvert^2-1)\lvert x_{pj}\rvert^2]  +\sum_{p=1}^{n}\sum_{q\neq p}a_{q	p}b_{pp}c_{qp}\mathbb{E}[(\lvert x_{pj}\rvert^2-1)x_{pj}^2]\mathbb{E}\bar{x}_{qj}^2,
		\end{align}
		\begin{align}
			&\mathbb{E}[(\sum_{p=1}^{n}\sum_{q\neq p}a_{pq}\bar{x}_{pj}x_{qj})(\sum_{r=1}^{n}\sum_{s\neq r}b_{rs}\bar{x}_{rj}x_{sj})(\sum_{u=1}^{n}c_{uu}(\lvert x_{uj}\rvert^2-1))] \\
			\nonumber =&\sum_{p=1}^{n}\sum_{q\neq p} a_{pq}b_{pq}c_{pp}\mathbb{E}[(\lvert x_{pj}\rvert^2-1)\bar{x}_{pj}^2]\mathbb{E}x_{qj}^2  +\sum_{p=1}^{n}\sum_{q\neq p}a_{pq}b_{qp}c_{pp}\mathbb{E}[(\lvert x_{pj}\rvert^2-1)\lvert x_{pj}\rvert^2] \\
			\nonumber &\quad +\sum_{p=1}^{n}\sum_{q\neq p}a_{qp}b_{pq}c_{pp}\mathbb{E}[(\lvert x_{pj}\rvert^2-1)\lvert x_{pj}\rvert^2]  +\sum_{p=1}^{n}\sum_{q\neq p}a_{q	p}b_{qp}c_{pp}\mathbb{E}[(\lvert x_{pj}\rvert^2-1)x_{pj}^2]\mathbb{E}\bar{x}_{qj}^2.
		\end{align}
		Finally, we only need to consider the second term, 
		\begin{align}\label{the second term of the three product}
			& \mathbb{E}[(\sum_{p=1}^{n}\sum_{q\neq p}a_{pq}\bar{x}_{pj}x_{qj})(\sum_{r=1}^{n}\sum_{s\neq r}b_{rs}\bar{x}_{rj}x_{sj})(\sum_{u=1}^{n}\sum_{v\neq u}c_{uv}\bar{x}_{uj}x_{vj})] \\
			\nonumber =&\sum_{p=1}^{n}\sum_{q\neq p}a_{pq}b_{pq}c_{pq}\mathbb{E}\bar{x}_{pj}^3\mathbb{E}x_{qj}^3+\sum_{p=1}^{n}\sum_{q\neq p}a_{pq}b_{qp}c_{pq}\mathbb{E}[\bar{x}_{pj}\lvert x_{pj}\rvert^2]\mathbb{E}[\bar{x}_{qj}\lvert x_{qj}\rvert^2] \\
			\nonumber & +\sum_{p=1}^{n}\sum_{q\neq p}a_{pq}b_{pq}c_{qp}\mathbb{E}[\bar{x}_{pj}\lvert x_{pj}\rvert^2]\mathbb{E}[\bar{x}_{qj}\lvert x_{qj}\rvert^2] +\sum_{p=1}^{n}\sum_{q\neq p}\sum_{r\neq p, q} a_{pq}b_{rq}c_{rp}\mathbb{E}[\lvert x_{pj}\rvert^2x_{qj}^2\bar{x}_{rj}^2] \\
			\nonumber &+ \sum_{p=1}^{n}\sum_{q\neq p}\sum_{r\neq p, q}a_{pq}b_{rp}c_{rq}\mathbb{E}[\lvert x_{pj}\rvert^2x_{qj}^2\bar{x}_{rj}^2]  + \sum_{p=1}^{n}\sum_{q\neq p}\sum_{r\neq p, q}a_{pr}b_{rq}c_{qp}\mathbb{E}[\lvert x_{pj}\rvert^2\lvert x_{qj}\rvert^2\lvert x_{rj}\rvert^2]  \\
			\nonumber & +\sum_{p=1}^{n}\sum_{q\neq p}\sum_{r\neq p, q}a_{pr}b_{qp}c_{rq}\mathbb{E}[\lvert x_{pj}\rvert^2\lvert x_{qj}\rvert^2\lvert x_{rj}\rvert^2] + \sum_{p=1}^{n}\sum_{q\neq p}\sum_{r\neq p, q}a_{pq}b_{qr}c_{rp}\mathbb{E}[\lvert x_{pj}\rvert^2\lvert x_{qj}\rvert^2\lvert x_{rj}\rvert^2] \\
			\nonumber &+ \sum_{p=1}^{n}\sum_{q\neq p}\sum_{r\neq p, q}a_{pq}b_{rp}c_{qr}\mathbb{E}[\lvert x_{pj}\rvert^2\lvert x_{qj}\rvert^2\lvert x_{rj}\rvert^2]  +\sum_{p=1}^{n}\sum_{q\neq p}\sum_{r\neq p, q}a_{pq}b_{pr}c_{rq}\mathbb{E}[\bar{x}_{pj}^2 x_{qj}^2\lvert x_{rj}\rvert^2]  \\
			\nonumber & +\sum_{p=1}^{n}\sum_{q\neq p}\sum_{r\neq p, q}a_{pq}b_{rq}c_{pr}\mathbb{E}[\bar{x}_{pj}^2 x_{qj}^2\lvert x_{rj}\rvert^2] +\sum_{p=1}^{n}\sum_{q\neq p}\sum_{r\neq p, q}a_{pr}b_{qr}c_{qp}\mathbb{E}[\lvert x_{pj}\rvert^2 \bar{x}_{qj}^2 x_{rj}^2] \\
			\nonumber & +\sum_{p=1}^{n}\sum_{q\neq p}\sum_{r\neq p, q}a_{pr}b_{qp}c_{qr}\mathbb{E}[\lvert x_{pj}\rvert^2\bar{x}_{qj}^2x_{rj}^2]  + \sum_{p=1}^{n}\sum_{q\neq p}\sum_{r\neq p, q}a_{pr}b_{pq}c_{rq}\mathbb{E}[\bar{x}_{pj}^2x_{qj}^2\lvert x_{rj}\rvert^2] \\
			\nonumber &+ \sum_{p=1}^{n}\sum_{q\neq p}\sum_{r\neq p, q}a_{pr}b_{rq}c_{pq}\mathbb{E}[\bar{x}_{pj}^2x_{qj}^2\lvert x_{rj}\rvert^2]   +\sum_{p=1}^{n}\sum_{q\neq p}\sum_{r\neq p, q}a_{pq}b_{qr}c_{pr}\mathbb{E}[\bar{x}_{pj}^2 \lvert x_{qj}\rvert^2 x_{rj}^2] \\
			\nonumber &+\sum_{p=1}^{n}\sum_{q\neq p}\sum_{r\neq p, q}a_{pq}b_{pr}c_{qr}\mathbb{E}[\bar{x}_{pj}^2 \lvert x_{qj}\rvert^2 x_{rj}^2]  +\sum_{p=1}^{n}\sum_{q\neq p}\sum_{r\neq p, q}a_{pr}b_{pq}c_{qr}\mathbb{E}[\bar{x}_{pj}^2 \lvert x_{qj}\rvert^2 x_{rj}^2] \\
			\nonumber & +\sum_{p=1}^{n}\sum_{q\neq p}\sum_{r\neq p, q}a_{pr}b_{qr}c_{pq}\mathbb{E}[\bar{x}_{pj}^2 \lvert x_{qj}\rvert^2 x_{rj}^2].
		\end{align}
		Let $\bbA_{\mathrm{d}}$ represent the matrix consisting of the diagonal elements of $A$. Similarly we can define $\bbB_{\mathrm{d}}$ and $\bbC_{\mathrm{d}}$. Let $\bba_{\mathrm{d}}$ denote $n\times 1$ vector consisting of the diagonal elements of $A$. Similarly we can define $\bbb_{\mathrm{d}}$ and $\bbc_{\mathrm{d}}$. We will give some examples to show how to transform the above summation terms.
		\begin{align*}
			& \sum_{p=1}^{n}a_{pp}b_{pp}c_{pp}	=\operatorname{tr}\bbA_{\mathrm{d}}\bbB_{\mathrm{d}}\bbC_{\mathrm{d}}, \\
			&\sum_{p=1}^{n}\sum_{q\neq p} a_{pp}b_{pq}c_{qq}	=\sum_{p=1}^{n}\sum_{q=1}^{n} a_{pp}b_{pq}c_{qq}-\sum_{p=1}^{n}a_{pp}b_{pp}c_{pp} =\bba_{\mathrm{d}}^{T}\bbB\bbc_{\mathrm{d}}-\operatorname{tr}\bbA_{\mathrm{d}}\bbB_{\mathrm{d}}\bbC_{\mathrm{d}}, \\
			& \sum_{p=1}^{n}\sum_{q\neq p} a_{pp}b_{pq}c_{pq}=\sum_{p=1}^{n}\sum_{q=1}^{n} a_{pp}b_{pq}c_{pq}-\sum_{p=1}^{n}a_{pp}b_{pp}c_{pp}=\operatorname{tr}\bbA_{\mathrm{d}}\bbB \bbC^{T}-\operatorname{tr}\bbA_{\mathrm{d}}\bbB_{\mathrm{d}}\bbC_{\mathrm{d}}, \\
			& \sum_{p=1}^{n}\sum_{q\neq p}a_{pq}b_{pq}c_{qp}=\sum_{p=1}^{n}\sum_{q=1}^{n}a_{pq}b_{pq}c_{qp}-\sum_{p=1}^{n}a_{pp}b_{pp}c_{pp} =\operatorname{tr}(\bbA\circ \bbB)\bbC-\operatorname{tr}\bbA_{\mathrm{d}}\bbB_{\mathrm{d}}\bbC_{\mathrm{d}},  \\
			&\sum_{p=1}^{n}\sum_{q\neq p}\sum_{r\neq p,q}a_{pq}b_{qr}c_{rp}=\sum_{p=1}^{n}\sum_{q=1}^{n}\sum_{r=1}^{n}a_{pq}b_{qr}c_{rp}-\sum_{p=1}^{n}\sum_{q=1}^{n}a_{pq}b_{qq}c_{qp}-\sum_{p=1}^{n}\sum_{q=1}^{n}a_{pq}b_{qp}c_{pp} \\
			&\quad\quad\quad\quad\quad\quad\quad\quad\quad\quad-\sum_{p=1}^{n}\sum_{q=1}^{n}a_{pp}b_{pq}c_{qp}+2\sum_{p=1}^{n}a_{pp}b_{pp}c_{pp} \\
			&\quad\quad\quad\quad\quad\quad\quad = \operatorname{tr} \bbA\bbB\bbC-\operatorname{tr}\bbA\bbB_{\mathrm{d}}\bbC-\operatorname{tr}\bbC_{\mathrm{d}}\bbA\bbB-\operatorname{tr}\bbA_{\mathrm{d}}\bbB\bbC+2\operatorname{tr}\bbA_{\mathrm{d}}\bbB_{\mathrm{d}}\bbC_{\mathrm{d}}, \\
			& \sum_{p=1}^{n}\sum_{q\neq p}\sum_{r\neq p,q}a_{pq}b_{rp}c_{rq}= \sum_{p=1}^{n}\sum_{q=1}^{n}\sum_{r=1}^{n} a_{pq}c_{rq}b_{rp}-\sum_{p=1}^{n}\sum_{q=1 }^{n}a_{pq}b_{qp}c_{qq}-\sum_{p=1}^{n}\sum_{q=1}^{n}a_{pq}b_{pp}c_{pq} \\
			&\quad\quad\quad\quad\quad\quad\quad\quad\quad\quad-\sum_{p=1}^{n}\sum_{q=1}^{n}a_{pp}b_{qp}c_{qp}+2\sum_{p=1}^{n}a_{pp}b_{pp}c_{pp} \\
			&\quad\quad\quad\quad\quad\quad\quad =\operatorname{tr} \bbA\bbC^{T}\bbB -\operatorname{tr}\bbA\bbB\bbC_\mathrm{d}-\operatorname{tr}\bbB_{\mathrm{d}}\bbA\bbC^{T}-\operatorname{tr}\bbA_{\mathrm{d}}\bbB^{T}\bbC+2\operatorname{tr}\bbA_{\mathrm{d}}\bbB_{\mathrm{d}}\bbC_{\mathrm{d}}.
		\end{align*}
		Thus, due to the moment assumptions in Theorem \ref{main theorem}, from \eqref{the first term of the three product}--\eqref{the second term of the three product}, we have
		\begin{align*}
			& \eqref{the three product of quadratic minus trace} = K\operatorname{tr}\bbA_{\mathrm{d}}\bbB_{\mathrm{d}}\bbC_{\mathrm{d}}+K\bba_\mathrm{d}^{T}\bbC\bbb_\mathrm{d}+K\bbb_\mathrm{d}^{T}\bbC\bba_\mathrm{d}+K\bba_\mathrm{d}^{T}\bbB\bbc_\mathrm{d} \\
			&\quad\quad\quad +K\bbc_\mathrm{d}^{T}\bbB\bba_\mathrm{d}+K\bbb_\mathrm{d}^{T}\bbA\bbc_\mathrm{d}+K\bbc_\mathrm{d}^{T}\bbA\bbb_\mathrm{d}+K\operatorname{tr}\bbA_\mathrm{d}\bbB\bbC^{T} \\
			&\quad\quad\quad +K\operatorname{tr}\bbA_\mathrm{d}\bbB\bbC+K\operatorname{tr}\bbA_\mathrm{d}\bbC\bbB+K\operatorname{tr}\bbA_\mathrm{d}\bbB^{T}\bbC+K\operatorname{tr}\bbA\bbC^{T}\bbB_\mathrm{d} \\
			&\quad\quad\quad +K\operatorname{tr}\bbA\bbC\bbB_\mathrm{d}+K\operatorname{tr}\bbA\bbB_\mathrm{d}\bbC+ K\operatorname{tr}\bbA\bbB^{T}\bbC_\mathrm{d}  \\
			&\quad\quad\quad+K\operatorname{tr}\bbA\bbB\bbC_\mathrm{d}+K\operatorname{tr}\bbA\bbC_\mathrm{d}\bbB+K\operatorname{tr}(\bbA\circ \bbB)\bbC^{T} \\
			&\quad\quad\quad+K\operatorname{tr}(\bbA\circ \bbC)\bbB+K\operatorname{tr}(\bbA\circ \bbB)\bbC+K \operatorname{tr}\bbA\bbB^{T}\bbC+K\operatorname{tr}\bbA\bbC^{T}\bbB \\
			&\quad\quad\quad+K\operatorname{tr}\bbA\bbB\bbC +K\operatorname{tr}\bbA\bbC\bbB+K \operatorname{tr}\bbA\bbC^{T}\bbB^{T}+K\operatorname{tr}\bbA\bbB^{T}\bbC^{T} \\
			&\quad\quad\quad+K \operatorname{tr}\bbA^{T}\bbB\bbC^{T}+K\operatorname{tr}\bbA^{T}\bbC\bbB^{T}+K\operatorname{tr}\bbA\bbB\bbC^{T}+K\operatorname{tr}\bbA\bbC\bbB^{T} \\
			&\quad\quad\quad +K \operatorname{tr}\bbA^{T}\bbB\bbC+K \operatorname{tr}\bbA^{T}\bbC\bbB.
		\end{align*}
		Due to $\bbA$, $\bbB$ and $\bbC$ with bounded spectral norm and Theorem A.19 in \cite{BaiS10S}, we can get 
		\begin{align*}
			\lvert\mathbb{E}(\bbX_{\cdot j}^*\bbA\bbX_{\cdot j}-\operatorname{tr}\bbA)(\bbX_{\cdot j}^*\bbB\bbX_{\cdot j}-\operatorname{tr}\bbB)(\bbX_{\cdot j}^*\bbC\bbX_{\cdot j}-\operatorname{tr}\bbC)\rvert\leq Kn.
		\end{align*}
		Eventually the conclusion of this lemma is obtained.
	\end{proof}

	\begin{lemma}\label{E|trTDj-1-EtrTDj-1|}
		For all $z\in\mathcal{C}$ and any $m\geq 1$, 
		\begin{align*}
			\mathbb{E}\lvert \operatorname{tr}\bbT\bbD_j^{-1}(z)-\mathbb{E}\operatorname{tr}\bbT\bbD_j^{-1}(z)\rvert^m\leq K.
		\end{align*}
	\end{lemma}
	\begin{proof}
		Due to the martingale difference decomposition, \eqref{D_j^{-1}-D_kj^{-1}},  Lemmas \ref{Burkholder_2}, \ref{bound moment of b_j and beta_j} and \ref{a(v)_quadratic_minus_trace}, we write for any $m>1$,
		\begin{align*}
			&\mathbb{E}\lvert \operatorname{tr}\bbT\bbD_j^{-1}(z)-\mathbb{E}\operatorname{tr}\bbT\bbD_j^{-1}(z)\rvert^m =\mathbb{E}\lvert \sum_{i=1}^{n}(\mathbb{E}_i-\mathbb{E}_{i-1})\operatorname{tr} \bbT\bbD_j^{-1}(z)\rvert^m \\
			=&\mathbb{E}\lvert \sum_{i=1}^{n}(\mathbb{E}_i-\mathbb{E}_{i-1})\operatorname{tr} \bbT(\bbD_j^{-1}(z)-\bbD_{ij}^{-1}(z))\rvert^m 
			=\mathbb{E}\lvert \sum_{i=1}^{n} (\mathbb{E}_i-\mathbb{E}_{i-1})\beta_{ij}(z)\bbr_i^{*}\bbD_{ij}^{-1}(z)\bbT\bbD_{ij}^{-1}(z)\bbr_i\rvert^m\\
			\leq& K \left[\mathbb{E}\sum_{i=1}^{n}\lvert (\mathbb{E}_i-\mathbb{E}_{i-1})\beta_{ij}(z)\bbr_i^{*}\bbD_{ij}^{-1}(z)\bbT\bbD_{ij}^{-1}(z)\bbr_i\rvert^m\right.  \\
			& \quad\quad\quad\quad \left.  +(\sum_{i=1}^{n}\mathbb{E}\lvert (\mathbb{E}_i-\mathbb{E}_{i-1})\beta_{ij}(z)\bbr_i^{*}\bbD_{ij}^{-1}(z)\bbT\bbD_{ij}^{-1}(z)\bbr_i\rvert^2)^{m/2}\right] \\
			\leq & K \left[  \sum_{i=1}^{n}\mathbb{E}\lvert(\mathbb{E}_i-\mathbb{E}_{i-1})[(\beta_{ij}(z)-\widetilde{\beta}_{ij}(z))\bbr_i^{*}\bbD_{ij}^{-1}(z)\bbT\bbD_{ij}^{-1}(z)\bbr_i \right.  +\widetilde{\beta}_{ij}(z)\bbr_i^{*}\bbD_{ij}^{-1}(z)\bbT\bbD_{ij}^{-1}(z)\bbr_i]\rvert^m  \\
			&\quad+(\sum_{i=1}^{n}\mathbb{E}\lvert(\mathbb{E}_i-\mathbb{E}_{i-1}) [(\beta_{ij}(z)-\widetilde{\beta}_{ij}(z))\bbr_i^{*}\bbD_{ij}^{-1}(z)\bbT\bbD_{ij}^{-1}(z)\bbr_i \left. +\widetilde{\beta}_{ij}(z)\bbr_i^{*}\bbD_{ij}^{-1}(z)\bbT\bbD_{ij}^{-1}(z)\bbr_i ]\rvert^2)^{m/2} \right] \\
			\leq& K\left[ \sum_{i=1}^{n}[\mathbb{E}\lvert \beta_{ij}(z)\widetilde{\beta}_{ij}(z)(\bbr_i^{*}\bbD_{ij}^{-1}(z)\bbr_i-n^{-1}\operatorname{tr}\bbT\bbD_{ij}^{-1}(z))\bbr_i^{*}\bbD_{ij}^{-1}(z)\bbT\bbD_{ij}^{-1}(z)\bbr_i\rvert^m  \right. \\
			& \quad\quad\quad \left. + \mathbb{E}\lvert \widetilde{\beta}_{ij}(z)[\bbr_i^{*}\bbD_{ij}^{-1}(z)\bbT\bbD_{ij}^{-1}(z)\bbr_i-n^{-1}\operatorname{tr} \bbT\bbD_{ij}^{-1}(z)\bbT\bbD_{ij}^{-1}(z)]\rvert^m]\right] \\
			&+K\left[\sum_{i=1}^{n} \mathbb{E}\lvert \beta_{ij}(z)\widetilde{\beta}_{ij}(z)(\bbr_i^{*}\bbD_{ij}^{-1}(z)\bbr_i-n^{-1}\operatorname{tr}\bbT\bbD_{ij}^{-1}(z))\bbr_i^{*}\bbD_{ij}^{-1}(z)\bbT\bbD_{ij}^{-1}(z)\bbr_i\rvert^2\right.\\
			&\quad\quad \left. + \sum_{i=1}^{n}\mathbb{E}\lvert \widetilde{\beta}_{ij}(z)[\bbr_i^{*}\bbD_{ij}^{-1}(z)\bbT\bbD_{ij}^{-1}(z)\bbr_i-n^{-1}\operatorname{tr} \bbT\bbD_{ij}^{-1}(z)\bbT\bbD_{ij}^{-1}(z)]\rvert^2\right]^{m/2} \\
			\leq& K.
		\end{align*}
		If $m=1$, based on the above proof process and Lemma \ref{Burkholder_1},
		\begin{align*}
			&\mathbb{E}\lvert \operatorname{tr}\bbT\bbD_j^{-1}(z)-\mathbb{E}\operatorname{tr}\bbT\bbD_j^{-1}(z)\rvert 
			=\mathbb{E}\lvert \sum_{i=1}^{n} (\mathbb{E}_i-\mathbb{E}_{i-1})\beta_{ij}(z)\bbr_i^{*}\bbD_{ij}^{-1}(z)\bbT\bbD_{ij}^{-1}(z)\bbr_i\rvert\\
			\leq & (\sum_{i=1}^{n}\mathbb{E}\lvert(\mathbb{E}_i-\mathbb{E}_{i-1})\beta_{ij}(z)\bbr_i^{*}\bbD_{ij}^{-1}(z)\bbT\bbD_{ij}^{-1}(z)\bbr_i\rvert^2)^{1/2} 
			\leq  K.
		\end{align*}Then we complete the proof of this lemma.
	\end{proof}
	
	\begin{lemma}\label{E|beta_j-b_j|}
		For all $z\in\mathcal{C}$ and any $m\geq 1$,
		\begin{align*}
			\mathbb{E}\lvert \widetilde{\beta}_j(z)-b_j(z)\rvert^m=Kn^{-m}.
		\end{align*}
	\end{lemma}
	\begin{proof}
		Due to   Lemmas \ref{bound moment of b_j and beta_j}, \ref{a(v)_quadratic_minus_trace} and \ref{E|trTDj-1-EtrTDj-1|}, we write for any $m\geq1$,
		\begin{align*}
			& \mathbb{E}\lvert \widetilde{\beta}_j(z)-b_j(z)\rvert^m 
			= \mathbb{E}\lvert \widetilde{\beta}_j(z)b_j(z)(n^{-1}\operatorname{tr}\bbT\bbD_j^{-1}(z)-n^{-1}\mathbb{E}\operatorname{tr}\bbT\bbD_j^{-1}(z))\rvert^m \\
			\leq& K n^{-m} (\mathbb{E}\lvert \operatorname{tr}\bbT\bbD_j^{-1}(z)-\mathbb{E}\operatorname{tr}\bbT\bbD_j^{-1}(z)\rvert^{2m})^{1/2}(\mathbb{E}\lvert\widetilde{\beta}_j(z)\rvert^{2m})^{1/2} \\
			\leq & Kn^{-m}(\mathbb{E}\lvert \beta_j(z)\rvert^{2m}+\mathbb{E}\lvert \beta_j(z)\widetilde{\beta}_j(z)(\bbr_j^*\bbD_j^{-1}(z)\bbr_j-n^{-1}\operatorname{tr}\bbT\bbD_j^{-1}(z))\rvert^{2m})^{1/2}  \\
			\leq & Kn^{-m},
		\end{align*}
	which completes the proof.
	\end{proof}
	
	\begin{lemma}\label{|b_j-b|}
		For all $z\in\mathcal{C}$,
		$$\lvert b_j(z)-b(z)\rvert\leq Kn^{-1}.$$
	\end{lemma}
	
	\begin{proof}
		Due to Lemma \ref{bound moment of b_j and beta_j} and \eqref{D^{-1}-D_j^{-1}}, we have
		\begin{align}\label{b_j-b}
			\lvert b_j(z)-b(z)\rvert &=\lvert b_j(z) b(z) (n^{-1}\mathbb{E}\operatorname{tr}\bbT\bbD^{-1}(z)-n^{-1}\mathbb{E}\operatorname{tr} \bbT\bbD_j^{-1}(z))\rvert  \\
			\nonumber &\leq Kn^{-1} \lvert  \mathbb{E}\operatorname{tr}(\bbD^{-1}(z)-\bbD_j^{-1}(z))\bbT\rvert \\
			\nonumber &= Kn^{-1} \lvert \mathbb{E}\beta_{j}(z) \bbr_j^{*}\bbD_j^{-1}(z)\bbT\bbD_j^{-1}(z)\bbr_j\rvert.
		\end{align}
		Due to \eqref{beta_j} and Lemma \ref{a(v)_quadratic_minus_trace}, we can get
		\begin{align*}
			\eqref{b_j-b} &\leq Kn^{-1}\lvert \mathbb{E}\widetilde{\beta}_j(z)\bbr_j^{*}\bbD_j^{-1}(z)\bbT\bbD_j^{-1}(z)\bbr_j\rvert + Kn^{-1}\lvert\mathbb{E}\beta_{j}(z)\widetilde{\beta}_j(z)\varepsilon_j(z)\bbr_j^{*}\bbD_j^{-1}(z)\bbT\bbD_j^{-1}(z)\bbr_j\rvert \\
			&\leq Kn^{-1}\mathbb{E}\lvert\widetilde{\beta}_j(z)\bbr_j^{*}\bbD_j^{-1}(z)\bbT\bbD_j^{-1}(z)\bbr_j\rvert  + Kn^{-1}\mathbb{E}\lvert\beta_{j}(z)\widetilde{\beta}_j(z)\bbr_j^{*}\bbD_j^{-1}(z)\bbT\bbD_j^{-1}(z)\bbr_j\varepsilon_j(z)\rvert \\
			&\leq Kn^{-1}+Kn^{-3/2}\leq Kn^{-1},
		\end{align*}
	which completes the proof.
	\end{proof}
	
	\begin{lemma}\label{|Es_n-s_n^0|}
		For all $z\in\mathcal{C}$, we have
		\begin{align*}
			\lvert\mathbb{E}\underline{s}_n(z)-\underline{s}_n^0(z)\rvert \leq Kn^{-1}.
		\end{align*}
	\end{lemma}
	
	\begin{proof}
		The proof of this lemma is based on arguments in Section 5 of \cite{BaiS98N}, with the extension that the result holds for all $z\in\mathcal{C}$. The uniform bound on $\lVert \bbD\rVert_1^{-1}$ provided by Lemma \ref{bound moment of b_j and beta_j}, ensures the term $\sup _{z \in\mathcal{C}} \mathbb{E}\operatorname{tr}\bbD_1^{-1} \bar{\bbD}_1^{-1} \leq p \sup_{z\in\mathcal{C}}\mathbb{E}\lVert\bbD_1^{-2}\rVert\lVert\bar{\bbD}_1^{-2}\rVert \leq K n$. Moreover, 	by similarly discussing in \eqref{|a_n(z_1,z_2)|}, we have, for all $n$ there exists an $\varrho>0$,
		\begin{align*}
			\left(\frac{\Im\underline{s}_n^0y_n\int\frac{t^2dH_p(t)}{\lvert 1+t\underline{s}_n^0\rvert^2}}{\Im z+\Im\underline{s}_n^0y_n\int\frac{t^2dH_p(t)}{\lvert 1+t\underline{s}_n^0\rvert^2}}\right)^{1/2}\leq 1-\varrho.
		\end{align*}
		The remaining part of the proof is identical to that of  Section 5 in  \cite{BaiS98N}, and is therefore omitted.
	\end{proof}

	\section{Stein identity and its solution}\label{Stein's Equation and its Solution}
	\cite{Stein72Ba} introduced a noval approach to normal approximation, establishing that a random variable $W$ follows the standard normal distribution if and only if 
	\begin{align*}
		\mathbb{E}[g^{\prime}(W)-Wg(W)]=0,
	\end{align*}		
	for all absolutely continuous $g$ with a.s. derivative $g^{\prime}$ such that $\mathbb{E}\lvert g^{\prime}(Z)\rvert<\infty$. This implies that if a random variable $W$ with $\mathbb{E}W=0$ and $\operatorname{Var}W=1$ is not follow $\mathcal{N}(0,1)$, $\mathbb{E}[g^{\prime}(W)-Wg(W)]\neq 0$. The discrepancy between $\mathcal{L}(W)$ and $\mathcal{N}(0,1)$ can be characterized via the following equation, 
	\begin{align}\label{stein_equation}
		g^{\prime}(w)-wg(w)=h(w)-\mathbb{E}h(Z)
	\end{align}
	where $h$ is a given real valued measurable function  with $E\lvert h(Z)\rvert<\infty$. We denote $Eh(Z)$ by $Nh$ and call \eqref{stein_equation} as
	the Stein equation for $h$, or simply the Stein equation. From \eqref{stein_equation},	 		\begin{align}\label{stein_equation_second_order}
		g^{\prime\prime}(w) =g(w)+wg^{\prime}(w)+h^{\prime}(w). 
   	\end{align} 
	
	
	The unique bounded solution of $\eqref{stein_equation}$ is given by
	\begin{align}\label{stein_solution}
		g_h(w) &=e^{w^2/2}\int_{-\infty}^{w}(h(x)-Nh)e^{-x^2/2}dx\\
		\nonumber&=-e^{w^2/2}\int_{w}^{+\infty}(h(x)-Nh)e^{-x^2/2}dx
	\end{align}
	We now state some properties of the solution $g_h$ to the Stein equation $\eqref{stein_equation}$. 
	\begin{lemma}[Lemma 2.4 in \cite{ChenG11N}]
		For a given function $h:\mathbb{R}\to\mathbb{R}$, let $g_h$ be the solution $\eqref{stein_solution}$ to the Stein equation $\eqref{stein_equation}$. If $h$ is bounded, then
		\begin{align*}
			\lVert g_h\rVert\leq\sqrt{\pi/2}\lVert h(\cdot)-Nh\rVert \quad and \quad \lVert g^{\prime}_h\rVert\leq2\lVert h(\cdot)-Nh\rVert.
		\end{align*}
		If $h$ is absolutely continuous, then
		\begin{align*}
			\lVert g_h\rVert\leq 2\lVert h^{\prime}\rVert,\quad \lVert g^{\prime}_h\rVert\leq\sqrt{2/\pi}\lVert h^{\prime}\rVert \quad and \quad \lVert g^{\prime\prime}_h\rVert\leq2\lVert h^{\prime}\rVert.
		\end{align*}
	\end{lemma}

	Especially, if we smooth the indicator function of $( -\infty,w_0] $, i.e. the function
	\begin{align}\label{initial_smooth_indicator_function}
		h_{w_0,\alpha}(w)=\left\{
		\begin{array}{rcl}
			&1  \quad & {w\leq w_0},\\
			&1+(w_0-w)/\alpha  \quad & {w_0<w\leq w_0+\alpha},\\
			&0  \quad & {w>w_0+\alpha},
		\end{array}\right.
	\end{align}
	where $\alpha>0$, we have the following bound, which is shown in Lemma 2.5 in \cite{ChenG11N}.  
	
	\begin{lemma}\label{properties_of_the_smoothed_stein_solution}
		For $u,v\in\mathbb{R}$ and $\alpha>0$, let $g$ be the solution \eqref{stein_solution} to the Stein equation for the smoothed indicator function \eqref{initial_smooth_indicator_function}. Then, for all $u,v\in\mathbb{R}$,
		\begin{align*}
			0\leq g(u)\leq 1,\quad \lvert g^{\prime}(u)\rvert\leq 1, \quad\lvert g^{\prime}(u)-g^{\prime}(v)\rvert\leq 1,
		\end{align*} 
		and 
		\begin{align*}
			\lvert g^{\prime}(u+v)-g^{\prime}(u)\rvert \leq \rvert v\rvert\left(1+\lvert u\rvert+\frac{1}{\alpha}\int_{0}^{1}I_{\left[z,z+\alpha\right]} (u+rv)dr\right).
		\end{align*} 
	\end{lemma}
	
	\end{appendix}
	
%

	
\end{document}